\documentclass[a4paper,12pt]{article}
\usepackage{amsmath,amsthm,amsfonts,amssymb,bbm}
\usepackage{graphicx,psfrag,subfigure,color,cite}
\usepackage[UKenglish]{babel}
\usepackage{textcmds}
\usepackage{mathtools}
\usepackage{float}
\usepackage{changes}
\usepackage[colorlinks]{hyperref}

\numberwithin{equation}{section}

\usepackage[margin=1.2in]{geometry}

\newcommand{\Cov}{{\rm Cov}}

\newcommand{\e}{\varepsilon}
\newcommand{\Pb}{\mathbb{P}}
\newcommand{\E}{\mathbb{E}}

\newcommand{\R}{\mathbb{R}}
\newcommand{\N}{\mathbb{N}}
\newcommand{\Z}{\mathbb{Z}}

\newcommand{\Id}{\mathbbm{1}}

\newtheorem{prop}{Proposition}[section]

\newtheorem{thm}[prop]{Theorem}
\newtheorem{lem}[prop]{Lemma}

\newtheorem{cor}[prop]{Corollary}

\newtheorem{cla}[prop]{Claim}
\newtheorem{assumpt}[prop]{Assumption}

\newtheorem{rem}[prop]{Remark}
\newenvironment{remark}{\begin{rem}\normalfont}{\end{rem}}

\title{Decoupling and decay of two-point \\ functions in a two-species (T)ASEP}
\author{Patrik L.\ Ferrari\thanks{Institute for Applied Mathematics, Bonn University, Endenicher Allee 60, 53115 Bonn, Germany. Email: {\tt ferrari@uni-bonn.de}} \and
Sabrina Gernholt\thanks{Institute for Applied Mathematics, Bonn University, Endenicher Allee 60, 53115 Bonn, Germany. Email: {\tt sgernhol@uni-bonn.de}}
}
\date{}

\begin{document}

\maketitle

\begin{abstract}
We consider the two-species totally asymmetric simple exclusion process on $\Z$ with a translation-invariant stationary measure as the initial condition. We establish the asymptotic decoupling of the marginal height profiles along characteristic lines and prove the decay of the two-point functions in the large-time limit, thus confirming predictions of the nonlinear fluctuating hydrodynamics theory. Our approach builds on the queueing construction of the stationary measure introduced in~\cite{Ang06,FM07} and extends the theory of backwards paths for height functions developed in~\cite{BF22,FN24}. The arguments for asymptotic decoupling also apply to further homogeneous initial data, and the decay of the two-point functions is proven for the stationary two-species asymmetric simple exclusion process, beyond the totally asymmetric case.
\end{abstract}

\section{Introduction and main results} \label{section_introduction}

In this work, we study the large-time behaviour of two key observables in the two-species (totally) asymmetric simple exclusion process (TASEP or ASEP, respectively). First, for TASEP, we demonstrate the asymptotic decoupling of marginal height profiles for a class of spatially homogeneous initial conditions, including the stationary case. Second, we establish the decay of mixed space-time correlations, also known as two-point functions, for the stationary two-species ASEP.
Our results align with numerous insights for a broader class of multi-component models, which serve as the key motivation for this study. We outline this context in Section~\ref{section_motivation} and present our exact results for TASEP in Section~\ref{section_model_main_results}. Section~\ref{section_main_results_ASEP} extends the decay of mixed space-time correlations to ASEP.


\subsection{Motivation: Nonlinear fluctuating hydrodynamics and KPZ universality} \label{section_motivation}

The primary motivation for our work is the relation between the fluctuating hydrodynamics of multi-component lattice gases and coupled KPZ equations, see for instance \cite{FSS13}. Their approach can be summarised as follows.

One considers the fluctuations of the density fields in the lattice gas model around their stationary value and transforms the conservation laws into stochastic PDEs by expanding the current up to second order while incorporating noise and diffusion terms. By transitioning to normal modes, one generically obtains a system of coupled stochastic Burgers equations, whose integrated form corresponds to coupled KPZ equations. For distinct velocity parameters, the equations (approximately) decouple, indicating a diminishing correlation between distinct components. At the same time, correlations within each component are then expected to converge to the one-dimensional KPZ-universal limit.

This method is not restricted to the setting of~\cite{FSS13}, but provides a general framework for determining universal scaling exponents and limit functions for a variety of (multi-component) models. It has led to predictions comparable to those in~\cite{FSS13} in numerous works\footnote{In particular, one can identify other universality classes with distinct equations and scaling exponents beyond KPZ, see for instance~\cite{PSS15,PSSS15,SS15}.}, which are collectively referred to as \emph{nonlinear fluctuating hydrodynamics theory} (NLFH), see~\cite{vanB12,MS13,PSSS15,SW17,Spo13,Spo16,SS15} and the references therein.

For one-component growth models in the KPZ universality class, their large-time behaviour follows the statistics of the KPZ equation.
For specific models in $1+1$ dimensions, such as ASEP or TASEP, there exists a robust understanding of the emergence of universal objects. In particular, the scaling limit of the correlations in a stationary single-species TASEP has been determined exactly~\cite{BFP12,FS05a,PS01} and recently proven in ASEP~\cite{ACH24,LS25}. However, for multi-component models, while many numerical simulations exist (for instance in~\cite{RDKKS24,DDSMS14,KHS15,MS13,MS14,PSS15,SS15}), exact confirmations of the predictions from NLFH remain rare, especially for KPZ-like behaviour.

For the stationary AHR model,~\cite{FSS13} support their predictions by simulations without mathematical proofs. In~\cite{CGHS18,CGHSU22}, special initial conditions are considered and a decoupling of height functions is shown through an asymptotic analysis of an integral formula for the Green's function. Rigorous mathematical proofs yielding systems of coupled KPZ equations can be found in~\cite{CGMO24,CCGO26} for the ABC model and in~\cite{BFS21} for the multi-species weakly-asymmetric zero-range process, with no decoupling.

Below, we outline the interpretation of NLFH in~\cite{FSS13} and present its predictions for the special case of the stationary two-species TASEP, the model we rigorously analyse in the remainder of this work. The predictions for ASEP can be derived following the same procedure.

\paragraph{Multi-component lattice gas models.} Consider a lattice gas with $n$ components, where each component corresponds to a distinct particle type. Particles perform nearest-neighbour jumps on the integers $\Z$, with at most one particle per site. The jump rates are local, translation-invariant, and depend only on the particle type. We label the types by $\alpha\in\{1,\ldots,n\}$ and define
\begin{equation}
\eta_\alpha(j,t)=\left\{
\begin{array}{ll}
1 &\textrm{ if there is a particle of type }\alpha\textrm{ at site }j \textrm{ at time }t,\\
0& \textrm{ otherwise}.
\end{array}
\right.
\end{equation}
We assume that for any fixed $\vec{\rho}=(\rho_1,\ldots,\rho_n)$ (with $\sum_{i=1}^n \rho_i\leq 1$ and $\rho_i \in [0,1]$), there is an ergodic, translation-invariant stationary measure $\mu_{\vec{\rho}}$ such that $\rho_\alpha$ is the average density of the particles of type $\alpha$ under $\mu_{\vec{\rho}}$. This means that we have
\begin{equation}
\E_{\mu_{\vec{\rho}}}[\eta_\alpha(j,t)]=\rho_\alpha,\quad \alpha=1,\ldots,n,
\end{equation}
and
\begin{equation}
\lim_{L\to\infty}\frac{1}{2L}\sum_{j=-L+1}^L\eta_{\alpha}(j,t)=\rho_\alpha
\end{equation}
almost surely. Here, $\E_{\mu_{\vec{\rho}}}$ is the expectation with respect to the initial distribution $\mu_{\vec{\rho}}$.
We denote by ${\cal J}_\alpha(j,t)$ the instantaneous current of type-$\alpha$-particles from site $j$ to site $j+1$ at time $t$ and define
\begin{equation}
{\rm j}_\alpha(\vec\rho)=\E_{\mu_{\vec{\rho}}}[{\cal J}_\alpha(j,t)].
\end{equation}

The observables of interest are the time correlations
\begin{equation}
S_{\alpha,\beta}(j,t)=\E_{\mu_{\vec{\rho}}}[\eta_\alpha(j,t)\eta_\beta(0,0)]-\rho_\alpha\rho_\beta,\quad \alpha,\beta\in\{1,\ldots,n\}.
\end{equation}
They form an $n\times n$ matrix $S = (S_{\alpha,\beta})_{\alpha,\beta \in \{1, \dots, n\}}$ known as the two-point function. The function $S_{\alpha,\beta}(j,t)$ quantifies the impact of a perturbation in the density of type-$\beta$-particles at the origin at time $0$ on the average density of type-$\alpha$-particles at site $j$ at time $t$.
It fulfils
\begin{equation}
\sum_{j\in\Z} S(j,t)=\sum_{j\in\Z} S(j,0)= C,
\end{equation}
with $C$ being called the susceptibility matrix. Note that $C$ is symmetric and non-negative, and we assume $C>0$ (that is, $C$ is positive definite) to avoid having components that do not evolve over time. Moreover, defining the matrix $A$ with components
\begin{equation}
A_{\alpha,\beta}(\vec\rho)=\frac{\partial}{\partial \rho_\beta}{\rm j}_\alpha(\vec\rho),
\end{equation}
the following identity holds:
\begin{equation}\label{eq1}
A C = C A^{\rm T}.
\end{equation}
For multi-component particle systems, this has been proven in~\cite{TV03} for product stationary measures and in~\cite{GS11} in general. See also Appendix~A of~\cite{FSS13}. As discussed in that work, the variables $\vec\eta=(\eta_1,\ldots,\eta_n)$ are not the suitable coordinates for studying the system at the hydrodynamic level. Instead, one should consider the normal modes, which are obtained as follows. By $C=C^{\rm T}>0$ and \eqref{eq1}, the matrix $A$ has real eigenvalues $v_1, \dots, v_n$ and a non-degenerate system of left and right eigenvectors. The normal mode coordinates are defined as
\begin{equation}
\vec\xi=R \vec\eta,
\end{equation}
where $R$ is a matrix satisfying the relations
\begin{equation}\label{eq1.11}
R A R^{-1}={\rm diag}(v_1,\ldots,v_n)\quad\textrm{and}\quad R C R^{\rm T}=\Id.
\end{equation}
The two-point function for the normal modes is then given by
\begin{equation}
\begin{aligned}
S^{\#}_{\alpha,\beta}(j,t)=(R S R^{\rm T})_{\alpha,\beta}(j,t)&=\E_{\mu_{\vec{\rho}}}\left[(R\vec\eta)_\alpha(j,t)(R\vec\eta)_\beta(0,0)\right]- (R\vec\rho)_\alpha (R\vec\rho)_\beta\\
&=\E_{\mu_{\vec{\rho}}}\left[\xi_\alpha(j,t)\xi_\beta(0,0)\right]- (R\vec\rho)_\alpha (R\vec\rho)_\beta.
\end{aligned}
\end{equation}

We can now state the predictions from the nonlinear fluctuating hydrodynamics theory as in~\cite{FSS13}. Starting from macroscopic versions of the conservation laws
\begin{equation}
\partial_t \eta_\alpha(j,t)-{\cal J}_\alpha(j-1,t)+{\cal J}_\alpha (j,t) = 0,
\end{equation}
one derives a system of coupled stochastic Burgers equations for the density fields in normal modes. Under the transformation $R$, the drift terms are diagonalised: instead of involving the full matrix $A$ acting on the density vector, each drift reduces to the product of an eigenvalue $v_\alpha$ and the respective density component. In view of this, under the assumption that the drift velocities $v_1, \dots, v_n$ are all distinct,~\cite{FSS13} predict a decoupling of the dynamics.

For the two-point function $S^\#$, they argue that its scaled version should converge to a diagonal matrix with one entry given by the KPZ-universal limit observed in the one-component case. Moreover, the model-dependent parameter $\lambda_\alpha$ in the scaling is determined by the average current $\vec{\rm j}(\vec \rho)$ and the susceptibility matrix $C$.
More precisely, if we focus on the region around the speed of the normal mode $\alpha$, we have
\begin{equation}\label{eq2}
(\lambda_\alpha t)^{2/3} S^{\#}_{\beta,\gamma}(v_\alpha t+ w(\lambda_\alpha t)^{2/3},t)\simeq \delta_{\beta,\alpha}\delta_{\gamma,\alpha} f_{\rm KPZ}(w),
\end{equation}
where $f_{\rm KPZ}$ is the KPZ scaling function for one-component systems~\cite{PS02b}, see also \eqref{eq_KPZ_scaling_limit_correlation}, and $\lambda_{\alpha}$ is a non-universal coefficient given as follows. Denote the Hessian of ${\rm j}_\alpha$ by
\begin{equation}
H^{\alpha}_{\beta,\gamma}(\vec\rho)=\frac{\partial^2}{\partial \rho_\beta \partial \rho_\gamma}{\rm j}_\alpha(\vec\rho),
\end{equation}
and define the matrix
\begin{equation}
G^{\alpha}=\frac12 \sum_{\tilde\alpha=1}^nR_{\alpha,\tilde\alpha} (R^{-1})^{\rm T} H^{\tilde \alpha} R^{-1}.
\end{equation}
Then, the scaling coefficient is given by $\lambda_\alpha=2\sqrt{2}|G^{\alpha}_{\alpha,\alpha}|$. \\

As mentioned earlier,~\cite{FSS13} considered the so-called AHR model, which consists of two types of particles, denoted $+$ and $-$. The $+$ particles move to the right with a jump rate $\beta$, while the $-$ particles move to the left with a jump rate $\alpha$. The exchange of $(+,-)$ and $(-,+)$ occurs with rate $1$, while the exchange of $(-,+)$ and $(+,-)$ happens with rate $q<1$. For this model on a ring, the stationary measure is known~\cite{RSS00}. In~\cite{FSS13}, numerical simulations confirmed \eqref{eq2}. The totally asymmetric version of the AHR model (that is, $q=0$), but with a non-stationary initial condition, was further explored in~\cite{CGHS18,CGHSU22}. Using exact formulas, they showed a decoupling of the height functions associated with the normal modes.

Another model worth mentioning is the ABC model: each site is occupied by one particle of type $A$, $B$, or $C$. One can think of particles of type $C$ as holes, which brings the system back to the $n=2$ case. In~\cite{CGMO24} the authors consider the ABC model on a torus with weakly asymmetric jump rates (scaling with the observation time $t$), similar to the weak asymmetry scaling under which the asymmetric simple exclusion process approximates the KPZ equation. They showed rigorously that the normal mode fields converge to a system of decoupled stochastic PDEs. Imposing a strict hierarchy on the particles, each of the equations essentially corresponds to a one-dimensional KPZ equation. Refer to Section~1.4 of~\cite{CGMO24} for the specific jump rates they consider.

\paragraph{Prediction for the model analysed in this work.}
The system we consider can be viewed as a special case of the ABC model on the integers $\Z$, with strong asymmetry. Specifically, we have nearest-neighbour jumps to the right and prioritise the particle types in the order $A>B>C$. These properties characterise our model as a two-species TASEP, where particles of type $A$ correspond to first class particles, particles of type $B$ correspond to second class particles, and particles of type $C$ represent holes. Following~\cite{FSS13}, we denote the configurations of first and second class particles by $\eta_1$ and $\eta_2$ in this motivational part.

For our exact analysis, a detailed introduction of the model is provided in Section~\ref{section_model_main_results}. In particular, the translation-invariant stationary measure is known and has a queueing representation, see Section~\ref{section_main_results_stationary} and Section~\ref{section_queueing_representation}. For average densities $\vec \rho = (\rho_1,\rho_2)$, we have (see Appendix~\ref{Appendix_GHD})
\begin{equation} \label{eq.C}
C=\left(\begin{array}{cc}
  \rho_1(1-\rho_1) & -\rho_1(1-\rho_1) \\
  -\rho_1(1-\rho_1) & \rho_2(1-\rho_2)+2\rho_1(1-\rho_1)-2\rho_1\rho_2
  \end{array}\right).
\end{equation}
We also determine the average current $\vec{\rm j}(\vec \rho)= (\rho_1(1-\rho_1),\rho_2(1-\rho_2)-2\rho_1\rho_2)$,
and its Jacobian is
\begin{equation}
A=\left(
\begin{array}{cc}
  1-2\rho_1 & 0\\
  -2\rho_2 & 1-2(\rho_1+\rho_2)
\end{array}
\right).
\end{equation}

To reformulate the prediction~\eqref{eq2} in our setting, we transition to normal modes and compute the model-dependent coefficients.
The two conditions in \eqref{eq1.11} are satisfied with
\begin{equation}R=\left(
\begin{array}{cc}
\frac{1}{\sqrt{\rho_1(1-\rho_1)}} & 0 \\
\frac{1}{\sqrt{(\rho_1+\rho_2)(1-\rho_1-\rho_2)}} & \frac{1}{\sqrt{(\rho_1+\rho_2)(1-\rho_1-\rho_2)}} \\
\end{array}
\right)
\end{equation}
and $\vec{v}=(1-2\rho_1,1-2(\rho_1+\rho_2))$.
This leads to
\begin{equation}
G^1=\left(
\begin{array}{cc}
 - \sqrt{\rho_1(1-\rho_1)} & 0\\
  0 & 0
\end{array}
\right),\quad G^2=\left(
\begin{array}{cc}
  0 & 0\\
  0 & -\sqrt{(\rho_1+\rho_2)(1-\rho_1-\rho_2)}
\end{array}
\right).
\end{equation}
Thus, the normal mode variables are
\begin{equation}\label{eq1.22}
\xi_1(j,t)=\frac{\eta_1(j,t)}{\sqrt{\rho_1(1-\rho_1)}},\quad \xi_2(j,t)=\frac{\eta_1(j,t)+\eta_2(j,t)}{\sqrt{(\rho_1+\rho_2)(1-\rho_1-\rho_2)}},
\end{equation}
and the two components propagate with speeds $v_1=1-2\rho_1$ and $v_2=1-2(\rho_1+\rho_2)$, respectively.
The fact that these are the normal mode variables is also evident from the definition of the model: up to scaling, they correspond to the configurations of first class particles and all particles. These marginal processes both exhibit single-species TASEP dynamics.

Finally, the prediction \eqref{eq2} states that, for
\begin{equation}
\lambda_1=2\sqrt{2}\sqrt{\rho_1(1-\rho_1)},\quad
\lambda_2=2\sqrt{2}\sqrt{(\rho_1+\rho_2)(1-\rho_1-\rho_2)},
\end{equation}
we have
\begin{equation}\label{eq3}
\lim_{t\to\infty}(\lambda_\alpha t)^{2/3} S^{\#}_{\alpha,\alpha}(v_\alpha t+ w(\lambda_\alpha t)^{2/3},t) =  f_{\rm KPZ}(w),\quad \alpha=1,2,
\end{equation}
and the other entries of $S^{\#}$ converge to zero on the same scale. The convergence of the diagonal entries to $f_{\rm KPZ}$ has been proven in the weak sense in~\cite{BFP12} by building on the arguments developed in~\cite{FS05a,PS01,PS02b}. See \eqref{eq_KPZ_scaling_limit_correlation} for the precise statement, expressed in a slightly different notation. The coupling between the two normal modes is in the off-diagonal entries of the two-point function. In Theorem~\ref{thm_decay_corr} and Corollary~\ref{corollary_decay_corr}, we establish the convergence to zero for these off-diagonal functions. The corresponding results for ASEP can be found in Section~\ref{section_main_results_ASEP}.

For simplicity, in Section~\ref{section_model_main_results} and the rest of the paper we will not normalise the linear combination of the normal modes as in \eqref{eq1.22}, but just keep $\eta_1$ and $\eta_1+\eta_2$ respectively. This implies that in the limit of the diagonal terms there will be an additional factor $\chi_1=\rho_1(1-\rho_1)$ and $\chi_2=(\rho_1+\rho_2)(1-\rho_1-\rho_2)$, consistent with \eqref{eq_KPZ_scaling_limit_correlation} and previous papers.\\

This concludes our motivation, which sets the context for our work in hydrodynamic theory. In the following section, we provide background information on the two-species TASEP, present our main results, and compare our methods to those recently developed in~\cite{ACH24} for the coloured ASEP. Subsequently, we extend some of our arguments to ASEP as well.

\subsection{Model formulation and exact results for the two-species TASEP} \label{section_model_main_results}

\paragraph{The two-species TASEP.} The totally asymmetric simple exclusion process (TASEP) was first introduced by Spitzer in~\cite{Spi70} and is one of the most studied interacting particle systems in one spatial dimension. We consider a two-species (or two-coloured), continuous-time TASEP on the integer lattice $\Z$. This model consists of two types of particles positioned on $\Z$, where each site either hosts a single particle or is empty. In the second case, we say the site contains a hole. Particles attempt to jump one step to the right after independent, exponential waiting times with rate one. The two types of particles, referred to as first and second class, are prioritised differently. A jump attempt by a first class particle is successful if the site to its right contains either a hole or a second class particle; in such cases, the particle swaps positions with the occupant of the neighbouring site. Otherwise, it remains in its current position. A second class particle only jumps to the right if the neighbouring site is empty, and stays put otherwise.

Below, we rephrase TASEP as a random growth model. As such, it serves as a prototypical representative of the KPZ universality class~\cite{KPZ86} in $1+1$ dimensions. For both multi-species and single-species versions of the model, the scaling limits have been determined in terms of KPZ fixed points and the directed landscape~\cite{ACH24,DOV22,DV21,MQR17}.

Our results pertain to the two-species TASEP under different initial conditions. The main focus of this work is on the stationary model, which is explored in Section~\ref{section_main_results_stationary}. Other translation-invariant deterministic or random initial configurations are discussed in  Section~\ref{section_main_results_deterministic}.

\subsubsection{Asymptotic decoupling in the stationary two-species TASEP} \label{section_main_results_stationary}

\paragraph{Model and notation.} Consider the stationary two-species TASEP with a translation-invariant stationary measure as initial condition, where the first and the second class particles have constant densities $\rho_1 \in (0,1)$ and $\rho_2 \in (0,1-\rho_1)$. We denote this measure by $\mu^{\rho_1,\rho_2}$. For its existence, uniqueness, and other properties, we refer to~\cite{Ang06,FKS91,FM07,Lig76,Spe94}, along with additional references listed below.
A sample of $\mu^{\rho_1,\rho_2}$ is a configuration $\eta : \Z \to \{1,2,+\infty\}$ with
\begin{equation} \begin{aligned}
 \eta(i) = \begin{cases}
1 &\text{if there is a first class particle at site } i, \\ 2  &\text{if there is a second class particle at site } i, \\ +\infty &\text{if there is a hole at site } i. \end{cases} \end{aligned} \end{equation}
With the same interpretation, $\eta_t$ denotes the configuration of the stationary two-species TASEP started from $\eta_0 = \eta$ at the time $t \geq 0$. We define the configurations of first class, respectively of all particles, by
\begin{equation} \label{eq_def_marginal_configurations}
\eta^{\rho_1}_t(i) = \mathbbm{1}_{\{\eta_t(i)=1\}}, \quad \eta_t^{\rho_1+\rho_2}(i) = \mathbbm{1}_{\{\eta_t(i) \in \{1,2\}\}}.
\end{equation}
Now, holes are represented by $0$ instead of $+\infty$. The marginal processes $\eta_t^{\rho_1}$ and $\eta_t^{\rho_1+\rho_2}$ are two single-species TASEPs under basic coupling, starting from the initial configurations $\eta^{\rho_1}(i) = \mathbbm{1}_{\{\eta(i)=1\}}$ and $\eta^{\rho_1+\rho_2}(i) = \mathbbm{1}_{\{\eta(i) \in \{1,2\}\}}$. They are stationary again, since the corresponding marginals of $\mu^{\rho_1,\rho_2}$ are the Bernoulli product measures with densities $\rho_1$ and $\rho_1+\rho_2$.

For $\rho \in \{\rho_1,\rho_1+\rho_2\}$, we consider the height function
\begin{equation} \label{eq_def_height_fct}
\begin{aligned}
h^{\rho}(j,t) = \begin{cases} 2N_t^\rho + \sum_{i=1}^j (1-2\eta_t^\rho(i)),  & j \geq 1, \\
2N_t^\rho,  & j=0, \\ 2N_t^\rho - \sum_{i=j+1}^0(1-2\eta_t^\rho(i)),  & j \leq -1,
\end{cases}
\end{aligned}
\end{equation}
where $N_t^\rho$ denotes the respective number of particles that jumped from site $0$ to site $1$ until time $t$. Each time a particle jumps, a local minimum of the height function turns into a local maximum. We rescale
\begin{equation} \label{eq_rescaled_height_fct}
	\mathfrak{h}^\rho(w,t) = \frac{h^\rho((1-2\rho)t+2w \chi^{1/3} t^{2/3},t) - (1-2\chi)t - 2w(1-2\rho)\chi^{1/3}t^{2/3}}{-2\chi^{2/3}t^{1/3}}
\end{equation}
for $w \in \R$ and $\chi = \rho(1-\rho)$.
It is known by~\cite{FS05a} that
\begin{equation} \label{eq_conv_to_baik_rains_distr}
\lim_{t \to \infty} \Pb ( \mathfrak{h}^{\rho}(w,t) \leq s) = F_{\text{BR},w}(s),
\end{equation}
where $F_{\text{BR},w}$ denotes the Baik-Rains distribution function with parameter $w$ (see~\cite{BFP09} for the joint distribution in $w$).

Furthermore, we study the two-point function of the process $\eta_t$.
We define the correlation matrix $S^\#$ of the densities by
\begin{equation} \label{eq_correlation_matrix}
\begin{aligned}
&\begin{pmatrix}
S^\#_{1,1}(j,t) & S^\#_{1,2}(j,t)\\S^\#_{2,1}(j,t) &  S^\#_{2,2}(j,t)
\end{pmatrix}
\\ &=
\begin{pmatrix}
\langle \eta_t^{\rho_1}(j) \eta_0^{\rho_1}(0)\rangle -\rho_1^2 & \langle \eta_t^{\rho_1}(j)\eta_0^{\rho_1+\rho_2}(0)\rangle - \rho_1(\rho_1+\rho_2) \\
\langle \eta_t^{\rho_1+\rho_2}(j)\eta_0^{\rho_1}(0)\rangle - \rho_1(\rho_1+\rho_2) & \langle\eta_t^{\rho_1+\rho_2}(j) \eta_0^{\rho_1+\rho_2}(0)\rangle - (\rho_1+\rho_2)^2
\end{pmatrix}
\end{aligned}
\end{equation}
for $j \in \Z$ and $t \geq 0$, where $\langle \cdot \rangle$ denotes the expectation with respect to $\mu^{\rho_1,\rho_2}$. The diagonal terms $S^\#_{1,1}(j,t)$ and $S^\#_{2,2}(j,t)$ correspond to the correlations in stationary single-species TASEPs and, under suitable scaling, converge to a KPZ-universal limit, namely the second derivative of the second moment of the Baik-Rains distribution with varying parameters.
For $i \in \{1,2\}$ and $\rho=\rho(i) \in \{\rho_1,\rho_1+\rho_2\}$, we have
\begin{equation} \label{eq_KPZ_scaling_limit_correlation}
\lim_{t \to \infty} 2 \chi^{1/3} t^{2/3} S^\#_{i,i}((1-2\rho)t+2w \chi^{1/3} t^{2/3},t) = \frac{\chi}{4} g^{''}_{\text{sc}}(w)=\chi f_{\rm KPZ}(w)
\end{equation}
when integrated against smooth functions of $w$ with compact support, where $g_{\text{sc}}(w) = \int_\R s^2 d F_{\text{BR},w}(s)$.
This has been established in Corollary~2 of~\cite{BFP12} and was previously conjectured in~\cite{FS05a}. It is compatible with \eqref{eq3} since in \eqref{eq1.22} the normal modes are normalised by $\chi^{-1/2}$. For speeds $v \neq 1-2\rho$, the rescaled correlations converge to zero. This follows from the fact that $\chi^{-1} S(j,t)$ equals the probability that a second class particle, initially at the origin, is at site $j$ at time $t$, see~\cite{PS01}, combined with the property that $\int_\R f_{\rm KPZ}(w) dw = 1$.

\paragraph{Main results and key techniques.} The primary goal of this work is to derive the asymptotic decoupling of the height profiles as well as the decay of mixed correlations for the processes $\eta^{\rho_1}_t$ and $\eta^{\rho_1+\rho_2}_t$. Our first result confirms that as marginals of the two-species TASEP, due to their distinct particle densities, $h^{\rho_1}$ and $h^{\rho_1+\rho_2}$ decouple in the large-time limit along characteristic lines.

\begin{thm} \label{thm_asymptotic_decoupling_height_functions_stationary}
	The height fluctuations of $\mathfrak{h}^{\rho_1}(w,t)$ and $\mathfrak{h}^{\rho_1+\rho_2}(w,t)$ are asymptotically independent: for any $w,z,r,s \in \R$, it holds
	\begin{equation} \label{eq_thm_asymptotic_decoupling_height_functions_stationary} \begin{aligned}
	 \lim_{t \to  \infty} \Pb ( \mathfrak{h}^{\rho_1}(w,t) \leq s, \mathfrak{h}^{\rho_1+\rho_2}(z,t) \leq r )  & = \lim_{t \to  \infty} \Pb ( \mathfrak{h}^{\rho_1}(w,t) \leq s) \Pb( \mathfrak{h}^{\rho_1+\rho_2}(z,t) \leq r ) \\ &= F_{\textup{BR},w}(s)F_{\textup{BR},z}(r).
	\end{aligned}
	\end{equation}
\end{thm}

The proof of Theorem~\ref{thm_asymptotic_decoupling_height_functions_stationary} in Section~\ref{section_asymptotic_decoupling_stationary} relies on two main ingredients: (a) a local independence of configurations sampled from $\mu^{\rho_1,\rho_2}$, as stated in Lemma~\ref{lemma_stationary_measure_independence_second_class_p}, and (b) the localisation of backwards paths in both a stationary (single-species) TASEP and a TASEP with step initial condition\footnote{In the step initial condition, all sites in $\Z_{\leq 0}$ are occupied by particles, while all sites in $\N$ are empty.}.

For TASEP height functions, the notion of backwards paths has been established in~\cite{BF22,FN24}. For the step initial condition, their localisation is given in Proposition~4.9 of~\cite{BF22}. In Proposition~\ref{prop_localisation_backwards_geodesics_general}, we show that by a comparison to paths in a TASEP with step initial condition, we can localise the backwards paths in a single-species TASEP with any initial data, provided that their endpoints are controlled. This gives the localisation of backwards paths in the stationary case in Corollary~\ref{cor_localisation_backwards_geodesics_stationary_TASEP}. \\

Our second observation is that, under KPZ-scaling, the off-diagonal terms of $S^\#$ vanish in the large-time limit. Together with \eqref{eq_KPZ_scaling_limit_correlation}, it rigorously confirms the prediction \eqref{eq3}.
The starting point of the analysis is the following new expression for the sum of the off-diagonals, which holds true not only for TASEP but also for ASEP.
\begin{prop} \label{prop_mixed_correlations_formula} For any $t, \tilde{t} \geq 0$ and $i,x,\tilde{x} \in \Z$, it holds
\begin{equation} \label{eq_prop_mixed_correlations_formula}
S^\#_{1,2}(x+i,t) + S^\#_{2,1}(\tilde{x}+i,\tilde{t})
 = \frac{1}{4} \Delta \textup{Cov}(h^{\rho_1}(x+i,t), h^{\rho_1+\rho_2}(\tilde{x}+i,\tilde{t}))
 \end{equation}
where $\Delta$ is the discrete Laplace operator given by $(\Delta f)(i) = f(i+1)-2f(i)+f(i-1)$.
\end{prop}
Proposition~\ref{prop_mixed_correlations_formula} generalises the formula for the variance, see Proposition 4.1 of~\cite{PS01}, and is proven in Section~\ref{section_formula_mixed_correlations}.

Using \eqref{eq_prop_mixed_correlations_formula} along with several specific properties of the stationary measure $\mu^{\rho_1,\rho_2}$ from Section~\ref{section_stationary_measure} and Section~\ref{section_mixed_correlations}, we show that
\begin{equation}
t^{2/3} S^\#_{1,2}((1-2\rho_1)t+w t^{2/3},t) \text{ and } t^{2/3} S^\#_{2,1}((1-2(\rho_1+\rho_2))t+w t^{2/3},t)
\end{equation}
weakly converge to zero. As before, this means they converge to zero when integrated against smooth functions of $w$ with compact support. The same holds for general speeds $v$ not equal to $1-2\rho_1$ or $1-2(\rho_1+\rho_2)$, respectively; see Remark~\ref{remark_general_speeds}.

\begin{thm} \label{thm_decay_corr}
Let $\phi:\R \to \R$ be a smooth function with compact support. Then, it holds
\begin{equation}
\lim_{t \to \infty} t^{-2/3} \sum_{w \in t^{-2/3} \Z} \phi(w) t^{2/3} S^\#_{2,1}((1-2(\rho_1+\rho_2))t+w t^{2/3},t) = 0.
\end{equation}
\end{thm}

Since Theorem~\ref{thm_decay_corr} holds for any choice of $\rho_1 \in (0,1)$ and $\rho_2 \in (0,1-\rho_1)$, the particle-hole duality yields the corresponding statement for $S^\#_{1,2}$:

\begin{cor} \label{corollary_decay_corr}
Let $\phi:\R \to \R$ be a smooth function with compact support. Then, it holds
\begin{equation}
\lim_{t \to \infty} t^{-2/3} \sum_{w \in t^{-2/3} \Z} \phi(w) t^{2/3} S^\#_{1,2}((1-2\rho_1)t+w t^{2/3},t) = 0.
\end{equation}
\end{cor}

Theorem~\ref{thm_decay_corr} and Corollary~\ref{corollary_decay_corr} are proven in Section~\ref{section_decay_correlations}.

Let us summarise the above statements for the $2\times 2$ matrix. For a given $v$, define
\begin{equation}
{\cal S}_v^\#(\phi):=\lim_{t \to \infty} t^{-2/3} \sum_{w \in t^{-2/3} \Z} \phi(w) t^{2/3} S^\#( vt+w  t^{2/3},t)
\end{equation}
with $\phi:\R \to \R$ being a smooth function with compact support. It holds: \\
(a) if $v=1-2\rho_1$, then
\begin{equation}
{\cal S}_v^\#(\phi)=\left(
\begin{array}{cc}
\chi_1 \int_\R \phi(w) \lambda_1^{-2/3}f_{\rm KPZ}( \lambda_1^{-2/3}w)dw & 0 \\
0 & 0 \\
\end{array}
\right)
\end{equation}
with $\chi_1=\rho_1(1-\rho_1)$ and $\lambda_1=2\sqrt{2\chi_1}$,\\
(b) if $v=1-2(\rho_1+\rho_2)$, then
\begin{equation}
{\cal S}_v^\#(\phi)=\left(
\begin{array}{cc}
0 & 0 \\
0 & \chi_2 \int_\R \phi(w) \lambda_2^{-2/3}f_{\rm KPZ}( \lambda_2^{-2/3}w)dw \\
\end{array}
\right)
\end{equation}
with $\chi_2=(\rho_1+\rho_2)(1-\rho_1-\rho_2)$ and $\lambda_2=2\sqrt{2\chi_2}$,\\
(c) for all other values of $v$,
\begin{equation}
{\cal S}_v^\#(\phi)=\left(
\begin{array}{cc}
0 & 0 \\
0 & 0 \\
\end{array}
\right).
\end{equation}

\paragraph{Related work on stationary exclusion processes.}
Initially, the translation-invariant stationary measure $\mu^{\rho_1,\rho_2}$ of the two-species TASEP was studied due to its significance in understanding the microscopic structure of shocks in marginal single-species processes. Its existence was established using standard coupling techniques~\cite{Lig76}, and its uniqueness and extremality were demonstrated in~\cite{FKS91,Spe94}. An explicit construction of the measure via the matrix product ansatz was provided in~\cite{DJLS93}; see also~\cite{Spe94}. Probabilistic interpretations derived from this solution were discussed in~\cite{FFK94}.
Subsequently,~\cite{Ang06} developed a combinatorial construction of the measure, resulting in a queueing representation detailed in~\cite{FM07}. Our work on $\mu^{\rho_1,\rho_2}$ builds on this construction, as described in Section~\ref{section_queueing_representation}.
In~\cite{FM07}, it was extended to the $n$-species TASEP, with additional context provided in~\cite{FM06}. Complementing the queueing representation, the matrix product solution from~\cite{DJLS93} was generalised to the $n$-species TASEP in~\cite{EFM09}.

In a broader framework,~\cite{AAV11} introduced the TASEP speed process, which projects onto the translation-invariant stationary measures of each multi-species TASEP. In~\cite{BSS23}, its convergence to the stationary horizon, originating from~\cite{Bus24}, was established, using the queueing construction from~\cite{FM07}.

A parallel theory has been developed for the ASEP. The stationary distributions of the multi-species ASEP were constructed in~\cite{Mar20} using queues, while the matrix product ansatz was applied in~\cite{EMP09}. Already~\cite{AAV11} conjectured the existence of an ASEP speed process, and demonstrated some of its properties. Its existence was later confirmed in~\cite{ACG23}.

In~\cite{ACH24}, the convergence of height functions of the multi-species ASEP and the coloured stochastic six-vertex model to the Airy sheet was established. Among their corollaries are the convergence to the stationary horizon for the stationary multi-species ASEP and decoupling results comparable to our work, see Remark~\ref{remark_comparison_ACH24}.

The queueing representation as well as the concept of speed processes have been subject to several generalisations for studying stationary measures in multi-class models beyond (T)ASEP, see for example~\cite{ABGM21,ANP25,BSS24} and the references therein.

\begin{remark} \label{remark_comparison_ACH24}
Our work has relations with results in~\cite{ACH24}, but with some essential differences. Indeed, the asymptotic decoupling of height profiles and decay of correlations were recently addressed for the two-species ASEP in Section~2.2.9 of~\cite{ACH24}. The authors consider asymptotics under a double scaling limit and therefore use properties of the directed landscape and the stationary horizon to obtain their results. In contrast, our work directly takes the large-time limit and provides the corresponding properties of the TASEP height functions and the stationary measure $\mu^{\rho_1,\rho_2}$ in the pre-limit. Accordingly, we analyse the translation-invariant stationary measure $\mu^{\rho_1,\rho_2,q}$ for ASEP (see Section~\ref{section_main_results_ASEP}) and obtain the decay of mixed correlations. For the decoupling of height functions, we restrict ourselves to the totally asymmetric case, as our tools are not yet fully available for the ASEP. Our results are proven for arbitrary densities, while~\cite{ACH24} focuses on densities in a \mbox{$t^{-1/3}$-neighbourhood} of $\tfrac{1}{2}$ (as they use inputs from other papers which were worked out for density $\tfrac{1}{2}$ only).

Corollary~2.15 of~\cite{ACH24} states that the two rescaled height functions are close to independent processes with probability converging to $1$, uniformly in space and for finitely many times. Due to the double limit, this result essentially concerns KPZ fixed points coupled via the directed landscape. In contrast, we consider the TASEP height functions at fixed points and state the product limit law explicitly. Still, our proofs of Theorem~\ref{thm_asymptotic_decoupling_height_functions_stationary}, Theorem~\ref{thm_asymptotic_decoupling_height_functions_flat} and Theorem~\ref{thm_asymptotic_decoupling_height_functions_general} extend to the spatial processes within $\mathcal{O}(t^{2/3})$-neighbourhoods of the characteristic lines. Corollary~2.15 of~\cite{ACH24} imposes a slope condition on the diffusive scaling limit of the initial height profiles, whereas we consider homogeneous initial data with suitable tail bounds. In both cases, a key requirement is that the two functions are nearly independent. To illustrate how this assumption can be satisfied, we construct a family of initial data with sufficiently weak, but non-zero, correlations in Section~\ref{section_decoupling_random_IC}.

Our Theorem~\ref{thm_decay_corr}, Corollary~\ref{corollary_decay_corr} and Corollary~\ref{cor_decay_corr_ASEP} correspond to Corollary~2.16 of~\cite{ACH24}, where the convergence of correlations is stated for the stationary two-species ASEP under the double scaling limit. In their proof, the authors perform a summation by parts twice by taking a further average first, which leads to boundary terms that need to be controlled in the scaling limit. Using the new formula of Proposition~\ref{prop_mixed_correlations_formula}, we do not require the averaging argument in our case, which is the first simplification. Furthermore, as explained in more detail in Section~\ref{section_decay_correlations}, Corollary~\ref{corollary_decay_corr} can be proven by similar, but not identical, arguments as Theorem~\ref{thm_decay_corr} (requiring additional estimates). This approach is followed by~\cite{ACH24} in their setting. We provide a simpler proof, observing that Corollary~\ref{corollary_decay_corr} directly follows from Theorem~\ref{thm_decay_corr} through the particle-hole duality and reflection properties of the stationary measure. This is particularly beneficial for the results in Corollary~\ref{cor_decay_corr_ASEP} on ASEP, as here the approach mentioned above is not applicable when considering only the large-time limit.
\end{remark}

\subsubsection{Asymptotic decoupling in the two-species TASEP with homogeneous initial data}  \label{section_main_results_deterministic}

The decoupling of height profiles, as established in Theorem~\ref{thm_asymptotic_decoupling_height_functions_stationary}, holds not only for stationary initial conditions but also for a broader class of translation-invariant initial data, which may be deterministic or random.

\paragraph{Deterministic initial data.} We consider a two-species TASEP with the following deterministic initial condition: for $\rho_1 \in (0,1)$ and $\rho_2 \in (0,1-\rho_1)$, we first place particles on $\Z$ such that the corresponding initial height profile fulfils
\begin{equation}
h^{\rho_1+\rho_2}(j,0) = (1-2(\rho_1+\rho_2)) j + H(j)
\end{equation}
with $\| H \|_\infty < \infty$. Among these particles, we select first class particles such that their initial height profile fulfils
\begin{equation}
h^{\rho_1}(j,0) = (1-2\rho_1)j + \tilde{H}(j)
\end{equation}
with $\| \tilde{H} \|_\infty < \infty$. The remaining particles are destined to be second class. The height functions are related to the marginal particle configurations as in \eqref{eq_def_height_fct}.

The asymptotic behaviour of the respective height profiles is known by Corollary~2.8 of~\cite{FO17}:
	let $\rho \in \{\rho_1,\rho_1+\rho_2\}$, $\chi = \rho(1-\rho)$ and
	\begin{equation}
	\mathfrak{h}^\rho(0,t) = \frac{h^\rho((1-2\rho)t,t) - (1-2\chi)t}{-2\chi^{2/3} t^{1/3}}.
	\end{equation}
	Then, for any $s \in \R$, it holds
	\begin{equation}
	\lim_{t \to \infty} \Pb(\mathfrak{h}^\rho(0,t) \leq s) = F_{\textup{GOE}}(2^{2/3} s),
	\end{equation}
	where $F_{\textup{GOE}}$ denotes the GOE Tracy-Widom distribution function.
	
As marginals of the two-species TASEP, we again obtain an asymptotic decoupling in the large-time limit.

\begin{thm} \label{thm_asymptotic_decoupling_height_functions_flat}
	The height fluctuations of $\mathfrak{h}^{\rho_1}(0,t)$ and $\mathfrak{h}^{\rho_1+\rho_2}(0,t)$ are asymptotically independent: for any $r,s \in \R$, it holds
	\begin{equation}
	\begin{aligned}
	 \lim_{t \to \infty} \Pb (\mathfrak{h}^{\rho_1}(0,t) \leq s, \mathfrak{h}^{\rho_1 + \rho_2}(0,t) \leq r)
	& = \lim_{t \to \infty} \Pb (\mathfrak{h}^{\rho_1}(0,t) \leq s) \Pb (\mathfrak{h}^{\rho_1 + \rho_2}(0,t) \leq r) \\
	& = F_{\textup{GOE}}(2^{2/3} s) F_{\textup{GOE}}(2^{2/3} r).
	\end{aligned}
	\end{equation}
\end{thm}

Theorem~\ref{thm_asymptotic_decoupling_height_functions_flat} is proven in Section~\ref{section_asymptotic_decoupling_periodic} using similar methods as those in the proof of Theorem~\ref{thm_asymptotic_decoupling_height_functions_stationary}. Therefore, in Corollary~\ref{cor_localisation_geodesics_deterministic_IC_half_periodic}, we provide the localisation of backwards paths in a TASEP with periodic initial condition.
Without further argument, it also applies to deterministic initial conditions with different constant densities on the left and right of the origin.

In this deterministic case, we do not require the specific properties of the initial condition that were needed for $\mu^{\rho_1,\rho_2}$. Instead, we apply translation invariance and tail estimates for the marginal processes.

In Theorem~\ref{thm_asymptotic_decoupling_height_functions_flat}, we could also allow shifts of order $\mathcal{O}(t^{2/3})$ as in Theorem~\ref{thm_asymptotic_decoupling_height_functions_stationary}, but they would not alter the limit distribution.

\paragraph{Random initial data.}
The proof of the decay of the two-point function in Theorem~\ref{thm_decay_corr} relies on a bound for correlations in the stationary configuration $\eta \sim \mu^{\rho_1,\rho_2}$, as stated in Lemma~\ref{lemma_formula_for_IC_process_M}. Given this bound, the decoupling of height profiles in the stationary case can be established using the same strategy as in the proof of Theorem~\ref{thm_asymptotic_decoupling_height_functions_flat}. In particular, the specific properties of $\mu^{\rho_1,\rho_2}$ used in the proof of Theorem~\ref{thm_asymptotic_decoupling_height_functions_stationary} are not necessary for the decoupling. Instead, there exist sufficient criteria that can also be met by other random initial conditions.

As before, we let $\rho_1 \in (0,1)$ and $\rho_2 \in (0,1-\rho_1)$, and consider a two-species TASEP with a random initial configuration $\eta$. For $\rho \in \{\rho_1,\rho_1+\rho_2\}$, we define $\eta^\rho$ and $h^\rho$ as in \eqref{eq_def_marginal_configurations} and \eqref{eq_def_height_fct}. Again, the rescaled height profiles are given by
\begin{equation}
\mathfrak{h}^\rho(w,t) = \frac{h^\rho((1-2\rho)t+2w \chi^{1/3} t^{2/3},t) - (1-2\chi)t - 2w(1-2\rho)\chi^{1/3}t^{2/3}}{-2\chi^{2/3}t^{1/3}}
\end{equation}
for $w \in \R$ and $\chi = \rho(1-\rho)$.

We impose the following assumption on the initial condition:

\begin{assumpt} \label{assumption_decoupling_general_IC}
We suppose that the following properties are satisfied:
	\begin{itemize}
		\item[(a)] Spatial homogeneity: the distribution of $\eta$ is translation-invariant.
		\item[(b)] Tail bounds: for $\rho \in \{\rho_1,\rho_1+\rho_2\}$ and uniformly for $t$ large enough, there exist constants $C,c>0$ such that
		\begin{equation}
		\Pb(|h^{\rho}(xt^{2/3},0)-(1-2\rho)xt^{2/3}| > s t^{1/3}) \leq C e^{-c s |x|^{-1/2}}
		\end{equation}
		for all $x \neq 0$ and $s > 0$.
		\item[(c)] Small initial correlations: it holds $h^{\rho_1}(y,0) = \hat{h}^{\rho_1}(y,0) + \delta_h(y)$ such that $\hat{h}^{\rho_1}(\cdot,0)$ and $h^{\rho_1+\rho_2}(\cdot,0)$ are independent, and
		\begin{equation}
		\lim_{t \to \infty} \Pb\bigg(\sup_{|y| \leq t^{2/3+\e}}|\delta_h(y)| > t^\sigma \bigg) = 0
		\end{equation}
		for some $\e,\sigma \in (0,\tfrac{1}{3})$.
	\end{itemize}
\end{assumpt}

As for deterministic initial data, it suffices to have Assumption~\ref{assumption_decoupling_general_IC}(a) up to uniformly bounded perturbations. We require (b) for an approximate localisation of backwards paths and for rough tail bounds for $h^{\rho}((1-2\rho)t,t)$. The bound on the correlated part of the initial height profiles in (c) is crucial to observe a decoupling at large times, see also Remark~E.4 of~\cite{ACH24} for a comparable condition in the double-limit setting.

With these ingredients, Theorem~\ref{thm_asymptotic_decoupling_height_functions_flat} generalises to random initial data.

\begin{thm} \label{thm_asymptotic_decoupling_height_functions_general}
	Suppose the initial condition of the two-species TASEP satisfies Assumption~\ref{assumption_decoupling_general_IC} and $\mathfrak{h}^{\rho_1}(w,t)$ and $\mathfrak{h}^{\rho_1+\rho_2}(w,t)$ have continuous limit distributions. Then, they are asymptotically independent: for any $w,z,r,s \in \R$, it holds
	\begin{equation}
	\lim_{t \to \infty} \Pb(\mathfrak{h}^{\rho_1}(w,t) \leq s, \mathfrak{h}^{\rho_1+\rho_2}(z,t) \leq r ) = \lim_{t \to \infty} \Pb(\mathfrak{h}^{\rho_1}(w,t) \leq s) \Pb (\mathfrak{h}^{\rho_1+\rho_2}(z,t) \leq r ).
	\end{equation}
\end{thm}

We shortly prove Theorem~\ref{thm_asymptotic_decoupling_height_functions_general} in Section~\ref{section_decoupling_random_IC}, where we also give an explicit example of a class of initial conditions that satisfy Assumption~\ref{assumption_decoupling_general_IC}. This class is obtained through a generalisation of the queueing construction of $\mu^{\rho_1,\rho_2}$ explained in Section~\ref{section_queueing_representation}.

\subsection{Exact results for the stationary two-species ASEP} \label{section_main_results_ASEP}

While the primary focus of our work is on TASEP, the results concerning the two-point function in the stationary regime extend to the more general asymmetric simple exclusion process.

\paragraph{Model and notation.}
We consider a two-species, continuous-time ASEP on $\Z$, where the interaction rules between different particle types are the same as in the totally asymmetric case discussed previously. Particles attempt to jump to the right with rate $1$ and to the left with rate $q \in [0,1)$; thus, $q=0$ corresponds to total asymmetry.
For given densities $\rho_1 \in (0,1)$ and $\rho_2 \in (0,1-\rho_1)$ of first and second class particles and a fixed asymmetry parameter $q$, there exists a unique translation-invariant stationary measure $\mu^{\rho_1,\rho_2,q}$ for the process. Its projections onto the marginal single-species ASEPs, consisting of either the first class or all particles, are again Bernoulli product measures with densities $\rho_1$ and $\rho_1+\rho_2$, respectively. We refer to \cite{EMP09,FKS91,Lig76,Mar20} for the existence and properties of $\mu^{\rho_1,\rho_2,q}$.

We define the configurations $\eta_t$, $\eta_t^{\rho_1}$, $\eta_t^{\rho_1+\rho_2}$ and the height functions $h^\rho$ for $\rho \in \{\rho_1,\rho_1+\rho_2\}$ as in Section~\ref{section_main_results_stationary}. Now, the quantity $N_t^\rho$ is the \emph{net} particle current across $(0,1)$. That is, $N^\rho_t$ equals the number of particles that jumped from $0$ to $1$ minus the number of particles that jumped from $1$ to $0$ until time $t$.

For ASEP, the rescaled height function $\mathfrak{h}^\rho(w,t)$ is given by
\begin{equation} \label{eq_rescaled_height_fct_ASEP}
\frac{h^\rho((1-2\rho)t+2w \chi^{1/3} t^{2/3},(1-q)^{-1}t) - (1-2\chi)t - 2w(1-2\rho)\chi^{1/3}t^{2/3}}{-2\chi^{2/3}t^{1/3}}
\end{equation}
for $w \in \R$ and $\chi = \rho(1-\rho)$, and \cite{Agg18} established
\begin{equation}
\lim_{t \to \infty} \Pb(\mathfrak{h}^\rho(w,t) \leq s) = F_{\text{BR},w}(s).
\end{equation}
The two-point function of the process $\eta_t$ is defined as in \eqref{eq_correlation_matrix}. Corollary~2.6 of \cite{LS25} provides convergence of the diagonal terms to the KPZ-universal scaling limit: in the large-time limit and for $i \in \{1,2\}$ and $\rho = \rho(i) \in \{\rho_1,\rho_1+\rho_2\}$,
\begin{equation}
2\chi^{1/3}t^{2/3} S^\#_{i,i}((1-2\rho)t+2w\chi^{1/3}t^{2/3}, (1-q)^{-1} t)
\end{equation}
converges to $\chi f_{\text{KPZ}}(w)$ when integrated against smooth functions of $w$ with compact support.
For other speeds, the rescaled correlations converge to zero, as their connection to the distribution of a second class particle remains valid for ASEP. See also \cite{BS09} for an explicit proof for more general exclusion processes.

\paragraph{Main result.} Modifying the proofs of Theorem~\ref{thm_decay_corr} and Corollary~\ref{corollary_decay_corr}, we establish that the rescaled off-diagonal terms of the two-point function converge to zero also for ASEP.

\begin{cor} \label{cor_decay_corr_ASEP}
	Let $S^\#$ denote the two-point function of the stationary two-species ASEP with $\eta_t \sim \mu^{\rho_1,\rho_2,q}$ and let $\phi : \R \to \R$ be a smooth function with compact support. Then, it holds
	\begin{equation}
	\lim_{t \to \infty} t^{-2/3} \sum_{w \in t^{-2/3} \Z} \phi(w) t^{2/3} S^\#_{2,1}((1-2(\rho_1+\rho_2))t+wt^{2/3}, (1-q)^{-1}t) = 0
	\end{equation}
	and
	\begin{equation}
	\lim_{t \to \infty} t^{-2/3} \sum_{w \in t^{-2/3} \Z} \phi(w) t^{2/3} S^\#_{1,2}((1-2\rho_1)t+wt^{2/3}, (1-q)^{-1}t) = 0.
	\end{equation}
	The same is true for general speeds $v\neq 1-2(\rho_1+\rho_2)$ or $v \neq 1-2\rho_1$, respectively.
\end{cor}

Corollary~\ref{cor_decay_corr_ASEP} is proven in Section~\ref{section_decay_correlations_ASEP}, where we outline the modifications required for the general asymmetric case. The rest of our work focuses on TASEP, as this setting supports a more direct analytical approach. Moreover, with our methods, the decoupling of height functions in the large-time limit is not feasible for ASEP, due to the lack of the concatenation property \eqref{eq_concatenation_property_height_fct} and a corresponding theory of backwards paths.

\paragraph{Outline.} In Section~\ref{section_backwards_geodesics}, we recall the theory of backwards paths for single-species TASEP height functions and establish their localisation for the stationary and the (half-)periodic initial conditions. Section~\ref{section_stationary_measure} discusses the queueing construction and further properties of the stationary measure $\mu^{\rho_1,\rho_2}$. Section~\ref{section_asymptotic_decoupling} contains the proofs of the first two asymptotic decoupling results, Theorem~\ref{thm_asymptotic_decoupling_height_functions_stationary} and Theorem~\ref{thm_asymptotic_decoupling_height_functions_flat}. In Section~\ref{section_mixed_correlations}, we examine the mixed correlations in the stationary two-species TASEP, and then generalise the analysis to the stationary two-species ASEP. Finally, in Section~\ref{section_decoupling_random_IC}, we extend the proof of Theorem~\ref{thm_asymptotic_decoupling_height_functions_flat} to random initial data. The appendices include proofs of auxiliary results on backwards paths, one-point estimates, and the stationary measure.

\section{TASEP height function backwards paths} \label{section_backwards_geodesics}

This section discusses the definition and the localisation of backwards paths for single-species TASEP height functions.

First, we recall the notion of basic coupling and revisit some concepts from Section~3 of~\cite{FN24} and Section~4 of~\cite{BF22}.

As explored by Harris~\cite{Har72,Har78}, the evolution of TASEP can be constructed graphically by describing the jump attempts at each site $z \in \Z$ by a Poisson process $\mathcal{P}_z$ with rate one, where $\{\mathcal{P}_z, z \in \Z\}$ are independent of each other and of the initial condition of the process. Almost surely, there is at most one jump attempt at any given time. Two processes are coupled by \emph{basic coupling} if they are constructed using the same family of Poisson processes.

We fix a TASEP height function $h$ and denote by  $h^{\textup{step}}_{y,\tau}$ the height function of a TASEP starting at time $\tau \geq 0$ from a step initial condition centred at $y \in \Z$, meaning $h^{\textup{step}}_{y,\tau}(j,\tau) = |j-y|$. Under basic coupling, it holds
\begin{equation} \label{eq_concatenation_property_height_fct}
h(j,t) = \min_{y \in \Z} \{ h(y,\tau) + h^{\text{step}}_{y,\tau}(j,t) \}
\end{equation}
for each $\tau \in [0,t]$.

A backwards trajectory $\{x(\tau)\}_{\tau:t\downarrow 0 }$ with $x(t) = x$ is called a \emph{backwards geodesic} if
\begin{equation} \label{eq_geodesic_property}
h(x,t) = h(x(\tau),\tau)+h^{\text{step}}_{x(\tau),\tau}(x,t)
\end{equation}
for all $\tau \in [0,t]$. This notion is an analogue of the geodesics in LPP models. A valuable property of a backwards geodesic is that if we localise it in a deterministic space-time region, then $h(x,t)$ is independent of the randomness outside this region~\cite{BF22}.

We define \emph{backwards paths} as in Definition~3.5 of~\cite{FN24}. By Proposition~4.2 of~\cite{BF22}, they are, in particular, backwards geodesics.

We start at time $t$ from $x(t) = x$ and go backwards in time. For each $s \in [0,t]$ such that $x(s)=y$ and there is a jump attempt at site $y$ at time $s$, we update $x(s)$ as follows: if a jump occurred, meaning $h(x(s),s) = h(x(s),s^-)+2$, then we let $x(s^-) = x(s)$. If no jump occurs, then we choose $x(s^-) \in \{x(s)-1,x(s)+1\}$ such that $h(x(s^-),s) = h(x(s),s)-1$. By this means, backwards paths are not unique, see also Figure~\ref{figure_height_fct_backwards_paths}.

\begin{figure}[t!]
\centering
\includegraphics[scale=1]{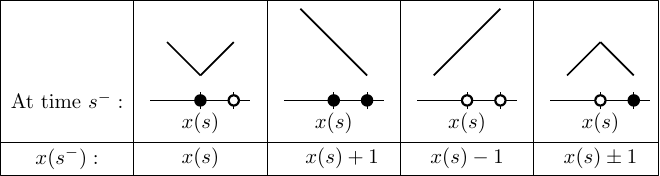}
\caption{Possible values of $x(s^-)$ depending on the height profile respectively the particle configuration at time $s^-$ around the position $x(s)$. The black dots represent particles while the white dots correspond to holes.}
\label{figure_height_fct_backwards_paths}
\end{figure}

A localisation of backwards paths can be achieved by a comparison to paths in a TASEP with step initial condition.

\subsection{Comparison of backwards paths} \label{section_comparison_backwards_paths}

We consider times $t_1,t_2 \in [0,t]$ with $t_1 < t_2$, and sites $x_1,x_2 \in \Z$. Further, we couple all processes by basic coupling. Then, the following comparisons hold true:

 \begin{prop} \label{prop_comparison_rightmost_backwards_geodesics_to_step}
	Suppose $x(t_1) \leq x_1$ and $ x(t_2) \leq x_2$. Let $x^{\textup{step},r}(\tau)$ be the rightmost backwards path with respect to $h^{\textup{step}}_{x_1,t_1}$ starting at time $t_2$ from position $x_2$.
	Then, it holds
	\begin{equation}
	x(\tau) \leq x^{\textup{step},r}(\tau) \text{ for all } \tau \in [t_1,t_2].
	\end{equation}
\end{prop}

\begin{prop} \label{prop_comparison_leftmost_backwards_geodesics_to_step}
	Suppose $x(t_1) \geq x_1$ and $x(t_2) \geq x_2$. Let $x^{\textup{step},l}(\tau)$ be the leftmost backwards path with respect to $h^{\textup{step}}_{x_1,t_1}$ starting at time $t_2$ from position $x_2$.
	Then, it holds
	\begin{equation}
	x(\tau) \geq x^{\textup{step},l}(\tau) \text{ for all } \tau \in [t_1,t_2].
	\end{equation}
\end{prop}

Proposition~\ref{prop_comparison_rightmost_backwards_geodesics_to_step} and Proposition~\ref{prop_comparison_leftmost_backwards_geodesics_to_step} allow to localise any backwards path $x(\tau)$ starting at a given space-time point $(x,t)$ on the whole time interval $[0,t]$ by only two ingredients: the region of its endpoint $x(0)$ and the localisation of backwards paths in a TASEP with step initial condition from~\cite{BF22}. We apply this strategy in Section~\ref{section_localisation_backwards_paths}. Notice that we did not make any assumptions about $h(\cdot,0)$.

The comparison results are related to Lemma 3.7 of~\cite{FN24}, which compares a rightmost backwards path in a TASEP with initially no particles to the right of the origin to the corresponding path in a TASEP with step initial condition. There, due to the attractiveness of TASEP, the order of particle configurations is maintained over time. In our case, the initial particle configurations (at time $t_1$) do not possess the right order on the whole lattice. Nonetheless, the order persists in a certain subregion, such that the backwards paths can still be compared. Similar arguments appear in the proof of Proposition 4.9 of~\cite{BF22}, see also Lemma 3.7 of~\cite{FG26} in the setting of particle positions.

\begin{proof}[Proof of Proposition~\ref{prop_comparison_rightmost_backwards_geodesics_to_step}]
	Suppose the following statement holds true: 	
	
	\textbf{Claim:} For all times $\tau \in [t_1,t_2]$, whenever there is a particle of $h$ at a site $\leq x(\tau)$, there is also a particle of $h^{\textup{step}}_{x_1,t_1}$ at this site.
	
	Then, looking backwards in time, we first note $x(t_2) \leq x_2 = x^{\textup{step},r}(t_2)$. Suppose there is some $\tau \in [t_1,t_2]$ such that $x(\tau) = x^{\textup{step},r}(\tau) = x^*$ and there is a jump event at $x^*$ at time $\tau$. We claim:
	\begin{itemize}
	\item[(a)] If $x(\tau^-)=x(\tau)+1$, then $x^{\textup{step},r}(\tau^-) = x^{\textup{step},r}(\tau) +1$.
	\item[(b)] If $x^{\textup{step},r}(\tau^-) = x^{\textup{step},r}(\tau)-1$, then $x(\tau^-) = x(\tau)-1$.
	\end{itemize}
	
	Regarding (a), we observe that $x(\tau^-)=x(\tau)+1$ occurs only if at time $\tau^-$, there is a particle of $h$ at site $x^*+1$. But since $x^* +1 = x(\tau^-)$, the claim implies that there is a particle of $h^{\textup{step}}_{x_1,t_1}$ as well. As $x^{\textup{step},r}(\tau)$ is the rightmost backwards path, this means $x^{\textup{step},r}(\tau^-) = x^{\textup{step},r}(\tau) +1$.
	
	Regarding (b), we note that $x^{\textup{step},r}(\tau^-) = x^{\textup{step},r}(\tau)-1$ implies that at time $\tau^-$, there are holes of $h^{\textup{step}}_{x_1,t_1}$ at $x^*$ and $x^*+1$.
	If $x(\tau^-) = x(\tau)$, then there would be a particle of $h$ at time $\tau^-$ at $x^* = x(\tau^-) $. If $x(\tau^-) = x(\tau)+1$, then there would be a particle of $h$ at time $\tau^-$ at $x^* +1 = x(\tau^-)$. Since both options contradict the claim, we deduce $x(\tau^-) = x(\tau)-1$.
	
	To conclude, we prove the claim. The statement holds true at time $t_1$ because $x(t_1) \leq x_1$. Suppose there exists a time $\tau \in [t_1,t_2]$ such that it is not valid any more. We can choose $\tau$ minimal. Then, there are particles of $h$ respectively $h^{\textup{step}}_{x_1,t_1}$ with positions $p(\tau)$ and $p^{\textup{step}}(\tau)$ such that $p(\tau^-) = p^{\textup{step}}(\tau^-)$ and at time $\tau$, there is no particle of $h^{\textup{step}}_{x_1,t_1}$ at the site $p(\tau) \leq x(\tau)$. The first assertion holds by minimality of $\tau$ because $p(\tau) \leq x(\tau)$ implies $p(u) \leq x(u)$ for all $u \in [t_1,\tau]$. The only possible scenario is that at time $\tau$, the jump attempt of $p(\tau)$ is suppressed while $p^{\textup{step}}(\tau)$ jumps, meaning $p^{\text{step}}(\tau) = p(\tau) +1 \eqqcolon p^*$. Then, there must be another particle of $h$ with position $p_0(\tau) = p^*$. The minimality of $\tau$ implies $p^* = p_0(\tau) = p_0(\tau^-) > x(\tau^-) $. Since both sites $p^*-1$ and $p^*$ are occupied in $h$ at time $\tau^-$, $x(\tau) \geq p^*-1$ would lead to $x(\tau^-) \geq p^*$, a contradiction. Therefore, we get $x(\tau) < p^*-1$. But this in turn contradicts $p^*-1=p(\tau) \leq x(\tau)$. Thus, the time $\tau$ did not exist.
\end{proof}

\begin{proof}[Proof of Proposition~\ref{prop_comparison_leftmost_backwards_geodesics_to_step}]
	The proof of Proposition~\ref{prop_comparison_leftmost_backwards_geodesics_to_step} is analogous to the proof of Proposition~\ref{prop_comparison_rightmost_backwards_geodesics_to_step}. Here, we use the following fact: for all $\tau \in [t_1,t_2]$, whenever there is a hole of $h$ at a site $\geq x(\tau)+1$, there is a hole of $h^{\textup{step}}_{x_1,t_1}$ at the same position.
\end{proof}

\subsection{Localisation of backwards paths} \label{section_localisation_backwards_paths}

In this section, we first localise the backwards paths of TASEP height functions for general initial data. Then, we show that the assumptions of our result are fulfilled for stationary and (half-)periodic initial conditions.

\begin{prop} \label{prop_localisation_backwards_geodesics_general}
	Consider a backwards path in a TASEP with any initial condition, starting at $x(t) = \alpha t$ with $\alpha \in \R$ fixed. For some $\beta$ with $\alpha - \beta \in (-1,1)$, suppose there are constants $C,c > 0$ such that uniformly for $t$ large enough, $\Pb(|x(0) - \beta t| \geq M t^{2/3}) \leq C e^{-c M^2}$. Then, there exist constants $C,c > 0$ such that uniformly for $t$ large enough, it holds
	\begin{equation}
	\Pb(|x(\tau) - \beta t - (\alpha - \beta) \tau| \leq M t^{2/3} \text{ for all } \tau \in [0,t]) \geq 1 - C e^{-c M^2}.
	\end{equation}
\end{prop}

\begin{proof}
	Let $E_0 \coloneqq \{ |x(0) - \beta t| < \tfrac{M}{2} t^{2/3} \}$. Then, given that $E_0$ occurs, Proposition~\ref{prop_comparison_rightmost_backwards_geodesics_to_step} and Proposition~\ref{prop_comparison_leftmost_backwards_geodesics_to_step} yield:
	\begin{itemize}
		\item For all $\tau \in [0,t]$, it holds $x(\tau) \leq x^{\text{step},r}(\tau)$, where $x^{\text{step},r}(\tau)$ is the rightmost backwards path with $x^{\text{step},r}(t) = \alpha t + \tfrac{M}{2} t^{2/3}$ in a TASEP with step initial condition centred at $\beta t + \tfrac{M}{2} t^{2/3}$.
		\item For all $\tau \in [0,t]$, it holds $x(\tau) \geq x^{\text{step},l}(\tau)$, where $x^{\text{step},l}(\tau)$ is the leftmost backwards path with $x^{\text{step},l}(t) = \alpha t - \tfrac{M}{2} t^{2/3}$ in a TASEP with step initial condition centred at $\beta t -\tfrac{M}{2} t^{2/3}$.
	\end{itemize}
	As a consequence, we obtain
	\begin{equation} \begin{aligned}
	&\Pb(|x(\tau) - \beta t -  (\alpha - \beta ) \tau| > M t^{2/3} \text{ for some } \tau \in [0,t]) \\
	& \leq C e^{-cM^2} + \Pb(x^{\text{step},r}(\tau) > \beta t + (\alpha-\beta)\tau + M t^{2/3} \text{ for some } \tau \in [0,t]) \\
	&\hphantom{\leq} + \Pb(x^{\text{step},l}(\tau) < \beta t + (\alpha-\beta)\tau - M t^{2/3} \text{ for some } \tau \in [0,t]).
	\end{aligned} \end{equation}
	Since we consider TASEPs with step initial condition shifted by $\beta t \pm \tfrac{M}{2} t^{2/3}$, we find
	$(x^{\text{step},r}(\tau)) \overset{(d)}{=} (x^{0,r}(\tau) + \beta t + \tfrac{M}{2} t^{2/3})$ and $(x^{\text{step},l}(\tau)) \overset{(d)}{=} (x^{0,l}(\tau) + \beta t - \tfrac{M}{2} t^{2/3})$, where $x^{0,r}(\tau)$ and $x^{0,l}(\tau)$ are rightmost respectively leftmost backwards paths in a TASEP with step initial condition centred at the origin, starting at site $ (\alpha - \beta) t$ at time $t$. The sum above equals
	\begin{equation} \label{eq_pf_prop_localisation_0}
	\begin{aligned}
	& C e^{-c M^2} + \Pb( x^{0,r}(\tau) > (\alpha-\beta)\tau + \tfrac{M}{2} t^{2/3} \text{ for some } \tau \in [0,t]) \\
	& + \Pb( x^{0,l}(\tau) < (\alpha-\beta)\tau - \tfrac{M}{2} t^{2/3} \text{ for some } \tau \in [0,t]).
	\end{aligned}
	\end{equation}
	By Proposition 4.9 of~\cite{BF22}, see also Proposition 3.8 of~\cite{FN24}, we conclude that $ \eqref{eq_pf_prop_localisation_0} \leq C e^{-c M^2}$ uniformly for all $t$ large enough.
\end{proof}

\subsubsection{Localisation of backwards paths in a stationary TASEP} \label{section_localisation_backwards_paths_stationary}

We can apply Proposition~\ref{prop_localisation_backwards_geodesics_general} to a stationary TASEP. Although we use a slightly different definition of backwards paths than~\cite{BF22}, the following localisation of endpoints still applies because it only uses the property \eqref{eq_geodesic_property}.

\begin{lem}[Proposition 4.7 of~\cite{BF22}] \label{lemma_end_point_geodesics_stationary_TASEP}
	Consider a backwards path in a stationary TASEP with density $\rho \in (0,1)$, starting at $x(t) = (1-2\rho)t$. Then, there exist constants $C,c > 0$ such that uniformly for all $t$ large enough, it holds
	\begin{equation}
	\Pb(|x(0)| \geq M t^{2/3}) \leq C e^{-c M^2}.
	\end{equation}
\end{lem}

Combining Lemma~\ref{lemma_end_point_geodesics_stationary_TASEP} with Proposition~\ref{prop_localisation_backwards_geodesics_general} for $\alpha = 1-2\rho$ and $\beta =0$, we obtain:

\begin{cor} \label{cor_localisation_backwards_geodesics_stationary_TASEP}
	Consider a backwards path in a stationary TASEP with density $\rho \in (0,1)$, starting at $x(t) = (1-2\rho)t$. Then, there exist constants $C,c > 0$ such that uniformly for all $t$ large enough, it holds
	\begin{equation}
	\Pb(|x(\tau) - (1-2\rho)\tau| \leq M t^{2/3} \text{ for all } \tau \in [0,t]) \geq 1 - C e^{-c M^2}.
	\end{equation}
\end{cor}

\subsubsection{Localisation of backwards paths in a TASEP with (half-)periodic initial condition}
\label{section_localisation_backwards_paths_periodic}

Next, we examine a TASEP with a deterministic initial condition. We suppose that initially, its densities on $\Z_{\leq 0}$ and $\N$ are $\rho \in (0,1)$ and $\lambda \in [0,\rho]$, such that its initial height profile satisfies
\begin{equation} \label{equation_initial_condition_half_periodic_TASEP}
\begin{aligned}
h(j,0) = \begin{cases} (1-2\rho) j + H(j),  &j < 0, \\ (1-2\lambda)j + H(j),  &j \geq 0, \end{cases}
\end{aligned}
\end{equation}
with $ \| H  \|_{\infty}  < \infty$. We localise the endpoints of backwards paths:

\begin{prop} \label{prop_geodesic_end_point_deterministic_IC_half_periodic}
	Consider a backwards path in a TASEP with the initial condition \eqref{equation_initial_condition_half_periodic_TASEP}, starting from $x(t) =\alpha t$ with $\alpha \in [1-2\rho, 1-2\lambda] \cap (-1,1)$. Then, there exist constants $C,c > 0$ such that uniformly for all $t$ large enough, it holds
	\begin{equation}
	\Pb(|x(0)| \geq M t^{2/3}) \leq C e^{-c M^2}.
	\end{equation}
\end{prop}

Proposition~\ref{prop_geodesic_end_point_deterministic_IC_half_periodic} is proven in Appendix~\ref{appendix_A}. Combined with Proposition~\ref{prop_localisation_backwards_geodesics_general}, it yields:
\begin{cor} \label{cor_localisation_geodesics_deterministic_IC_half_periodic}
	Consider a backwards path in a TASEP with the initial condition \eqref{equation_initial_condition_half_periodic_TASEP}, starting from $x(t) =\alpha t$ with $\alpha \in [1-2\rho, 1-2\lambda] \cap (-1,1)$. Then, there exist constants $C,c > 0$ such that uniformly for all $t$ large enough, it holds
	\begin{equation}
	\Pb ( | x(\tau) - \alpha \tau | \leq M t^{2/3} \text{ for all } \tau \in [0,t]) \geq 1 - C e^{-c M^2}.
	\end{equation}
\end{cor}

\begin{remark}
If we choose $\lambda > \rho$, then at time $t$, there is a shock around the position $(1-\rho-\lambda)t$. For $x(t)$ starting at a position macroscopically away from $(1-\rho-\lambda)t$, we can localise its endpoint similarly as in Proposition~\ref{prop_geodesic_end_point_deterministic_IC_half_periodic}. At $(1-\rho-\lambda)t$, the two characteristic lines $-(\lambda-\rho)t+(1-2\rho)\tau$ and $(\lambda-\rho)t+(1-2\lambda)\tau$ meet, meaning that our arguments do not provide a localisation of the endpoint in a single $\mathcal{O}(t^{2/3})$--interval.

In all results of Section~\ref{section_localisation_backwards_paths}, we can choose $M = \mathcal{O}(t^{\e})$ with $\e \in (0,\tfrac{1}{3})$ fixed.
\end{remark}

\section{The translation-invariant stationary measure for the two-species TASEP}  \label{section_stationary_measure}

In this section, we recall the construction and analyse the translation-invariant stationary measure $\mu^{\rho_1,\rho_2}$ of a two-species TASEP with densities $\rho_1 \in (0,1)$ and $\rho_2 \in (0,1-\rho_1)$ of first and second class particles.

\subsection{Queueing representation} \label{section_queueing_representation}

The stationary measure is constructed through a queueing representation, following the interpretation introduced in~\cite{Ang06} and extended in~\cite{FM07}. For additional context, see also~\cite{AAV11,BSS23,EFM09}. In contrast to~\cite{BSS23,FM07}, our construction is for a process with rightward jumps rather than leftward.

We consider a line of arrival points $\mathcal{A}$ that is a Bernoulli process on $\Z$ with rate $\rho_1$, and a line of service points $\mathcal{S}$ that is a Bernoulli process on $\Z$ with rate $\rho_1 + \rho_2$. The processes $\mathcal{A}$ and $\mathcal{S}$ are independent of each other. We write
\begin{equation}
a(i) = \begin{cases}
1 &\text{ if there is a point at } i \in \Z \text{ in } \mathcal{A}, \\ 0 &\text{ else},
\end{cases}
\end{equation}
and define $s(i)$ similarly with respect to $\mathcal{S}$. Further, we denote by $\mathcal{A}_I$ respectively $\mathcal{S}_I$ the number of points of the processes in an interval $I \subseteq \Z$.

To construct the stationary measure $\mu^{\rho_1,\rho_2}$, we go through the points in $\mathcal{A}$ in an arbitrary order. For each point in $\mathcal{A}$ at some site $i \in \Z$, we draw an edge to the maximal position $j \leq i$ such that $s(j) = 1$ and the point in $\mathcal{S}$ is not connected to another point in $\mathcal{A}$ yet.

The sites $j$ obtained from this procedure form the positions of first class particles in our process. In the queueing interpretation, they mark departure points. The corresponding process $\mathcal{D}$ is again a Bernoulli process with rate $\rho_1$. We define $d(i)$ similarly as $a(i),s(i)$ above.
The unused points in $\mathcal{S}$ are taken as the positions of second class particles. By this means, second class particles have density $\rho_2$ but are not positioned independently of each other. We illustrate this construction in Figure~\ref{figure_queueing_representation_stationary_measure}. A realisation of $\mu^{\rho_1,\rho_2}$ is given by
\begin{equation} \eta : \Z \to \{1,2,+\infty\},\quad \eta(z) = \begin{cases}
1 &\text{ if } d(z) = 1, \\ 2  &\text{ if } s(z) = 1, d(z) = 0, \\ +\infty &\text{ else.}
\end{cases} \end{equation}
The order in which we regard the points in $\mathcal{A}$ in the construction has no influence on the values of $\eta$~\cite{Ang06,EFM09}. In the following, we assume that the edges in the assignment of points in $\mathcal{A}$ to points in $\mathcal{S}$ do not intersect.

\begin{figure}[t!]
\centering
\includegraphics[scale=1]{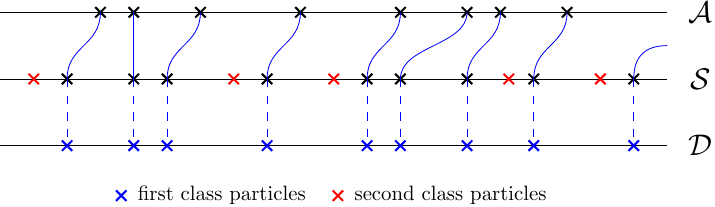}
\caption{Assigning the points in $\mathcal{A}$ to points in $\mathcal{S}$, we construct the departure process $\mathcal{D}$, that represents the positions of first class particles. The unused points in $\mathcal{S}$ are interpreted as second class particles.}
\label{figure_queueing_representation_stationary_measure}
\end{figure}

The marginals of $\mu^{\rho_1,\rho_2}$ for first and for (first + second) class particles are Bernoulli measures with densities $\rho_1$ and $\rho_1 + \rho_2$, and, as such, stationary for the single-species TASEP.

\subsection{Independence induced by second class particles}

An important property of the stationary measure $\mu^{\rho_1,\rho_2}$ is that the values of its realisations to the right and to the left of a given second class particle are independent of each other. This has been proven in~\cite{DJLS93,FFK94,Spe94}, see also~\cite{Ang06} for a simple explanation in the queueing representation and~\cite{FM07} for a generalisation to $n$-TASEP.
\begin{lem} \label{lemma_stationary_measure_independence_second_class_p}
	Let $\eta$ be a sample of $\mu^{\rho_1,\rho_2}$ and fix $z \in \Z$.
	Then, conditioned on $\eta (z) = 2$, the configurations $(\eta(i), i < z)$ and $(\eta(i), i > z)$ are independent.
\end{lem}

In order to apply Lemma~\ref{lemma_stationary_measure_independence_second_class_p}, we require that in a sufficiently large time interval, we find a second class particle with high probability.

\begin{lem} \label{lemma_second_class_particle_in_interval_with_high_prob}
	Let $\eta$ be a sample of $\mu^{\rho_1,\rho_2}$. There are constants $C, c > 0$ such that for each interval $I \subseteq \Z$ with length $|I| = \ell$, it holds
	\begin{equation}
	\Pb (\exists i \in I: \ \eta(i) = 2) \geq 1 - C e^{-c \ell}.
	\end{equation}
\end{lem}

It is likely that Lemma~\ref{lemma_second_class_particle_in_interval_with_high_prob} is known in the literature, but as we do not know where, we produce a proof of it in Appendix~\ref{appendix_B}.

\section{Asymptotic decoupling of height functions} \label{section_asymptotic_decoupling}

In this section, we prove our results on the asymptotic decoupling of the marginal height functions in two-species TASEPs with stationary or periodic initial conditions.

\subsection{Decoupling in a stationary two-species TASEP} \label{section_asymptotic_decoupling_stationary}

\begin{proof}[Proof of Theorem~\ref{thm_asymptotic_decoupling_height_functions_stationary}]
Without loss of generality, we set $w = z = 0$.
We choose $\nu \in (\tfrac{2}{3},1)$ and $\e \in (0,\nu-\tfrac{2}{3})$. For $\rho \in \{ \rho_1, \rho_1 + \rho_2\}$, \eqref{eq_concatenation_property_height_fct} implies
	\begin{equation} \begin{aligned}
	h^{\rho}((1-2\rho)t,t) = & h^{\rho}((1-2\rho)t^\nu,t^\nu) \\
	&+ \min_{x \in \Z} \{ h^\rho(x,t^\nu) - h^{\rho}((1-2\rho)t^\nu,t^\nu) + h^{\text{step}}_{x,t^\nu} ((1-2\rho)t,t) \}.
	\end{aligned}
	\end{equation}
	The occurring processes are coupled by basic coupling. We set
	\begin{equation}
	E_\rho = \{ |x(\tau) - (1-2\rho) \tau| \leq t^{2/3+\e} \text{ for all } \tau \in [0,t]\},
	\end{equation}
	where $x(\tau)$ is a backwards path with respect to $h^\rho$ starting at $x(t) = (1-2\rho)t$. Corollary~\ref{cor_localisation_backwards_geodesics_stationary_TASEP} yields $ \Pb(E_\rho) \geq 1 - C e^{-c t^{2\e}}$. Conditioned on $E_\rho$, we have
	\begin{equation} \begin{aligned}  \label{eq_pf_thm_decoupling_stationary_0}
	h^{\rho}((1-2\rho)t,t) = & h^{\rho}((1-2\rho)t^\nu,t^\nu) \\
	&+ \min_{x \in I_\rho} \{ h^\rho(x,t^\nu) - h^{\rho}((1-2\rho)t^\nu,t^\nu) + h^{\text{step}}_{x,t^\nu} ((1-2\rho)t,t) \}
	\end{aligned} \end{equation}
	for $I_\rho = \{(1-2\rho)t^\nu - t^{2/3+\e},\dots , (1-2\rho)t^\nu + t^{2/3+\e}\}$.
	
	Firstly, we argue that the values $(h^{\text{step}}_{x,t^\nu} ((1-2\rho)t,t), x \in I_\rho)$ are asymptotically independent for $\rho = \rho_1 $ and $\rho= \rho_1 + \rho_2$.
	
	For different $x$, the initial conditions of $h^{\text{step}}_{x,t^\nu}$ are ordered. By attractiveness of TASEP, see also Lemma 3.7 of~\cite{FN24}, we obtain the following comparison of the backwards paths: for $x \in I_\rho$, each path $x^{\text{step},x}(\tau)$ of $h^{\text{step}}_{x,t^\nu}$ starting at $(1-2\rho)t$ at time $t$ is enclosed by the leftmost backwards path of $h^{\text{step}}_{(1-2\rho)t^\nu - t^{2/3+\e},t^\nu}$ and the rightmost backwards path of $h^{\text{step}}_{(1-2\rho)t^\nu + t^{2/3+\e},t^\nu}$ starting at the same position. Thus, by the localisation of backwards paths in a TASEP with step initial condition~\cite{BF22,FN24}, the event
	\begin{equation}
	\begin{aligned}
	F_\rho = \{ |x^{\text{step},x} (\tau) - (1-2\rho) \tau | \leq 2 t^{2/3+\e} \text{ for all } \tau \in [t^\nu,t], x \in I_\rho\}
	\end{aligned}
	\end{equation}
	has probability $\Pb(F_\rho) \geq 1 - C e^{-c t^{2\e}}$. Conditioned on $F_{\rho_1} \cap F_{\rho_1+\rho_2}$, the respective backwards paths are contained in disjoint deterministic regions for large $t$. Then, $(h^{\text{step}}_{x,t^\nu} ((1-2\rho)t,t), x \in I_\rho)$ are independent for $\rho = \rho_1 $ and $\rho= \rho_1 + \rho_2$.
	
	Secondly, we argue that the differences $(h^\rho(x,t^\nu) - h^{\rho}((1-2\rho)t^\nu,t^\nu), x \in I_\rho)$ are asymptotically independent for $\rho=\rho_1$ and $\rho=\rho_1 + \rho_2$.
	
	By \eqref{eq_def_height_fct}, $(h^\rho(x,t^\nu) - h^{\rho}((1-2\rho)t^\nu,t^\nu), x \in I_\rho)$ is characterised by $(\eta^\rho_{t^\nu}(i), i \in I_\rho)$.
	We set
	\begin{equation}
	G = \{ \exists i \in \{(1-2(\rho_1+\rho_2)) t^\nu + t^{2/3+\e}+1, \dots, (1-2\rho_1)t^\nu - t^{2/3+\e} -1 \} : \ \eta_{t^\nu}(i) = 2\}.
	\end{equation}
	Then, Lemma~\ref{lemma_second_class_particle_in_interval_with_high_prob} implies $\Pb(G) \geq 1 - C e^ {-c t^{\nu}} $ for $t$ large enough. By Lemma~\ref{lemma_stationary_measure_independence_second_class_p}, conditioned on $G$, $(\eta_{t^\nu}(i), i \in I_\rho )$ are independent for $\rho=\rho_1$ and $\rho = \rho_1 + \rho_2$. This passes down to the marginals, such that $(h^\rho(x,t^\nu) - h^{\rho}((1-2\rho)t^\nu,t^\nu), x \in I_\rho)$ are asymptotically independent for $\rho=\rho_1$ and $\rho=\rho_1 + \rho_2$.
	
	Lastly, we observe that
	\begin{equation}  \label{eq_proof_thm_asympt_decoupl_stationary_0}
	\frac{h^{\rho}((1-2\rho)t^\nu,t^\nu) - (1-2\chi)t^\nu}{-2\chi^{2/3} t^{1/3}}
	\end{equation}
	weakly converges to zero since the height fluctuations are of order $\mathcal{O}(t^{\nu/3})$ with $\nu < 1$. We define $\hat{\mathfrak{h}}^\rho(0,t)$ as
	\begin{equation}
	\frac{\min_{x \in I_\rho} \{ h^\rho(x,t^\nu) - h^{\rho}((1-2\rho)t^\nu,t^\nu) + h^{\text{step}}_{x,t^\nu} ((1-2\rho)t,t) \} - (1-2\chi)(t-t^\nu)}{-2\chi^{2/3} t^{1/3}}.
	\end{equation}
	Then, this implies for any $\delta > 0$ and $A^\delta_\rho= \{ |\mathfrak{h}^\rho(0,t) - \hat{\mathfrak{h}}^\rho(0,t) | \leq \delta \} \cap E_\rho$ that
	$ \lim_{t \to \infty} \Pb(A^\delta_\rho) =1 $.
	
	Putting everything together, we deduce by conditioning on $G \cap A^\delta_\rho \cap F_\rho$ for $\rho=\rho_1$ and $\rho=\rho_1+\rho_2$ that
	\begin{equation}
	\begin{aligned}
	&\lim_{t \to  \infty} \Pb ( \mathfrak{h}^{\rho_1}(0,t) \leq s-2\delta) \Pb ( \mathfrak{h}^{\rho_1+\rho_2}(0,t) \leq r -2\delta) \\
	&\leq \lim_{t \to  \infty} \Pb ( \mathfrak{h}^{\rho_1}(0,t) \leq s, \mathfrak{h}^{\rho_1+\rho_2}(0,t) \leq r ) \\
	&\leq \lim_{t \to  \infty} \Pb ( \mathfrak{h}^{\rho_1}(0,t) \leq s+2\delta) \Pb ( \mathfrak{h}^{\rho_1+\rho_2}(0,t) \leq r +2\delta).
	\end{aligned}
	\end{equation}
	By \eqref{eq_conv_to_baik_rains_distr} and since the Baik-Rains distribution is continuous~\cite{BR00}, we can take $\delta \to 0$ in both bounds and obtain \eqref{eq_thm_asymptotic_decoupling_height_functions_stationary}.
\end{proof}

\subsection{Decoupling in a periodic two-species TASEP} \label{section_asymptotic_decoupling_periodic}

For the proof of Theorem~\ref{thm_asymptotic_decoupling_height_functions_flat}, we follow a similar line of reasoning as in the proof of Theorem~\ref{thm_asymptotic_decoupling_height_functions_stationary}. Previously, we established a decoupling of the differences of height function values at time $t^\nu$ in \eqref{eq_pf_thm_decoupling_stationary_0} by utilising properties of the stationary measure $\mu^{\rho_1,\rho_2}$. In the periodic setting, we use one-point estimates instead and replace the differences by values of the initial deterministic height profile.

\begin{proof}[Proof of Theorem~\ref{thm_asymptotic_decoupling_height_functions_flat}]
	We let $\rho \in \{\rho_1,\rho_1+\rho_2\}$ and choose $\nu \in (\tfrac{2}{3},1)$, $\e \in (0,\nu-\tfrac{2}{3})$. By Corollary~\ref{cor_localisation_geodesics_deterministic_IC_half_periodic}, for any backwards path starting at $x(t) = (1-2\rho)t$,
	\begin{equation}
	E_\rho = \{ |x(\tau) - (1-2\rho)\tau| \leq t^{2/3+\e} \text{ for all } \tau \in [0,t]\}
	\end{equation}
	fulfils $\Pb(E_\rho) \geq 1 - C e^{-c t^{2\e}}$. We set $I_\rho = \{ (1-2\rho)t^\nu - t^{2/3+\e}, \dots,  (1-2\rho)t^\nu + t^{2/3+\e}\}$. Under basic coupling and conditioned on $E_\rho$, we have
	\begin{equation} \label{eq_proof_thm_asympt_decoupl_flat_0} \begin{aligned}
	h^{\rho}((1-2\rho)t,t) = & h^{\rho}((1-2\rho)t^\nu,t^\nu) \\
	&+ \min_{x \in I_\rho} \{ h^\rho(x,t^\nu) - h^{\rho}((1-2\rho)t^\nu,t^\nu) + h^{\text{step}}_{x,t^\nu} ((1-2\rho)t,t) \}.
	\end{aligned} \end{equation}
	
	As in the proof of Theorem~\ref{thm_asymptotic_decoupling_height_functions_stationary}, there exist sets $F_\rho$ with $\Pb(F_\rho) \geq 1 - C e^{-c t^{2\e}}$ such that conditioned on $F_{\rho_1} \cap F_{\rho_1+\rho_2}$ and for $t$ large enough, $(h^{\text{step}}_{x,t^\nu}((1-2\rho)t,t), x \in I_\rho)$ are independent for $\rho=\rho_1$ and $\rho= \rho_1+\rho_2$.
	
	Next, we study $(h^\rho(x,t^\nu) - h^{\rho}((1-2\rho)t^\nu,t^\nu))$.
	For a fixed $x$, the conservation law implies $h^\rho(x,t^\nu) - h^\rho(x-(1-2\rho)t^\nu,0) \overset{(d)}{=} h^\rho((1-2\rho)t^\nu,t^\nu)$ up to a uniformly bounded shift, due to the fact that the initial condition is not completely translation-invariant. By one-point estimates\footnote{Since the right hand sides in \eqref{eq_proof_thm_asympt_decoupl_flat_1} grow faster than $t^{\nu/3}$, the required estimates can be derived by \eqref{eq_concatenation_property_height_fct} and the estimates from Proposition~A.9 of~\cite{BF22} for TASEP height functions with step initial condition. In the setting of particle positions, this has been done in the proof of Proposition~2.4 of~\cite{FN19}. Refer to the proof of Lemma~\ref{lemma_general_IC_translation} for random initial conditions.}, we find
	\begin{equation}  \label{eq_proof_thm_asympt_decoupl_flat_1}
	\begin{aligned}
	& \Pb ( \exists x \in I_\rho: \ |h^\rho(x,t^\nu) - h^\rho((1-2\rho)t^\nu,t^\nu) - h^\rho(x-(1-2\rho)t^\nu,0)| > \delta t^{1/3}) \\
	& \leq \sum\nolimits_{x \in I_\rho} \Pb ( |h^\rho(x,t^\nu) - h^\rho(x-(1-2\rho)t^\nu,0) - (1-2\chi)t^\nu| > \tfrac{\delta}{2} t^{1/3}) \\
	& \hphantom{\leq \sum\nolimits_{x \in I_\rho} } + \Pb( |h^\rho((1-2\rho)t^\nu,t^\nu) - (1-2\chi)t^\nu| > \tfrac{\delta}{2} t^{1/3}) \\
	& \leq C e^{-ct^{(1-\nu)/3}}
	\end{aligned}
	\end{equation}
	for all $t$ large enough and any fixed $\delta > 0$.
	We define $ \hat{\mathfrak{h}}^\rho(0,t)$ as
	\begin{equation}
	\frac{ \min_{x \in I_\rho} \{ h^\rho(x-(1-2\rho)t^\nu,0) + h^{\text{step}}_{x,t^\nu}((1-2\rho)t,t)\} - (1-2\chi)(t-t^\nu)}{-2\chi^{2/3} t^{1/3}}
	\end{equation}
	and $A^\delta_\rho= \{ |\mathfrak{h}^\rho(0,t) - \hat{\mathfrak{h}}^\rho(0,t)| \leq \delta \} \cap E_\rho$. By \eqref{eq_proof_thm_asympt_decoupl_flat_0}, \eqref{eq_proof_thm_asympt_decoupl_flat_1} and since the fluctuations of $h^{\rho}((1-2\rho)t^\nu,t^\nu)$ vanish under $\mathcal{O}(t^{1/3})$-scaling, we get $\lim_{t \to \infty} \Pb (A^\delta_\rho) = 1$. The remaining arguments are as in the proof of Theorem~\ref{thm_asymptotic_decoupling_height_functions_stationary}. Here, $(h^\rho(x-(1-2\rho)t^\nu,0))$ are deterministic. Further, also $F_{\text{GOE}}$ is a continuous distribution function~\cite{TW96}.
\end{proof}

\section{Two-point function} \label{section_mixed_correlations}

\subsection{Formula for the sum of mixed correlations} \label{section_formula_mixed_correlations}

Before showing the decay of mixed correlations for the stationary measure $\mu^{\rho_1,\rho_2}$ of the two-species TASEP, we prove the formula for the sum of the correlation functions, Proposition~\ref{prop_mixed_correlations_formula}. In doing so, we generalise the arguments from~\cite{PS01}.

\begin{proof}[Proof of Proposition~\ref{prop_mixed_correlations_formula}]

We denote $\eta=\eta^{\rho_1}$, $\tilde{\eta} = \eta^{\rho_1+\rho_2}$ and, accordingly, $\rho=\rho_1$, $\tilde{\rho} = \rho_1+\rho_2$. We set $j=x+i$, $\tilde{j}=\tilde{x}+i$ and define the height functions associated with the configurations $\eta_t$, $\tilde\eta_t$ as
$h(j,t)=h^{\rho_1}(j,t)$, $\tilde h(\tilde j,\tilde t)=h^{\rho_1+\rho_2}(\tilde j,\tilde t)$. Further, $N_t = N_t^{\rho_1}$, $\tilde{N}_{\tilde{t}} = N_{\tilde{t}}^{\rho_1+\rho_2}$ denote the total current of particles from site $0$ to site $1$ for $\eta_t$ or $\tilde\eta_t$, respectively.

It holds
\begin{equation}
\Delta \text{Cov}(h(j,t), \tilde{h}(\tilde j,\tilde{t})) = \Delta \E [h(j,t) \tilde{h}(\tilde j,\tilde{t})] - \Delta \E [h(j,t)] \E [ \tilde{h}(\tilde j,\tilde{t})],
\end{equation}
where
$\Delta f(i)=f(i+1)-2 f(i)+f(i-1)$ is the discrete Laplacian.

Setting $\chi = \rho(1-\rho)$ and $\tilde{\chi} = \tilde{\rho}(1-\tilde{\rho})$, we find
\begin{equation}
\begin{aligned}
\Delta  \E [h(j,t)] \E [ \tilde{h}(\tilde j,\tilde{t})]
= & (2 \chi t + (1-2\rho)(j+1))(2 \tilde{\chi} \tilde{t} + (1-2\tilde{\rho})(\tilde{j}+1))  \\
& - 2 (2 \chi t + (1-2\rho)j)(2 \tilde{\chi} \tilde{t} + (1-2\tilde{\rho})\tilde{j}) \\
& + (2 \chi t + (1-2\rho)(j-1))(2 \tilde{\chi} \tilde{t} + (1-2\tilde{\rho})(\tilde{j}-1)) \\
= & 2  (1-2\rho)(1-2\tilde{\rho}).
\end{aligned}
\end{equation}

Further, we have
\begin{equation} \begin{aligned}
&h(j+1,t)\tilde{h}(\tilde{j}+1,\tilde{t}) = (h(j,t) + (1-2\eta_t(j+1)))(\tilde{h}(\tilde{j},\tilde{t}) + (1-2\tilde{\eta}_{\tilde{t}}(\tilde{j}+1))), \\
&h(j-1,t)\tilde{h}(\tilde{j}-1,\tilde{t}) = (h(j,t) - (1-2\eta_t(j)))(\tilde{h}(\tilde{j},\tilde{t}) - (1-2\tilde{\eta}_{\tilde{t}}(\tilde{j}))).
\end{aligned} \end{equation}
This implies
\begin{equation} \label{eq_pf_cov_formula_mixed_correlations_1}
\begin{aligned}
 \Delta \E [h(j,t) \tilde{h}(\tilde{j},\tilde{t})] = &\E [ 2(\eta_t(j) - \eta_t(j+1)) \tilde{h}(\tilde{j},\tilde{t}) + 2(\tilde{\eta}_{\tilde{t}}(\tilde{j}) - \tilde{\eta}_{\tilde{t}}(\tilde{j}+1))h(j,t) \\
& +  (1-2\eta_t(j+1))(1-2\tilde{\eta}_{\tilde{t}}(\tilde{j}+1)) + (1-2\eta_t(j))(1-2\tilde{\eta}_{\tilde{t}}(\tilde{j}))] .
\end{aligned}
\end{equation}
We decompose the last term as
\begin{equation} \label{eq_pf_cov_formula_mixed_correlations_0}
\begin{aligned}
& 4 \E [(\eta_t(j) - \eta_t(j+1)) \tilde{N}_{\tilde{t}} + (\tilde{\eta}_{\tilde{t}}(\tilde{j}) - \tilde{\eta}_{\tilde{t}}(\tilde{j}+1)) N_t] \\
& + \E [ 2(\eta_t(j) - \eta_t(j+1)) ( \tilde{h}(\tilde{j},\tilde{t}) - 2 \tilde{N}_{\tilde{t}}) + 2(\tilde{\eta}_{\tilde{t}}(\tilde{j}) - \tilde{\eta}_{\tilde{t}}(\tilde{j}+1))(h(j,t) - 2 N_t) \\
&\hphantom{  +\E [} +  (1-2\eta_t(j+1))(1-2\tilde{\eta}_{\tilde{t}}(\tilde{j}+1)) + (1-2\eta_t(j))(1-2\tilde{\eta}_{\tilde{t}}(\tilde{j}))] .
\end{aligned}
\end{equation}
We denote by $N_t^-$ and $\tilde{N}_{\tilde{t}}^-$ the numbers of jumps from $-1$ to $0$ up to time $t$ or $\tilde{t}$.
By translation invariance, it holds $\eta_t(j+1)\tilde{N}_{\tilde{t}} \overset{(d)}{=} \eta_t(j) \tilde{N}_{\tilde{t}}^-$ and $\tilde{\eta}_{\tilde{t}}(\tilde{j}+1)N_t \overset{(d)}{=} \tilde{\eta}_{\tilde{t}}(\tilde{j})N_t^-$. Further, the conservation law yields
$N_t - N_t^- = \eta_0(0) - \eta_t(0)$ and $\tilde{N}_{\tilde{t}} - \tilde{N}_{\tilde{t}}^- = \tilde{\eta}_0(0) - \tilde{\eta}_{\tilde{t}}(0)$. We deduce
\begin{equation} \label{eq_pf_cov_formula_mixed_correlations_2}
\begin{aligned}
&  4 \E [(\eta_t(j) - \eta_t(j+1)) \tilde{N}_{\tilde{t}} + (\tilde{\eta}_{\tilde{t}}(\tilde{j}) - \tilde{\eta}_{\tilde{t}}(\tilde{j}+1)) N_t] \\
&= 4 \E [\eta_t(j)(\tilde{\eta}_0(0) - \tilde{\eta}_{\tilde{t}}(0)) + \tilde{\eta}_{\tilde{t}}(\tilde{j}) (\eta_0(0) - \eta_t(0))].
\end{aligned}
\end{equation}
Regarding the second expectation in \eqref{eq_pf_cov_formula_mixed_correlations_0}, we observe
\begin{equation}  \label{eq_pf_cov_formula_mixed_correlations_3}  \begin{aligned}
& 2 (\eta_t(j) - \eta_t(j+1))(\tilde{h}(\tilde{j},\tilde{t})-2\tilde{N}_{\tilde{t}}) - 2 \eta_t(j) (1-2\tilde{\eta}_{\tilde{t}}(\tilde{j})) \\
&= \begin{cases} 2 \eta_t(j) \sum_{i=1}^{\tilde{j}-1} (1-2\tilde{\eta}_{\tilde{t}}(i)) - 2 \eta_t(j+1) \sum_{i=1}^{\tilde{j}} (1-2\tilde{\eta}_{\tilde{t}}(i)), & \tilde{j} \geq 1, \\
- 2 \eta_t(j)(1-2\tilde{\eta}_{\tilde{t}}(0)), & \tilde{j} = 0, \\
- 2 \eta_t(j) \sum_{i=\tilde{j}}^0(1-2\tilde{\eta}_{\tilde{t}}(i)) + 2 \eta_t(j+1) \sum_{i=\tilde{j}+1}^0(1-2\tilde{\eta}_{\tilde{t}}(i)), & \tilde{j} \leq -1.
\end{cases}
\end{aligned}
\end{equation}
Translation invariance applied to the expression \eqref{eq_pf_cov_formula_mixed_correlations_3} implies
\begin{equation}
\E [2 (\eta_t(j) - \eta_t(j+1))(\tilde{h}(\tilde{j},\tilde{t})-2\tilde{N}_{\tilde{t}}) - 2 \eta_t(j) (1-2\tilde{\eta}_{\tilde{t}}(\tilde{j}))] = -2\E[\eta_t(j) (1-2\tilde{\eta}_{\tilde{t}}(0))]
\end{equation}
as the two sums are equal in law except for one remaining term. Similarly, we obtain
\begin{equation}
\E [ 2(\tilde{\eta}_{\tilde{t}}(\tilde{j}) - \tilde{\eta}_{\tilde{t}}(\tilde{j}+1))(h(j,t) - 2 N_t) - 2 \tilde{\eta}_{\tilde{t}}(\tilde{j}+1)(1-2\eta_t(j+1))] = -2 \E [\tilde{\eta}_{\tilde{t}}(\tilde{j}) (1-2\eta_t(0))].
\end{equation}
Therefore, the second summand in \eqref{eq_pf_cov_formula_mixed_correlations_0} equals
\begin{equation}
-2\E[\eta_t(j) (1-2\tilde{\eta}_{\tilde{t}}(0))] -2 \E [\tilde{\eta}_{\tilde{t}}(\tilde{j}) (1-2\eta_t(0))] +2-2\rho-2\tilde{\rho}
\end{equation}
and \eqref{eq_pf_cov_formula_mixed_correlations_1} and \eqref{eq_pf_cov_formula_mixed_correlations_2} yield
\begin{equation}
\begin{aligned}
 \Delta \E [h(j,t) \tilde{h}(\tilde{j},\tilde{t})]
&= 4 \E [ \eta_t(j)\tilde{\eta}_0(0) + \tilde{\eta}_{\tilde{t}}(\tilde{j}) \eta_0(0) ] + 2 - 4 \rho - 4 \tilde{\rho}.
\end{aligned}
\end{equation}
We conclude
\begin{equation} \begin{aligned}
\Delta \text{Cov}(h(j,t), \tilde{h}(\tilde{j},\tilde{t}))
= & 4 \E [ \eta_t(j)\tilde{\eta}_0(0) + \tilde{\eta}_{\tilde{t}}(\tilde{j}) \eta_0(0) ] - 8 \rho \tilde{\rho}.
\end{aligned}
\end{equation}
\end{proof}

\subsection{Decay of mixed correlations for TASEP} \label{section_decay_correlations}

In this section, we prove Theorem~\ref{thm_decay_corr} and Corollary~\ref{corollary_decay_corr}. We compared our methods to those in the proof of Corollary~2.16 of~\cite{ACH24} in Remark~\ref{remark_comparison_ACH24}.

Corollary~\ref{corollary_decay_corr} could be proven by similar means as Theorem~\ref{thm_decay_corr}. Nonetheless, the proofs are not completely analogous because by exchanging $\mathfrak{h}^{\rho_1+\rho_2}$ and $\mathfrak{h}^{\rho_1}$ in \eqref{eq_pf_prop_decay_corr_3}, we now consider the process $\delta_{\mathfrak{h}}(y)$ defined below for all $y \in \Z$. As explained in~\cite{ACH24}, this can be dealt with by localising the endpoint of a corresponding backwards path, which minimises the expression.

Instead of following this approach, we observe that Theorem~\ref{thm_decay_corr} already implies Corollary~\ref{corollary_decay_corr} by the particle-hole duality.

\begin{proof}[Proof of Corollary~\ref{corollary_decay_corr}]
We write
\begin{equation} \begin{aligned}
& S^\#_{1,2}((1-2\rho_1)t+w t^{2/3},t) \\
&= \langle (1-\eta_t^{0+\rho_2}((1-2\rho_1)t+w t^{2/3}))(1-\eta_0^0(0))\rangle - \rho_1(\rho_1+\rho_2) \\
&=  \langle \eta_t^{0+\rho_2}(-(1-2(1-\rho_1))t+w t^{2/3}) \eta_0^0(0) \rangle - (1-\rho_1)(1-\rho_1-\rho_2),
\end{aligned}
\end{equation}
where $\eta_t^{0+\rho_2}$ denotes the configuration of holes and second class particles and $\eta_t^0$ is the configuration of holes. By the particle-hole duality, they are again marginals of a two-species TASEP, now with inverted colours and leftward jumps: we exchange the roles of holes and first class particles, while second class particles stay the same. Since this process has the same generator as the former two-species TASEP with rightward jumps, the (unique translation-invariant) stationary measure remains the same. Further, the stationary measure of a coloured TASEP with jumps to the left is the reflection of the stationary measure of a coloured TASEP with the same densities and rightward jumps, see Theorem~4.1 of~\cite{BSS23}. Therefore, the term above equals
\begin{equation}
S^\#_{2,1}((1-2(1-\rho_1))t-w t^{2/3},t)
\end{equation}
with density $1-\rho_1$ of all particles and density $1-\rho_1-\rho_2$ of first class particles. By this means, we can use Theorem~\ref{thm_decay_corr} to obtain the claimed convergence.
\end{proof}

\begin{remark} \label{remark_alternative_queueing_construction}
The arguments in the proof of Corollary~\ref{corollary_decay_corr} suggest an alternative construction of the stationary measure $\mu^{\rho_1,\rho_2}$, see also~\cite{Ang06}. Reflecting the queueing construction from Section~\ref{section_queueing_representation}, we consider Bernoulli processes $\hat{\mathcal{A}}$ with rate $1-\rho_1-\rho_2$ and $\hat{\mathcal{S}}$ with rate $1-\rho_1$. Due to the reversed jump direction, we assign each point in $\hat{\mathcal{A}}$ to the nearest unused point to its right in $\hat{\mathcal{S}}$. By setting $\hat{\mathcal{A}} = 1-\mathcal{S}$ and $\hat{\mathcal{S}} = 1-\mathcal{A}$, where $\mathcal{A}$ and $\mathcal{S}$ are from Section~\ref{section_queueing_representation}, we obtain the following interpretation: we fix the particle positions in $\mathcal{A}$ as first class particles, and draw an edge from each hole in $\mathcal{S}$ to the next available hole in $\mathcal{A}$ at a position to its right. The remaining empty sites in $\mathcal{A}$ are interpreted as second class particles.

This second perspective is particularly useful for balancing the asymmetric nature of the queueing construction.
\end{remark}

The remainder of the section focuses on the proof of Theorem~\ref{thm_decay_corr}. For this purpose, we require an asymptotic independence statement for the height profiles $h^{\rho_1}(\cdot,0)$ and $h^{\rho_1+\rho_2}(\cdot,0)$. Similar results exist for the stationary horizon, see Lemma~E.6 of~\cite{ACH24}.

\begin{lem} \label{lemma_formula_for_IC_process_M}
Let $\hat{h}^{\rho_1}(x,0)$ be the initial height profile of a stationary TASEP with density $\rho_1$, constructed from the process $\mathcal{A}$ in Section~\ref{section_queueing_representation}. We rescale
\begin{equation}
\begin{aligned}
\mathfrak{h}^{\rho_1}(x) &=  t^{-1/3} (h^{\rho_1}(x,0)-(1-2\rho_1)x), \\
\mathfrak{h}^{\rho_1+\rho_2}(x) &=  t^{-1/3}(h^{\rho_1+\rho_2}(x,0)-(1-2(\rho_1+\rho_2))x), \\
\hat{\mathfrak{h}}^{\rho_1}(x) &=  t^{-1/3}(\hat{h}^{\rho_1}(x,0)-(1-2\rho_1)x).
\end{aligned}
\end{equation}
Then, $\mathfrak{h}^{\rho_1+\rho_2}$ is independent of $\hat{\mathfrak{h}}^{\rho_1}$. There exists a process $\delta_{\mathfrak{h}}$ such that
\begin{equation}
\mathfrak{h}^{\rho_1}(x) = \hat{\mathfrak{h}}^{\rho_1}(x)+\delta_{\mathfrak{h}}(x)
\end{equation}
and there exist constants $C,c > 0$ such that for all $R,K > 0$ and $t$ large enough, it holds
\begin{equation}
\Pb \bigg( \sup_{|x| \leq R t^{2/3}} |\delta_{\mathfrak{h}}(x)| > K t^{-1/3} \bigg) \leq C R t^{2/3} e^{-c K}.
\end{equation}
\end{lem}

\begin{proof}
The independence of $\mathfrak{h}^{\rho_1+\rho_2}$ and $\hat{\mathfrak{h}}^{\rho_1}$ is due to the construction of the stationary measure in Section~\ref{section_queueing_representation}. Lemma~\ref{lemma_formula_for_IC_process_sup} below yields
\begin{equation}
\mathfrak{h}^{\rho_1}(x) = \hat{\mathfrak{h}}^{\rho_1}(x) + \delta_{\mathfrak{h}}(x)
\end{equation}
for
\begin{equation} \label{eq_pf_lemma_formula_for_IC_process_M-1}
\begin{aligned}
\delta_{\mathfrak{h}}(x) \coloneqq & - \sup_{j \geq x} \{ \mathfrak{h}^{\rho_1+\rho_2}(j)-\mathfrak{h}^{\rho_1+\rho_2}(x)-(\hat{\mathfrak{h}}^{\rho_1}(j)-\hat{\mathfrak{h}}^{\rho_1}(x))-2\rho_2(j-x)t^{-1/3} \} \\
& + \sup_{j \geq 0} \{ \mathfrak{h}^{\rho_1+\rho_2}(j)-\hat{\mathfrak{h}}^{\rho_1}(j)-2\rho_2 j t^{-1/3} \}.
\end{aligned}
\end{equation}
We bound the upper tail of the first supremum. It holds
\begin{equation}
\begin{aligned}
& \sup_{j \geq x} \{ \mathfrak{h}^{\rho_1+\rho_2}(j)-\mathfrak{h}^{\rho_1+\rho_2}(x)-(\hat{\mathfrak{h}}^{\rho_1}(j)-\hat{\mathfrak{h}}^{\rho_1}(x))-2\rho_2(j-x)t^{-1/3} \} \\
& = \sup_{j \geq x} \left \{ 2 t^{-1/3} \sum\nolimits_{i=x+1}^j(Z_i-Y_i) \right \},
\end{aligned}
\end{equation}
where $Z_i = a(i) \sim \text{Ber}(\rho_1)$ and $Y_i = s(i) \sim \text{Ber}(\rho_1+\rho_2)$ are independent. As this is the total supremum of a random walk with negative drift $\E [Z_i - Y_i ] = -\rho_2 < 0$, we obtain by Lundberg's inequality, see Section~5 in Chapter~XIII of~\cite{Asm03}:
\begin{equation} \label{eq_pf_lemma_formula_for_IC_process_M}
\Pb \left( \sup_{j \geq x} \left \{ \sum\nolimits_{i=x+1}^j(Z_i-Y_i) \right \} \geq K\right) \leq C e^{-c \theta K},
\end{equation}
where $\theta >0$ solves $\E[e^{\theta(Z_1-Y_1)}] =1$. Explicitly, we obtain
\begin{equation}
\theta = \text{ln}\left( \frac{\rho_1(1-\rho_1-\rho_2)+\rho_2}{\rho_1(1-\rho_1-\rho_2)}  \right) > 0.
\end{equation}
In particular, the bound does not depend on $x$ and~\eqref{eq_pf_lemma_formula_for_IC_process_M} implies
\begin{equation}
\begin{aligned}
&\Pb\left(\sup_{j \geq x} \{ \mathfrak{h}^{\rho_1+\rho_2}(j)-\mathfrak{h}^{\rho_1+\rho_2}(x)-(\hat{\mathfrak{h}}^{\rho_1}(j)-\hat{\mathfrak{h}}^{\rho_1}(x))-2\rho_2(j-x)t^{-1/3} \} > K t^{-1/3} \right) \\
& \leq C e^{-cK}.
\end{aligned}
\end{equation}
Since both suprema in~\eqref{eq_pf_lemma_formula_for_IC_process_M-1} are nonnegative, this yields the pointwise estimate
\begin{equation}  \label{eq_pf_lemma_formula_for_IC_process_M_0}
\Pb(|\delta_{\mathfrak{h}}(x)| > K t^{-1/3}) \leq C e^{-cK}.
\end{equation}
As the process is defined on the integers, we deduce
\begin{equation}
\Pb \bigg( \sup_{|x| \leq R t^{2/3}} |\delta_{\mathfrak{h}}(x)| > K t^{-1/3} \bigg)\leq C R t^{2/3} e^{-cK}.
\end{equation}
\end{proof}

\begin{lem} \label{lemma_formula_for_IC_process_sup} Let $\hat{h}^{\rho_1}$ be the height profile of a stationary TASEP with density $\rho_1$, constructed from the process $\mathcal{A}$ in Section~\ref{section_queueing_representation}. Then, for any $x \in \Z$, it holds
\begin{equation} \label{eq_lemma_formula_for_IC_process_sup}
\begin{aligned}
h^{\rho_1}(x,0) = \  & \hat{h}^{\rho_1}(x,0) + \sup_{j \geq 0} \{h^{\rho_1+\rho_2}(j,0)-\hat{h}^{\rho_1}(j,0)\} \\ & - \sup_{j \geq x} \{h^{\rho_1+\rho_2}(j,0)-h^{\rho_1+\rho_2}(x,0)-(\hat{h}^{\rho_1}(j,0)-\hat{h}^{\rho_1}(x,0))\}.
\end{aligned}
\end{equation}
\end{lem}

The proof of Lemma~\ref{lemma_formula_for_IC_process_sup} is similar to that of Lemma~4.4 of~\cite{BSS23}. However, the outcome differs slightly because, unlike in~\cite{BSS23}, we reflected the construction of the stationary measure from~\cite{FM07} in Section~\ref{section_queueing_representation} to obtain a process with jumps to the right instead of to the left. In~\cite{BSS23}, the authors first worked with a process with leftward jumps and performed the reflection afterwards.

A structure similar to \eqref{eq_lemma_formula_for_IC_process_sup} was identified for the stationary horizon, see~\cite{ACH24,Bus24,SS21}.

\begin{proof}[Proof of Lemma~\ref{lemma_formula_for_IC_process_sup}]
Suppose $x \geq 1$; the proof for $x \leq -1$ is analogous. We recall the construction of $\mu^{\rho_1,\rho_2}$ from Section~\ref{section_queueing_representation}. Setting $\mathcal{A}_{[i,i-1]} = 0$ and $\mathcal{S}_{[i,i-1]} = 0$, we denote the queue length at site $i \in \Z$ by
\begin{equation} \label{eq_5.27}
Q_i = \max \left \{ \sup_{j \geq i} \{ \mathcal{A}_{[i,j]}-\mathcal{S}_{[i,j]} \},0 \right \} = \sup_{j \geq i-1} \{ \mathcal{A}_{[i,j]}-\mathcal{S}_{[i,j]} \}.
\end{equation}
Induction on the length of the interval $[i,j] \cap \Z$ yields
\begin{equation}
\mathcal{D}_{[i,j]} = Q_{j+1}-Q_i+\mathcal{A}_{[i,j]}.
\end{equation}
Thus, it holds
\begin{equation} \label{eq_pf_lemma_formula_for_IC_process_sup}
h^{\rho_1}(x,0) = x - 2 \mathcal{D}_{[1,x]} = x - 2 \mathcal{A}_{[1,x]} - 2 Q_{x+1}+2Q_1.
\end{equation}
Together with $\hat{h}^{\rho_1}(j,0)-\hat{h}^{\rho_1}(i,0) = j-i-2 \mathcal{A}_{[i+1,j]} $ and $h^{\rho_1+\rho_2}(j,0)-h^{\rho_1+\rho_2}(i,0) = j-i-2\mathcal{S}_{[i+1,j]}$, \eqref{eq_pf_lemma_formula_for_IC_process_sup} implies \eqref{eq_lemma_formula_for_IC_process_sup}.
\end{proof}

We are ready to prove Theorem~\ref{thm_decay_corr}.

\begin{proof}[Proof of Theorem~\ref{thm_decay_corr}]
Suppose $\text{supp}(\phi) \subset [-L,L]$. Proposition~\ref{prop_mixed_correlations_formula} implies
\begin{equation}
\begin{aligned}
& S^\#_{2,1}((1-2(\rho_1+\rho_2))t+w t^{2/3},t) \\
&= \frac{1}{4} \Delta \text{Cov}(h^{\rho_1+\rho_2}((1-2(\rho_1+\rho_2))t+w t^{2/3},t), h^{\rho_1}(2Lt^{2/3}+w t^{2/3},0)) \\ & \hphantom{=}  - S^\#_{1,2}(2Lt^{2/3}+w t^{2/3},0),
\end{aligned}
\end{equation}
where the discrete Laplacian $\Delta$ is applied with respect to $w t^{2/3}$. We define
\begin{equation}
A \coloneqq \{ \exists \text{ second class particle in } \eta_0 \text{ at a site in }  \{1, \dots, Lt^{2/3}-1\} \}
\end{equation}
and observe
\begin{equation}
\begin{aligned}
| S^\#_{1,2}(2Lt^{2/3}+w t^{2/3},0) |
&\leq | \text{Cov}(\eta_0^{\rho_1}(2Lt^{2/3}+w t^{2/3}),\eta_0^{\rho_1+\rho_2}(0) | A )| + C \Pb(A^c)
\end{aligned}
\end{equation}
for all $w \in \text{supp}(\phi)$ and some constant $C > 0$.
By Lemma~\ref{lemma_stationary_measure_independence_second_class_p}, the covariance conditioned on $A$ equals $0$ because the values of the configuration $\eta_0$ are independent of each other, given a second class particle in between the sites under consideration. By Lemma~\ref{lemma_second_class_particle_in_interval_with_high_prob}, we derive
\begin{equation}
| S^\#_{1,2}(2Lt^{2/3}+w t^{2/3},0) | \leq C e^{-cLt^{2/3}}
\end{equation}
for all $w \in \text{supp}(\phi)$. In particular, it holds
\begin{equation}
\lim_{t \to \infty} \bigg| t^{-2/3} \sum_{w \in t^{-2/3} \Z} \phi(w) t^{2/3} S^\#_{1,2}(2Lt^{2/3}+w t^{2/3},0) \bigg| \leq \lim_{t \to \infty} C L \|\phi\|_{\infty} t^{2/3} e^{-cLt^{2/3}} = 0.
\end{equation}
Therefore, Theorem~\ref{thm_decay_corr} follows if we show that
\begin{equation} \label{eq_pf_prop_decay_corr_0}
t^{-2/3} \sum_{w \in t^{-2/3} \Z} \phi(w) \frac{t^{2/3}}{4} \Delta \text{Cov}(h^{\rho_1+\rho_2}((1-2(\rho_1+\rho_2))t+w t^{2/3},t), h^{\rho_1}(2Lt^{2/3}+w t^{2/3},0))
\end{equation}
converges to $0$ as $t \to \infty$.
Applying summation by parts twice, we rewrite \eqref{eq_pf_prop_decay_corr_0} as
\begin{equation} \label{eq_pf_prop_decay_corr_1}
t^{-2/3} \sum_{ w \in t^{-2/3} \Z} \frac{t^{4/3}}{4} \Delta \phi(w) \text{Cov} (\mathsf{h}^{\rho_1+\rho_2}(w,t), \mathsf{h}_L^{\rho_1}(w,0))
\end{equation}
with
\begin{equation}
\begin{aligned}
 \mathsf{h}^{\rho_1+\rho_2}(w,t) = \ & t^{-1/3}(h^{\rho_1+\rho_2}((1-2(\rho_1+\rho_2))t+w t^{2/3},t)  \\ &- (1-2(\rho_1+\rho_2)(1-\rho_1-\rho_2))t - (1-2(\rho_1+\rho_2))w t^{2/3}), \\
\mathsf{h}^{\rho_1}_L(w,0) = \ & t^{-1/3}(h^{\rho_1}(2Lt^{2/3}+w t^{2/3},0)-(1-2\rho_1)(2Lt^{2/3}+w t^{2/3})).
\end{aligned}
\end{equation}
Then,
$\lim_{t \to \infty} \tfrac{t^{4/3}}{4} \Delta \phi(w) = \tfrac{1}{4} \phi^{''}(w)$
implies
\begin{equation}
\limsup_{t \to \infty} |\eqref{eq_pf_prop_decay_corr_1}| \leq \frac{1}{4} \int_{[-L,L]}  |\phi^{''}(w) | d w \times  \limsup_{t \to \infty} \sup_{w \in [-L,L]} | \text{Cov} (\mathsf{h}^{\rho_1+\rho_2}(w,t), \mathsf{h}_L^{\rho_1}(w,0))|.
\end{equation}
As the integral is bounded by $2L \| \phi^{''} \|_{\infty}$, it suffices to show that as $t \to \infty$, $|\text{Cov} (\mathsf{h}^{\rho_1+\rho_2}(w,t), \mathsf{h}_L^{\rho_1}(w,0))|$ converges to zero uniformly in $w \in t^{-2/3} \Z \cap [-L,L]$.
By \eqref{eq_concatenation_property_height_fct}, the covariance equals
\begin{equation} \label{eq_pf_prop_decay_corr_3}
\begin{aligned}
\text{Cov}\left(\min_{y \in \Z} \{ \mathfrak{h}^{\rho_1+\rho_2}(y)+\mathfrak{h}^{\text{step}}_y(w,t) \}, \mathfrak{h}^{\rho_1}(2Lt^{2/3}+w t^{2/3})\right)
\end{aligned}
\end{equation}
with $\mathfrak{h}^{\rho_1+\rho_2},\mathfrak{h}^{\rho_1}$ defined in Lemma~\ref{lemma_formula_for_IC_process_M} and with
\begin{equation}
\begin{aligned}
\mathfrak{h}^{\text{step}}_y(w,t) = \ & t^{-1/3}(h_y^{\text{step}}((1-2(\rho_1+\rho_2))t+w t^{2/3},t) \\ &- (1-2(\rho_1+\rho_2)(1-\rho_1-\rho_2))t - (1-2(\rho_1+\rho_2))(w t^{2/3}-y)),
\end{aligned}
\end{equation}
where $h^{\text{step}}_y$ is the height function of a TASEP with step initial condition shifted by $y$.
Lemma~\ref{lemma_formula_for_IC_process_M} yields
\begin{equation}
\mathfrak{h}^{\rho_1}(2Lt^{2/3}+w t^{2/3}) = \hat{\mathfrak{h}}^{\rho_1}(2Lt^{2/3}+w t^{2/3}) + \delta_{\mathfrak{h}}(2Lt^{2/3}+w t^{2/3})
\end{equation}
and $\hat{\mathfrak{h}}^{\rho_1}$ is independent of $\mathfrak{h}^{\rho_1+\rho_2}$ and $\mathfrak{h}_y^{\text{step}}$, meaning
\begin{equation}
\eqref{eq_pf_prop_decay_corr_3} = \text{Cov}\left(\min_{y \in \Z} \{ \mathfrak{h}^{\rho_1+\rho_2}(y)+\mathfrak{h}^{\text{step}}_y(w,t) \}, \delta_{\mathfrak{h}}(2Lt^{2/3}+w t^{2/3})\right) .
\end{equation}
In particular, the Cauchy-Schwarz inequality implies
\begin{equation}
| \eqref{eq_pf_prop_decay_corr_3}| \leq (\E[\mathsf{h}^{\rho_1+\rho_2}(w,t)^2] \times \E [ \delta_{\mathfrak{h}}(2Lt^{2/3}+w t^{2/3})^2])^{1/2}.
\end{equation}
Since $\mathsf{h}^{\rho_1+\rho_2}$ is a stationary single-species TASEP, it follows from tail estimates that
\begin{equation}
\limsup_{t \to \infty} \sup_{w \in [-L,L]}  \E[\mathsf{h}^{\rho_1+\rho_2}(w,t)^2] < \infty.
\end{equation}
Indeed, by Theorem~1 of~\cite{BFP12}, the limit exists and is given in terms of the supremum of the second moments of Baik-Rains distributions with parameters in a compact set.

As $\delta_{\mathfrak{h}}(x) = \mathfrak{h}^{\rho_1}(x)-\hat{\mathfrak{h}}^{\rho_1}(x)$, we have $|\delta_{\mathfrak{h}}(x)| \leq C t^{-1/3} |x|$ for some constant $C > 0$ and Lemma~\ref{lemma_formula_for_IC_process_M} yields
\begin{equation}  \label{eq_pf_prop_decay_corr_4}
\begin{aligned}
 \E [ \delta_{\mathfrak{h}}(2Lt^{2/3}+w t^{2/3})^2] \leq  & \ t^{-1/2} + C L^2 t^{2/3} \Pb \bigg( \sup_{|x| \leq 3 L t^{2/3}} |\delta_{\mathfrak{h}}(x)| > t^{-1/4} \bigg) \\ \leq & \ t^{-1/2}+CL^3t^{4/3} e^{-ct^{1/12}}
\end{aligned}
\end{equation}
uniformly for $w \in t^{-2/3} \Z \cap [-L,L]$. We conclude $\lim_{t \to \infty} \eqref{eq_pf_prop_decay_corr_0} =0$.
\end{proof}

\begin{remark} \label{remark_general_speeds}
The proof of Theorem~\ref{thm_decay_corr} also applies to $t^{2/3} S^\#_{2,1}(vt+w t^{2/3},t)$ for speeds $v \neq 1-2(\rho_1+\rho_2)$. The key modifications are as follows. In \eqref{eq_pf_prop_decay_corr_1}, we replace $ \mathsf{h}^{\rho_1+\rho_2}(w,t)$ by $\mathsf{h}_v^{\rho_1+\rho_2}(w,t)+\mathsf{h}_v^{\rho_1+\rho_2}(0)$, defined as
\begin{equation} \begin{aligned}
&\frac{h^{\rho}(vt+w t^{2/3},t)-h^\rho((v-(1-2\rho))t,0)-(1-2\rho(1-\rho))t-(1-2\rho)w t^{2/3}}{t^{1/3}} \\
&+\frac{h^{\rho}((v-(1-2\rho))t,0)-(1-2\rho)(v-(1-2\rho))t}{t^{1/3}}
 \end{aligned} \end{equation}
 with $\rho=\rho_1+\rho_2$. By Cauchy-Schwarz, translation invariance of $\mu^{\rho_1,\rho_2}$, and the conservation law, $|\text{Cov}(\mathsf{h}_v^{\rho_1+\rho_2}(w,t),\mathsf{h}^{\rho_1}_L(w,0))| = |\text{Cov}(\mathsf{h}_v^{\rho_1+\rho_2}(w,t),\delta_{\mathfrak{h}}(2Lt^{2/3}+w t^{2/3}))|$ satisfies the same bound as $|\eqref{eq_pf_prop_decay_corr_3}|$. Also $|\text{Cov}(\mathsf{h}_v^{\rho_1+\rho_2}(0),\mathsf{h}^{\rho_1}_L(w,0))|^2$ is bounded by the product of the second moments of $\mathsf{h}_v^{\rho_1+\rho_2}(0)$ and $\delta_{\mathfrak{h}}(2Lt^{2/3}+w t^{2/3})$. By \eqref{eq_pf_prop_decay_corr_4}, this yields the bound $Ct^{1/3}t^{-1/2}$. Thus, both covariances converge to zero uniformly for $w \in t^{-2/3} \Z \cap [-L,L]$.

The convergence of $t^{2/3} S^\#_{1,2}(vt+w t^{2/3},t)$ to zero for $v \neq 1-2\rho_1$ follows by the same arguments as Corollary~\ref{corollary_decay_corr}.
\end{remark}

\subsection{Decay of mixed correlations for ASEP} \label{section_decay_correlations_ASEP}

In this section, we prove Corollary~\ref{cor_decay_corr_ASEP} by extending the arguments from Section~\ref{section_decay_correlations} to the more general case of ASEP. To do so, we first review the queueing construction of the translation-invariant stationary measure $\mu^{\rho_1,\rho_2,q}$ for the two-species ASEP with asymmetry parameter $q \in (0,1)$ and with densities $\rho_1 \in (0,1)$ and \mbox{$\rho_2 \in (0,1-\rho_1)$} of first and second class particles by \cite{Mar20}.

\subsubsection{Queueing representation} \label{section_queueing_representation_ASEP}

The queueing construction introduced by \cite{Mar20} extends the representation in the totally asymmetric case developed in \cite{FM07}, which we discussed in Section~\ref{section_queueing_representation}. The key difference is that a positive queue length does \emph{not} necessarily imply that a service results in a departure.

We again start with independent Bernoulli processes $\mathcal{A} \sim \text{Ber}(\rho_1)$ and \mbox{$\mathcal{S} \sim \text{Ber}(\rho_1+\rho_2)$} on $\Z$, representing arrival and service lines. The sites in $\Z$ are interpreted as time indices progressing from right to left. We denote the queue length at site $i \in \Z$ by $Q_i$. In the totally asymmetric case, its relation to $\mathcal{A}$ and $\mathcal{S}$ is given by \eqref{eq_5.27}. For an asymmetry parameter $q > 0$, we additionally consider a family of independent geometric\footnote{We use the geometric distribution with support $\N$ and parameter $1-q$.} random variables $(b(i),i \in \Z)$, where $\Pb(b(i)\leq k )=1-q^k$. The queue length is defined recursively by
\begin{equation}
Q_i = Q_{i+1} + \Id_{\{a(i)=1,s(i)=0\}} - \Id_{\{a(i)=0,s(i)=1,b(i)\leq Q_{i+1}\}}.
\end{equation}
This can be interpreted as follows: suppose $Q_{i+1}=k$. If there is an arrival but no service at site $i$, the queue length increases by $1$. If both an arrival and a service occur, then there is a departure, so the queue length remains unchanged. Crucially, if there is a service but no arrival, then each customer in the queue rejects the service independently with probability $q$. Consequently, in this case, the service results in a departure with probability $1-q^k$ and remains unused with probability $q^k$. See also Figure~\ref{figure_queueing_representation_ASEP}.

\begin{figure}[t!]
\centering
\includegraphics[scale=1]{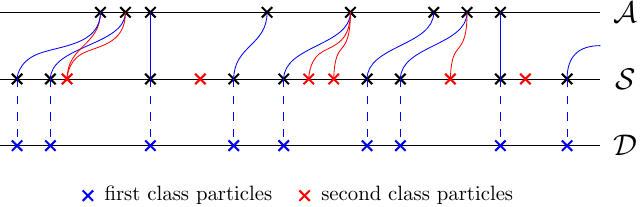}
\caption{Queueing construction of the measure $\mu^{\rho_1,\rho_2,q}$. The red lines correspond to rejections of services by customers in the queue.}
\label{figure_queueing_representation_ASEP}
\end{figure}

The departure process $\mathcal{D}$ satisfies
\begin{equation}
d(i) = \begin{cases}
1, & \text{if } a(i)=s(i)=1 \text{ or } a(i)=0,s(i)=1,b(i) \leq Q_{i+1}, \\ 0, & \text{otherwise},
\end{cases}
\end{equation}
and the resulting configuration $\eta$ is given by
\begin{equation} \label{eq_configuration_stationary_ASEP_2}
\eta(i) = \begin{cases}
1, & \text{if } a(i)=s(i)=1 \text{ or } a(i)=0,s(i)=1,b(i) \leq Q_{i+1}, \\
2, & \text{if } a(i)=0,s(i)=1 \text{ and } b(i) > Q_{i+1}, \\
+\infty, & \text{if } s(i)=0.
\end{cases}
\end{equation}
In \cite{Mar20}, it was observed that for $\rho_2 > 0$, the queue length $Q_i$ is positive recurrent and admits a unique stationary distribution. Assuming $Q_i$ is stationary, \cite{Mar20} proved that $\eta$ is distributed according to $\mu^{\rho_1,\rho_2,q}$.

\subsubsection{Properties of the stationary queue length process}
\label{section_properties_queue_length_ASEP}
We require properties of the stationary queue length $Q_i$ that relate to Lemma~\ref{lemma_second_class_particle_in_interval_with_high_prob} and Lemma~\ref{lemma_formula_for_IC_process_M}. In the totally asymmetric case, these properties were established through explicit computations. Due to the increased complexity of the queue length, we now rely on the more abstract framework of geometric ergodicity, elaborated for instance in \cite{MT93}.

The process $Q_i$ is irreducible, aperiodic, and positive recurrent.
A short computation, see Appendix~\ref{appendix_E}, shows that there exist a finite set $I$, constants $b < \infty, \beta > 0$, and $c>0$ such that $V(x)=e^{cx}$ satisfies
\begin{equation}
\E[V(Q_i)|Q_{i+1}=x]-V(x) \leq - \beta V(x) + b \Id_I(x).
\end{equation}
Theorem~14.3.7 of \cite{MT93} yields
\begin{equation}
\E[V(Q_i)] \leq C \E[\Id_I(Q_i)] \leq C
\end{equation}
for some constant $C > 0$. Having this, the Markov inequality, applied to $V(Q_i)$, implies exponential tail bounds for the queue length:
\begin{equation} \label{eq_tail_bound_queue_length}
\Pb(Q_i > n) = \Pb(V(Q_i) > V(n)) \leq \E[V(Q_i)] e^{-cn} \leq C e^{-cn},
\end{equation}
independently of $i$. Applying Theorem~15.2.6 of \cite{MT93} with $A=\{0\}$, we further deduce
\begin{equation}
\E[e^{c\tau_0} |Q_0=0] < \infty
\end{equation}
for some $c>0$, where $\tau_0 = \min\{n \geq 1: Q_{-n}=0\}$. Therefore, also the return time to $0$ satisfies the exponential bound
\begin{equation} \label{eq_bound_return_time}
\Pb(\tau_0 > t | Q_0=0) \leq C e^{-ct}.
\end{equation}
By translation invariance, the same holds for the return time to $0$ starting from any other site than the origin.

\subsubsection{Extending the proof of the decay of mixed correlations}
\label{section_extension_proof_ASEP}

Finally, we are able to extend the proof ideas from Section~\ref{section_decay_correlations} to the general asymmetric case.

Our first observation is that the proof of Corollary~\ref{corollary_decay_corr} transfers, as the particle-hole duality also applies to ASEP, and the stationary measure satisfies the same reflection property as in the totally asymmetric case (by similar arguments as in Remark~4.2 of \cite{BSS23}). Therefore, it suffices to consider the off-diagonal term $S^\#_{2,1}$ for arbitrary densities. This is crucial for ASEP because, unlike for TASEP, the proof idea for Theorem~\ref{thm_decay_corr} does not extend to $S^\#_{1,2}$ as explained in the beginning of Section~\ref{section_decay_correlations}; the concatenation property \eqref{eq_concatenation_property_height_fct} and the theory of backwards paths are not available.

Importantly, the starting formula in Proposition~\ref{prop_mixed_correlations_formula} holds for ASEP as well. Thus, to extend the proof of Theorem~\ref{thm_decay_corr} to ASEP, it suffices to show the following:
\begin{itemize}
		\item[(a)] It holds  $|S^\#_{1,2}(x,0)| = |\text{Cov}(\eta^{\rho_1}_0(x),\eta^{\rho_1+\rho_2}_0(0))| \leq C e^{-cx}$ for some constants $C,c>0$.
		\item[(b)] Lemma~\ref{lemma_formula_for_IC_process_M} applies to ASEP as well, when using the queueing construction from Section~\ref{section_queueing_representation_ASEP}.
		\item[(c)] In the large-time limit, the second moments of the rescaled stationary single-species ASEP height function are bounded uniformly for $w$ in a compact set.
\end{itemize}
We verify (a) and (b) using the properties of the queueing representation from Section~\ref{section_properties_queue_length_ASEP}. Known convergence and concentration results by \cite{Agg18,LS25} for the stationary single-species ASEP imply (c).

For general speeds, the arguments in Remark~\ref{remark_general_speeds} apply to ASEP as well.

\paragraph{(a)}
We can view the configurations $\eta_0^{\rho_1}$ and $\eta_0^{\rho_1+\rho_2}$ as marginals of the configuration $\eta \sim \mu^{\rho_1,\rho_2,q}$ constructed in Section~\ref{section_queueing_representation_ASEP}. The measure $\mu^{\rho_1,\rho_2,q}$ factorises similarly as $\mu^{\rho_1,\rho_2}$ in Lemma~\ref{lemma_stationary_measure_independence_second_class_p}: conditioned on $Q_z=0$ for some fixed site $z \in \Z$, the configuration values $(\eta(i),i<z)$ and $(\eta(i),i>z)$ are independent. This is a direct consequence of the queueing construction. By \eqref{eq_configuration_stationary_ASEP_2}, the values $(\eta(i),i<z)$ depend on $((a(i),s(i),b(i),Q_{i+1}), i < z)$, where $a(i),b(i),s(i)$ are independent for different sites and, given $Q_z=0$, $Q_{i+1}$ only depends on $a(i+1)$, $s(i+1)$ and $Q_{i+2}, \dots, Q_{z-1}$ for $i < z-1$.

Therefore, it suffices to show
\begin{equation}
\Pb(\exists z \in (0,x) \cap \Z: \ Q_z=0) \geq 1-Ce^{-cx}.
\end{equation}

As the queue length is positive recurrent, we have $Q_k=0$ for some $k \in \Z_{\geq x}$ with probability $1$. Thus, we confirm
\begin{equation}
\begin{aligned}
\Pb(\forall z \in (0,x) \cap \Z: Q_z>0) &= \Pb(\exists k \in \Z_{\geq x}: Q_k=0, \forall z \in (0,k-1] \cap \Z: Q_z>0) \\
& \leq \sum_{k=x}^\infty \Pb(\tau_0 \geq k | Q_0=0) \leq C e^{-cx},
\end{aligned}
\end{equation}
where we applied translation invariance and the bound \eqref{eq_bound_return_time}.

\paragraph{(b)} In the setting of Lemma~\ref{lemma_formula_for_IC_process_M}, we obtain similarly as in the proof of Lemma~\ref{lemma_formula_for_IC_process_sup} that
\begin{equation}
\mathfrak{h}^{\rho_1}(x) = \hat{\mathfrak{h}}^{\rho_1}(x) + \delta_{\mathfrak{h}}(x),
\end{equation}
where $\delta_{\mathfrak{h}}(x) = 2 t^{-1/3}(Q_1-Q_{x+1})$.
But then, \eqref{eq_tail_bound_queue_length} implies
\begin{equation}
\Pb(|\delta_{\mathfrak{h}}(x)| > K t^{-1/3}) \leq C e^{-cK}
\end{equation}
and the result of Lemma~\ref{lemma_formula_for_IC_process_M} follows by a union bound.

\paragraph{(c)}
Theorem~2.4 of \cite{LS25} provides tail estimates for the stationary single-species ASEP. Combining them with the convergence to the Baik-Rains distribution established by \cite{Agg18}, \cite{LS25} deduce that, by similar arguments as in \cite{BFP12}, the two-point function of the stationary single-species ASEP converges weakly to the KPZ-universal scaling limit (see Section~\ref{section_main_results_ASEP}). Precisely, this follows from the convergence of second moments of $\mathfrak{h}^{\rho}(w,t)$ in \eqref{eq_rescaled_height_fct_ASEP} to those of the Baik-Rains distribution, uniformly for $w$ in a compact set. Thus, the second moments of the rescaled height functions are, in particular, uniformly bounded in the large-time limit.

\section{Decoupling of height functions for general random initial data} \label{section_decoupling_random_IC}

In this section, we prove Theorem~\ref{thm_asymptotic_decoupling_height_functions_general} by extending the arguments from Section~\ref{section_asymptotic_decoupling}. Afterwards, we provide a class of random initial conditions that satisfies the assumptions of Theorem~\ref{thm_asymptotic_decoupling_height_functions_general} and contains the stationary measure $\mu^{\rho_1,\rho_2}$.

\begin{proof}[Proof of Theorem~\ref{thm_asymptotic_decoupling_height_functions_general}] Without loss of generality, we set $w=z = 0$. We let \mbox{$\rho \in \{\rho_1,\rho_1+\rho_2\}$} and $\nu \in (\tfrac{2}{3}+\e,1)$ for $\e \in (0,\tfrac{1}{3})$ from Assumption~\ref{assumption_decoupling_general_IC}(c). Further, we define
\begin{equation}
E_\rho=\{ |x(\tau)-(1-2\rho)\tau| \leq t^{2/3+\e} \text{ for all } \tau \in [0,t]\},
\end{equation}
where $x(\tau)$ is a backwards path with respect to $h^\rho$ starting at $x(t) = (1-2\rho)t$. Lemma~\ref{lemma_general_IC_backwards_path} yields $\lim_{t \to \infty} \Pb(E_\rho) = 1$. Conditioned on $E_\rho$, we have
\begin{equation} \begin{aligned}  \label{eq_pf_thm_decoupling_general_0}
h^{\rho}((1-2\rho)t,t) = & h^{\rho}((1-2\rho)t^\nu,t^\nu) \\
&+ \min_{x \in I_\rho} \{ h^\rho(x,t^\nu) - h^{\rho}((1-2\rho)t^\nu,t^\nu) + h^{\text{step}}_{x,t^\nu} ((1-2\rho)t,t) \}
\end{aligned} \end{equation}
with $I_\rho = \{(1-2\rho)t^\nu - t^{2/3+\e},\dots , (1-2\rho)t^\nu + t^{2/3+\e}\}$.
We wish to replace the height differences at time $t^\nu$ in \eqref{eq_pf_thm_decoupling_general_0} and set
\begin{equation}
G_\rho = \bigg\{ \sup_{|y| \leq t^{2/3+\e}} | h^{\rho}(y,0) - (h^{\rho}((1-2\rho)t^{\nu} +y,t^\nu)-h^{\rho}((1-2\rho)t^\nu,t^\nu))| \leq t^\theta \bigg\}
\end{equation} for some $\theta \in (\tfrac{\nu}{3},\tfrac{1}{3})$. Lemma~\ref{lemma_general_IC_translation} verifies $\lim_{t \to \infty} \Pb(G_\rho) = 1$. Further, by Assumption~\ref{assumption_decoupling_general_IC}(c),
\begin{equation}
H_{\rho_1} = \bigg\{\sup_{|y| \leq t^{2/3+\e}}|h^{\rho_1}(y,0)-\hat{h}^{\rho_1}(y,0)| \leq t^\sigma \bigg\}
\end{equation}
fulfils $\lim_{t \to \infty} \Pb(H_{\rho_1}) = 1$, with $\hat{h}^{\rho_1}(\cdot,0)$, $h^{\rho_1+\rho_2}(\cdot,0)$ independent, and $\sigma \in (0,\tfrac{1}{3})$.

Next, for $\rho = \rho_1+\rho_2$, we define $\tilde{\mathfrak{h}}^{\rho_1+\rho_2}(0,t)$ as
\begin{equation}
\frac{\min_{x \in I_\rho} \{h^{\rho}(x-(1-2\rho)t^\nu,0) + h^{\text{step}}_{x,t^\nu}((1-2\rho)t,t) \}-(1-2\chi)(t-t^\nu)}{-2\chi^{2/3} t^{1/3}},
\end{equation}
while for $\rho=\rho_1$, we define $\hat{\mathfrak{h}}^{\rho_1}(0,t)$ as
\begin{equation}
\frac{\min_{x \in I_\rho} \{\hat{h}^{\rho}(x-(1-2\rho)t^\nu,0) + h^{\text{step}}_{x,t^\nu}((1-2\rho)t,t) \}-(1-2\chi)(t-t^\nu)}{-2\chi^{2/3} t^{1/3}}.
\end{equation}
Given our observations above, it follows by the same reasoning as in the proofs of Theorem~\ref{thm_asymptotic_decoupling_height_functions_stationary} and Theorem~\ref{thm_asymptotic_decoupling_height_functions_flat} that for any $\delta > 0$, there exists an event $A^\delta$ with $\lim_{t \to \infty} \Pb(A^\delta) = 1$ such that, conditioned on $A^\delta$, $|\mathfrak{h}^{\rho_1+\rho_2}(0,t)-\tilde{\mathfrak{h}}^{\rho_1+\rho_2}(0,t)| \leq \delta$, $|\mathfrak{h}^{\rho_1}(0,t)-\hat{\mathfrak{h}}^{\rho_1}(0,t)| \leq \delta$, and $\tilde{\mathfrak{h}}^{\rho_1+\rho_2}(0,t)$ and $\hat{\mathfrak{h}}^{\rho_1}(0,t)$ are independent for large $t$. Exploiting the continuity of the limit distributions, the remaining steps proceed as in the previous proofs.
\end{proof}

One class of initial conditions satisfying Assumption~\ref{assumption_decoupling_general_IC} can be constructed using the queueing procedure from Section~\ref{section_queueing_representation} with more general processes $\mathcal{A}$ and $\mathcal{S}$ in $\{0,1\}^{\Z}$. We assume:
\begin{itemize}
\item[(A1)] Independence: $\mathcal{A}$ and $\mathcal{S}$ are independent.
\item[(A2)] Spatial homogeneity: the distributions of $\mathcal{A}$ and $\mathcal{S}$ are translation-invariant.
\item[(A3)] Tail bounds: uniformly for $t$ large enough, there exist constants $C,c > 0$ such that
\begin{equation} \begin{aligned}
&\Pb(|\mathcal{A}_{I_x} - \rho_1 |x|t^{2/3}| > s t^{1/3}) \leq C e^{-c s |x|^{-1/2}},  \\ &\Pb(|\mathcal{S}_{I_x} - (\rho_1+\rho_2) |x|t^{2/3}| > s t^{1/3}) \leq C e^{-c s |x|^{-1/2}}
\end{aligned} \end{equation}
for all $x \neq 0$, $s > 0$. We set $I_x = [1,xt^{2/3}] \cap \Z$ for $x > 0$ and $I_x = [xt^{2/3}+1,0] \cap \Z$ for $x<0$.
\item[(A4)] Weak convergence: it holds \begin{equation}
\frac{\mathcal{A}_{I_x}-\rho_1|x|t^{2/3}}{t^{1/3}} \Rightarrow \sigma_A \mathcal{B}_A(x),  \quad \frac{\mathcal{S}_{I_x}-(\rho_1+\rho_2)|x|t^{2/3}}{t^{1/3}} \Rightarrow \sigma_S \mathcal{B}_S(x),
\end{equation}
where $\sigma_A,\sigma_S > 0$, and $\mathcal{B}_A$, $\mathcal{B}_S$ are standard two-sided Brownian motions. The weak convergence is with respect to the uniform topology on compact sets.
\end{itemize}

As before, it suffices to have (A2) up to uniformly bounded perturbations. We construct the departure process $\mathcal{D}$ as in Section~\ref{section_queueing_representation}, and the initial configuration of the two-species TASEP is again characterised by
\begin{equation}
\eta(z) = \begin{cases}
1 & \text{ if } d(z) = 1, \\ 2  &\text{ if } s(z) = 1, d(z) = 0, \\ +\infty &\text{ else.}
\end{cases}
\end{equation}
The height profiles $h^{\rho_1+\rho_2}(\cdot,0)$ and $h^{\rho_1}(\cdot,0)$ are defined by $\mathcal{S}$ and $\mathcal{D}$ respectively. We denote by $\hat{h}^{\rho_1}(\cdot,0)$ the height profile with respect to $\mathcal{A}$. It holds $h^{\rho_1+\rho_2}(xt^{2/3},0) = xt^{2/3}-2\text{sgn}(x)\mathcal{S}_{I_x}$, and the same for ($h^{\rho_1},\mathcal{D}$) and ($\hat{h}^{\rho_1},\mathcal{A}$). Below, we verify that this initial condition satisfies Assumption~\ref{assumption_decoupling_general_IC}.

Assumption~\ref{assumption_decoupling_general_IC}(a) is ensured by (A2), as the queueing procedure preserves translation invariance. The tail bound for $h^{\rho_1+\rho_2}(xt^{2/3},0)$ in Assumption~\ref{assumption_decoupling_general_IC}(b) is equivalent to the bound for $\mathcal{S}_{I_x}$ in (A3). To obtain the bound for $h^{\rho_1}(xt^{2/3},0)$, we derive the corresponding result for $\mathcal{D}_{I_x}$.

We recall from the proof of Lemma~\ref{lemma_formula_for_IC_process_sup} that for any integers $i \leq j$, it holds $\mathcal{D}_{[i,j]} = Q_{j+1}-Q_i+\mathcal{A}_{[i,j]}$, where $Q_i$ denotes the queue length at site $i$. By (A2) and (A3), we have
\begin{equation} \label{eq_bound_on_queue_length}
\begin{aligned}
\Pb(Q_i > s t^{1/3}) = \ & \Pb\Bigl(\max \Bigl\{\sup_{j \geq i} \{\mathcal{A}_{[i,j]}-\mathcal{S}_{[i,j]}\},0\Bigr\} > s t^{1/3}\Bigr) \\
\leq & \sum\nolimits_{j \geq i+st^{1/3}} (\Pb( \mathcal{A}_{[i,j]} > (\rho_1+\tfrac{1}{2}\rho_2)(j-i+1)) \\
& \hphantom{\sum\nolimits_{j \geq i+st^{1/3}}} + \Pb( \mathcal{S}_{[i,j]} < (\rho_1+\tfrac{1}{2}\rho_2)(j-i+1))) \\
\leq & C e^{-c \sqrt{s}t^{1/6}}.
\end{aligned}
\end{equation}
Together with the bound on $\mathcal{A}_{I_x}$ in (A3), this implies
\begin{equation}
\begin{aligned}
&\Pb(|\mathcal{D}_{I_x}-\rho_1|x|t^{2/3}| > s t^{1/3}) \leq C e^{-cs|x|^{-1/2}}+Ce^{-c\sqrt{s}t^{1/6}}.
\end{aligned}
\end{equation}
It holds $C e^{-c \sqrt{s}t^{1/6}} \leq Ce^{-cs|x|^{-1/2}}$ for $s \leq  |x|t^{1/3}$ and the probability above equals $0$ for $s > \max\{1-\rho_1,\rho_1 \}|x|t^{1/3}$. This yields $\Pb(|\mathcal{D}_{I_x}-\rho_1|x|t^{2/3}| > s t^{1/3}) \leq  C e^{-cs|x|^{-1/2}}$, and the bound for $h^{\rho_1}(xt^{2/3},0)$ in Assumption~\ref{assumption_decoupling_general_IC}(b).

Assumption~\ref{assumption_decoupling_general_IC}(c) is satisfied because $\mathcal{D}$ almost equals $\mathcal{A}$, which is independent of $\mathcal{S}$. We have $|h^{\rho_1}(xt^{2/3},0) - \hat{h}^{\rho_1}(xt^{2/3},0)| = 2|\mathcal{D}_{I_x}-\mathcal{A}_{I_x}|$, and for any $\e,\sigma \in (0,\tfrac{1}{3})$, it holds
\begin{equation}
\begin{aligned}
&\Pb\Bigl(\sup_{|x| \leq t^\e} |\mathcal{D}_{I_x}-\mathcal{A}_{I_x}| > t^\sigma \Bigr)  \leq Ct^\e \Pb(Q_0 > \tfrac{1}{2} t^\sigma) \leq C t^\e e^{-c t^{\sigma/2}},
\end{aligned}
\end{equation}
which converges to zero as $t \to \infty$. \\

We have shown that the initial condition $\eta$ satisfies Assumption~\ref{assumption_decoupling_general_IC}.
By \eqref{eq_bound_on_queue_length}, $t^{-1/3}Q_{x t^{2/3}}$ converges to zero in probability, uniformly for $x$ in a compact set. Consequently, the weak convergence of $\mathcal{A}$ passes down to $\mathcal{D}$: it holds
\begin{equation}
\frac{\mathcal{D}_{I_x}-\rho_1|x|t^{2/3}}{t^{1/3}} \Rightarrow \sigma_A \mathcal{B}_A(x)
\end{equation}
with respect to the uniform topology on compact sets.
Thus, the limiting distributions of $\mathfrak{h}^{\rho_1}(w,t)$ and $\mathfrak{h}^{\rho_1+\rho_2}(w,t)$ belong to the family
\begin{equation}
F^{(\sigma)}_w(s) = \Pb \Bigl( \max_{u \in \R} \{ \sqrt{2} \sigma \mathcal{B}(u)+\mathcal{A}_2(u)-(u-w)^2\} \leq s \Bigr),
\end{equation}
where $\sigma > 0$, $\mathcal{B}$ denotes a standard two-sided Brownian motion and $\mathcal{A}_2$ is an Airy$_2$ process independent of $\mathcal{B}$; see (2.9) of~\cite{CFS16}.

Simple examples of processes satisfying (A2)--(A4) include the Bernoulli processes from Section~\ref{section_queueing_representation}, mixtures of deterministic and Bernoulli data\footnote{For instance, we can assign independent $\text{Ber}(\rho_1)$-variables to the even sites, and place points at the odd sites either deterministically with another density $\tilde{\rho_1}$ or independently according to $\text{Ber}(\tilde{\rho}_1)$-distributions.}, or, more generally, subdivisions of $\Z$ into sets of a fixed length $n$, with the placement of points in the different sets being independent and identically distributed.

Also the initial condition from the Monte Carlo simulations in~\cite{CFS16} is admissible. It generalises to arbitrary densities as follows. For $\rho \in (0,1)$ and $\alpha \in (0,1)$, let $\{ a(i), i \in \Z\}$ be a stationary Markov chain with transition matrix
\begin{equation}
\begin{pmatrix}
1-2\rho(1-\alpha) & 2\rho(1-\alpha) \\ 2(1-\rho)(1-\alpha) & 1-2(1-\rho)(1-\alpha)
\end{pmatrix}.
\end{equation}
The stationary one-point distribution is given by $\Pb(a(i)=0) =1-\rho$ and \mbox{$\Pb(a(i)=1)=\rho$}, and $a(i)$ represents the occupation variable of site $i \in \Z$ in $\mathcal{A}$. The Markov chain is ergodic and reversible.

 Then, $\mathcal{A}$ satisfies (A2)--(A4) with $\rho_1=\rho$ and $\sigma_A = \sqrt{\rho(1-\rho)} \sqrt{ \tfrac{\alpha}{1-\alpha}}$. The tail bounds can, for example, be obtained from Theorem~1.1 of~\cite{Lez98}. The weak convergence is established in Lemma~2.5 of~\cite{CFS16} for the special case $\rho=\tfrac{1}{2}$ and follows similarly, for instance using Corollary~1.5 of~\cite{KV86}, for arbitrary $\rho$.

\appendix
\section{Endpoints of backwards paths in a TASEP with (half-)periodic initial condition} \label{appendix_A}

The proof of Proposition~\ref{prop_geodesic_end_point_deterministic_IC_half_periodic} requires similar arguments like those of Proposition 4.7 and Proposition 4.8 of~\cite{BF22} as well as of Theorem 4.3 of~\cite{FG26}.

\begin{proof}[Proof of Proposition~\ref{prop_geodesic_end_point_deterministic_IC_half_periodic}]
We suppose $x(t) = \alpha t$ with $ \alpha \in [1-2\rho,1-2\lambda] \cap (-1,1)$. By \eqref{eq_concatenation_property_height_fct} and \eqref{eq_geodesic_property},
\begin{equation}
\inf_{y \leq M t^{2/3}} \{ h(y,0) + h^{\text{step}}_{y,0}(\alpha t,t)\} < \inf_{y > M t^{2/3}} \{ h(y,0) + h^{\text{step}}_{y,0}(\alpha t,t)\}
\end{equation}
implies $x(0) \leq M t^{2/3}$. Thus, it holds
\begin{equation}
\begin{aligned}
\Pb(x(0) \leq M t^{2/3}) \geq 1 &- \Pb \Bigl ( \inf_{y \leq M t^{2/3}} \{ h(y,0) + h^{\text{step}}_{y,0}(\alpha t,t)\} \geq A \Bigr ) \\ &- \Pb \Bigl ( \inf_{y > M t^{2/3}} \{ h(y,0) + h^{\text{step}}_{y,0}(\alpha t,t)\} \leq A \Bigr ),
\end{aligned}
\end{equation}
where we set $A = \tfrac{1}{2}(1+\alpha^2)t + \phi M^2 t^{1/3}$ with $\phi > 0$ constant.

We can neglect the bounded term $H(y)$ in the initial condition and find
\begin{equation}
\begin{aligned}
 \Pb \Bigl ( \inf_{y \leq M t^{2/3}} \{ h(y,0) + h^{\text{step}}_{y,0}(\alpha t,t)\} \geq A \Bigr )  &\leq \Pb ( h^{\text{step}}_{0,0}(\alpha t,t) \geq \tfrac{1}{2}(1+\alpha^2)t + \phi M^2 t^{1/3}) \\
&\leq C e^{-cM^2}
\end{aligned}
\end{equation}
by the one-point estimates from Proposition~A.9 of~\cite{BF22}.

The second probability is bounded from above by
\begin{equation} \label{eq_proof_prop_endpoint_det_IC_1}
\Pb \Bigl ( \inf_{y > M t^{2/3}} \{ \alpha y + h^{\text{step}}_{y,0}(\alpha t,t)\} \leq \tfrac{1}{2}(1+\alpha^2)t + \phi M^2 t^{1/3} \Bigr ).
\end{equation}
To bound \eqref{eq_proof_prop_endpoint_det_IC_1}, we split the domain of $y$ up at $M t^{2/3+\delta}$ with $ \delta > 0$ such that $M t^{2/3 + \delta} = o(t)$. Colour-position symmetry yields\footnote{This follows by similar arguments as Lemma~5.1 of~\cite{Ger24}.} $(h^{\text{step}}_{y,0}(\alpha t,t))_{y \in \Z} \overset{(d)}{=} (h^{\text{step}}(\alpha t-y,t))_{y \in \Z}$.
We first consider
\begin{equation} \label{eq_proof_prop_endpoint_det_IC_2}
\Pb \Bigl ( \inf_{y > M t^{2/3+\delta}} \{ \alpha y + h^{\text{step}}(\alpha t-y,t)\} \leq \tfrac{1}{2}(1+\alpha^2)t + \phi M^2 t^{1/3} \Bigr ).
\end{equation}
For $\eta > \tfrac{1+\alpha}{2}$ and $y \geq \eta t$, it holds
\begin{equation}
\alpha y + h^{\text{step}}(\alpha t-y,t) \geq \alpha y + |\alpha t - y | \geq \tfrac{1}{2}(1+\alpha^2)t + \mu t
\end{equation}
for some $\mu > 0$. This means that the inequality in \eqref{eq_proof_prop_endpoint_det_IC_2} cannot hold for $y \geq \eta t$.
For $y < \eta t$ and $\eta < 1+\alpha$, we find $\alpha - y t^{-1} \in [\alpha-\eta, \alpha] \subseteq (-1,1) $. Thus, the one-point estimates from Proposition~A.9 of~\cite{BF22} apply with uniform constants for all such $y$ and we obtain
\begin{equation}
\begin{aligned}
\eqref{eq_proof_prop_endpoint_det_IC_2} \leq &\sum_{M t^{2/3+\delta} < y < \eta t} \Pb (\alpha y + h^{\text{step}}(\alpha t-y,t)\leq \tfrac{1}{2}(1+\alpha^2)t + \phi M^2 t^{1/3}  ) \\
\leq &\sum_{M t^{2/3+\delta} < y < \eta t} C e^{-c y^2 t^{-4/3}}
\leq  C e^{-cM^2 t^{2\delta}}.
\end{aligned}
\end{equation}

Next, we bound
\begin{equation} \label{eq_proof_prop_endpoint_det_IC_3}
\Pb \Bigl ( \inf_{M t^{2/3} < y \leq M t^{2/3+\delta}} \{ \alpha y + h^{\text{step}}(\alpha t-y,t)\} \leq \tfrac{1}{2}(1+\alpha^2)t + \phi M^2 t^{1/3} \Bigr )
\end{equation}
by a comparison to a stationary TASEP. For $I_\ell = (\ell M t^{2/3}, (\ell+1) M t^{2/3}]$, it holds
\begin{equation} \label{eq_proof_prop_endpoint_det_IC_4}
\begin{aligned}
\eqref{eq_proof_prop_endpoint_det_IC_3} \leq \sum_{\ell = 1 }^{t^\delta} \Pb \Bigl ( \inf_{y \in I_\ell} \{ \alpha y + h^{\text{step}}(\alpha t-y,t)\} \leq \tfrac{1}{2}(1+\alpha^2)t + \phi M^2 t^{1/3} \Bigr )
\end{aligned}
\end{equation}
and the summands in \eqref{eq_proof_prop_endpoint_det_IC_4} are bounded by
\begin{align}
&\Pb ( h^{\text{step}}(\alpha t - \ell M t^{2/3},t) + \alpha \ell M t^{2/3} \leq \tfrac{1}{2}(1+\alpha^2)t + \phi M^2 t^{1/3}  + \psi \ell^2 M^2 t^{1/3} )  \label{eq_proof_prop_endpoint_det_IC_5} \\
&+ \Pb \Bigl ( \inf_{y \in I_\ell} \{ h^{\text{step}}(\alpha t - y,t) - h^{\text{step}}(\alpha t - \ell M t^{2/3},t) + \alpha(y-\ell M t^{2/3})  \} \leq - \psi \ell^2 M^2 t^{1/3} \Bigr). \label{eq_proof_prop_endpoint_det_IC_6}
\end{align}
For $\phi + \psi < \tfrac{1}{2}$, the one-point estimates imply
$\eqref{eq_proof_prop_endpoint_det_IC_5} \leq C e^{-c \ell^2 M^2}$
with uniform constants $C,c > 0$.
To bound \eqref{eq_proof_prop_endpoint_det_IC_6}, we consider a stationary TASEP with density $\sigma = \tfrac{1-\alpha}{2}$ and height function $h^\sigma$, coupled with $h^{\text{step}}$ by basic coupling. Notice that $\alpha t - y < \alpha t - \ell M t^{2/3}$ for $y \in I_\ell$. By Remark 4.4 of~\cite{BF22}, in $h^{\text{step}}$, there is always a backwards geodesic starting at $\alpha t-y$ and ending at $0$. By Lemma~\ref{lemma_end_point_geodesics_stationary_TASEP}, a backwards path in $h^\sigma$ starting at $\alpha t - \ell M t^{2/3} = (1-2\sigma)t - \ell M t^{2/3}$ ends in $[-2\ell M t^{2/3}, 0]$ with probability at least $1-Ce^{-c\ell^2 M^2}$. Given this event takes place, the backwards geodesics intersect and Lemma 4.6 of~\cite{BF22} implies
\begin{equation}
h^{\text{step}} (\alpha t - \ell M t^{2/3},t) - h^{\text{step}}(\alpha t -y, t)\leq h^{\sigma}(\alpha t - \ell M t^{2/3},t) - h^{\sigma}(\alpha t - y,t).
\end{equation}
Further, we have
$ h^{\sigma}(\alpha t - y,t) - h^{\sigma}(\alpha t - \ell M t^{2/3},t) \overset{(d)}{=} \sum_{j=1}^{y-\ell M t^{2/3}} (2Z_j-1)$,
where $Z_j$ are independent $\text{Ber}(\sigma)$-distributed random variables. We deduce
\begin{equation}  \label{eq_proof_prop_endpoint_det_IC_6.5} \begin{aligned}
 \eqref{eq_proof_prop_endpoint_det_IC_6}  \leq   \Pb \Bigl ( \sup_{y \in I_\ell} \Bigl \{ \sum\nolimits_{j=1}^{y-\ell M t^{2/3}} 2(\sigma-Z_j)  \Bigr  \}\geq  \psi \ell^2 M^2 t^{1/3} \Bigr) + C e^{-c \ell^2 M^2}.
\end{aligned}
\end{equation}
Doob's submartingale inequality yields $\eqref{eq_proof_prop_endpoint_det_IC_6.5} \leq C e^{-c \ell^2 M^2}$.

We conclude $\eqref{eq_proof_prop_endpoint_det_IC_3} \leq C e^{-c M^2}$ and
\begin{equation}
\Pb(x(0) \leq M t^{2/3}) \geq 1 - C e^{-c M^2}.
\end{equation}
The bound $\Pb(x(0) \geq - M t^{2/3}) \geq  1 - C e^{-c M^2}$ is proven analogously.
\end{proof}

\section{Endpoints of backwards paths and tail bounds in a TASEP with random initial data} \label{appendix_B}

We establish the auxiliary results for the proof of Theorem~\ref{thm_asymptotic_decoupling_height_functions_general}.

The localisation of the endpoints of backwards paths in a single-species TASEP with homogeneous random initial data follows the same strategy as in the proof of Proposition~\ref{prop_geodesic_end_point_deterministic_IC_half_periodic}. The tail estimates from Assumption~\ref{assumption_decoupling_general_IC}(b) are sufficient for a rough localisation, considering a region of width $t^{2/3+\e}$ instead of $M t^{2/3}$. For this purpose, the comparison to a stationary TASEP is not required.

\begin{lem} \label{lemma_general_IC_backwards_path}
	In the setting of the proof of Theorem~\ref{thm_asymptotic_decoupling_height_functions_general}, it holds $\lim_{t \to \infty} \Pb(E_\rho) = 1$.
\end{lem}

\begin{proof}
	By the proof of Proposition~\ref{prop_localisation_backwards_geodesics_general}, it suffices to localise the endpoint of the backwards path by
	\begin{equation}
	\lim_{t \to \infty} \Pb(|x(0)| \leq \tfrac{1}{2} t^{2/3+\e}) = 1.
	\end{equation}
	We have
	\begin{equation} \label{eq_pf_lemma_general_IC_backwards_path_0}
	\begin{aligned}
	\Pb(x(0) \leq \tfrac{1}{2} t^{2/3+\e}) \geq 1 & - \Pb(h^{\text{step}}_{0,0}((1-2\rho)t,t) \geq A) \\
	&- \Pb\Bigl(\inf_{y > \frac{1}{2}t^{2/3+\e}} \{h^{\rho}(y,0) + h^{\text{step}}_{y,0}((1-2\rho)t,t) \} \leq A \Bigr),
	\end{aligned}
	\end{equation}
	where we set $A = (1-2\chi)t+\tfrac{1}{32} t^{1/3+2\e}$. By the one-point estimates from Proposition~A.9 of~\cite{BF22}, we have
	\begin{equation}
	\Pb(h^{\text{step}}_{0,0}((1-2\rho)t,t) \geq A) \leq C e^{-c t^{2\e}}.
	\end{equation}
	To bound the second probability in \eqref{eq_pf_lemma_general_IC_backwards_path_0}, we choose $\mu \in (0,1-\rho)$ and observe
	that for $\tfrac{1}{2}t^{2/3+\e} < y < (1-\rho+\mu)t$, we have $1-2\rho-yt^{-1} \in [-\rho-\mu,1-2\rho] \subseteq (-1,1)$, such that Proposition~A.9 of~\cite{BF22} and Assumption~\ref{assumption_decoupling_general_IC}(b) imply
	\begin{equation}
	\begin{aligned}
	&\Pb(h^{\rho}(y,0) + h^{\text{step}}_{y,0}((1-2\rho)t,t) \leq A) \\
	&\leq \Pb(h^{\rho}(y,0) \leq (1-2\rho)y - \tfrac{1}{4} y^2 t^{-1}) \\ & \hphantom{\leq}  + \Pb(h^{\text{step}}_{y,0}((1-2\rho)t,t) \leq \tfrac{1}{2}(1+(1-2\rho-yt^{-1})^2)t -\tfrac{1}{8}y^2 t^{-1}) \\
	&\leq C e^{-c y^{3/2} t^{-1}}+Ce^{-cy^2t^{-4/3}}
	\end{aligned}
	\end{equation}
	with uniform constants $C,c > 0$. Sub-additivity yields
	\begin{equation}
	\begin{aligned}
	\Pb\Bigl(\inf_{\frac{1}{2}t^{2/3+\e}<y<(1-\rho+\mu)t} \{h^{\rho}(y,0) + h^{\text{step}}_{y,0}((1-2\rho)t,t) \} \leq A \Bigr) \leq C e^{-c t^{3\e/2}}.
	\end{aligned}
	\end{equation}
	For $y \geq (1-\rho+\mu)t$, we use $h^{\text{step}}_{y,0}((1-2\rho)t,t) \geq |(1-2\rho)t-y|$ and derive
	\begin{equation}
	\Pb(h^{\rho}(y,0) + h^{\text{step}}_{y,0}((1-2\rho)t,t) \leq A) \leq C e^{-c (y-(1-\rho)t)y^{-1/2}}
	\end{equation}
	by Assumption~\ref{assumption_decoupling_general_IC}(b), which implies
	\begin{equation}
	 \Pb\Bigl(\inf_{y\geq(1-\rho+\mu)t} \{h^{\rho}(y,0) + h^{\text{step}}_{y,0}((1-2\rho)t,t) \} \leq A \Bigr) \leq C e^{-c\mu t^{1/2}}.
	\end{equation}
	We conclude $\lim_{t \to \infty} \Pb(x(0) \leq \tfrac{1}{2} t^{2/3+\e}) = 1$, and $\lim_{t \to \infty} \Pb(x(0) \geq - \tfrac{1}{2} t^{2/3+\e}) = 1$ follows analogously.
\end{proof}

The second auxiliary result relies on a distributional identity stemming from translation invariance, along with rough one-point estimates. These can be derived from the expression \eqref{eq_concatenation_property_height_fct} and Assumption~\ref{assumption_decoupling_general_IC}(b), provided that one considers deviations of strictly higher order than the typical fluctuations.

\begin{lem} \label{lemma_general_IC_translation}
	In the setting of the proof of Theorem~\ref{thm_asymptotic_decoupling_height_functions_general}, it holds $\lim_{t \to \infty} \Pb(G_\rho) = 1$.
\end{lem}

\begin{proof}
	The conservation law and the translation invariance of $\eta$ yield
	\begin{equation}
	h^{\rho}((1-2\rho)t^\nu+y,t^\nu)-h^\rho(y,0) \overset{(d)}{=} h^{\rho}((1-2\rho)t^\nu,t^\nu)
	\end{equation}
	for each fixed $y$. This implies
	\begin{equation}
	\begin{aligned}
	\Pb(G_\rho^c) \leq C t^{2/3+\e} \Pb(|h^{\rho}((1-2\rho)t^\nu,t^\nu)-(1-2\chi)t^{\nu}| > \tfrac{1}{2} t^\theta),
	\end{aligned}
	\end{equation}
	where $\theta \in (\tfrac{\nu}{3},\tfrac{1}{3})$. By \eqref{eq_concatenation_property_height_fct}, we have
	\begin{equation}
	\begin{aligned}
	& \Pb(|h^{\rho}((1-2\rho)t^\nu,t^\nu)-(1-2\chi)t^{\nu}| > \tfrac{1}{2} t^\theta) \\
	& \leq \Pb(h^{\text{step}}_{0,0}((1-2\rho)t^\nu,t^\nu) - (1-2\chi)t^\nu > \tfrac{1}{2} t^\theta) \\
	& \hphantom{\leq} + \Pb \Bigl( \min_{y \in \Z} \{h^\rho(y,0)+h^{\text{step}}_{y,0}((1-2\rho)t^\nu,t^\nu) \} - (1-2\chi)t^\nu < - \tfrac{1}{2} t^\theta \Bigr).
	\end{aligned}
	\end{equation}
	The first probability is bounded by $Ce^{-c t^{\theta-\nu/3}}$ by Proposition~A.9 of~\cite{BF22}.
	To bound the second probability, we choose $\mu \in (0,\min\{1-\rho,\rho\})$, such that for $-(\rho+\mu)t^\nu \leq y \leq (1-\rho+\mu)t^\nu$, it holds $1-2\rho-yt^{-\nu} \in [-\rho-\mu,1-\rho+\mu] \subseteq (-1,1)$, and derive by Assumption~\ref{assumption_decoupling_general_IC}(b) and Proposition~A.9 of~\cite{BF22} that
	\begin{equation}
	\begin{aligned}
	&\Pb(h^\rho(y,0)+h^{\text{step}}_{y,0}((1-2\rho)t^\nu,t^\nu) - (1-2\chi)t^\nu < - \tfrac{1}{2}t^\theta) \\
	&\leq \Pb(h^{\rho}(y,0)-(1-2\rho)y < - \tfrac{1}{4} y^2 t^{-\nu} - \tfrac{1}{4} t^\theta) \\
	&\hphantom{\leq}+\Pb(h^{\text{step}}_{y,0}((1-2\rho)t^\nu,t^\nu) < \tfrac{1}{2}(1+(1-2\rho-yt^{-\nu})^2)t^\nu-\tfrac{1}{4} y^2 t^{-\nu} - \tfrac{1}{4} t^\theta) \\
	&\leq Ce^{-c(|y|^{3/2} t^{-\nu}+t^\theta|y|^{-1/2})} \mathbbm{1}_{\{y \neq 0\}}+Ce^{-c(y^2t^{-4\nu/3} + t^{\theta-\nu/3})}.
	\end{aligned}
	\end{equation}
	Using sub-additivity and splitting the sum at $|y| = t^{\nu/3+\theta}$, we bound
	\begin{equation}
	\Pb \Bigl( \min_{-(\rho+\mu)t^\nu \leq y \leq (1-\rho+\mu)t^\nu} \{h^\rho(y,0)+h^{\text{step}}_{y,0}((1-2\rho)t^\nu,t^\nu) \} - (1-2\chi)t^\nu < - \tfrac{1}{2} t^\theta \Bigr)
	\end{equation}
	by $C e^{-ct^{(\theta-\nu/3)/2}}$ for $t$ large. The probabilities for the regions $y < -(\rho+\mu)t^\nu$ and $y > (1-\rho+\mu)t^\nu$ can be bounded by $Ce^{-c\mu t^{\nu/2}}$ respectively, using $h^{\text{step}}_{y,0}((1-2\rho)t^\nu,t^\nu) \geq |(1-2\rho)t^\nu-y|$ and Assumption~\ref{assumption_decoupling_general_IC}(b).
	We conclude $\lim_{t \to \infty} \Pb(G_\rho) = 1$.
\end{proof}

\section{Existence of second class particles in samples of $\mu^{\rho_1,\rho_2}$} \label{appendix_C}

In this section, we prove Lemma~\ref{lemma_second_class_particle_in_interval_with_high_prob}. For this purpose, we require the almost sure existence of a second class particle at some site:

\begin{lem} \label{lemma_second_class_particle_exists}
	In a sample of $\mu^{\rho_1,\rho_2}$, there exists a second class particle with probability $1$.
\end{lem}

Lemma~\ref{lemma_second_class_particle_exists} is proven later in this section.

\begin{proof}[Proof of Lemma~\ref{lemma_second_class_particle_in_interval_with_high_prob}]
By Lemma~\ref{lemma_second_class_particle_exists}, we find
	\begin{equation} \label{eq_proof_second_class_p_in_interval_0} \begin{aligned}
	 \Pb ( \forall i \in I: \ \eta(i) \neq 2)  & \leq  \Pb ( \forall i \in I: \ \eta(i) \neq 2, \exists \text{ maximal } x_l < u: \eta(x_l) = 2 )
	\\ & \hphantom{\leq} + \Pb ( \forall i \in I: \ \eta(i) \neq 2, \exists \text{ minimal } x_r > v: \eta(x_r) = 2 )
	\end{aligned} 	\end{equation}
	for $I=[u,v] \cap \Z$.
	We bound the second probability in \eqref{eq_proof_second_class_p_in_interval_0} by a similar reasoning as that in~\cite{FFK94}, where the distance between two second class particles was bounded\footnote{We refer to the proof of Lemma~2.1, see also Remark~3.2, of~\cite{FFK94}.}. Their arguments require the existence of a second class particle to the right of the interval $I=[u,v]$.
	
	For $x \in \Z$, we define a random walk $Z_x$ by $Z_x(n) = \sum_{i=n}^{x-1} (s(i)-a(i))$ for $n < x$ and $Z_x(x) = 0$. If in $\eta$, there is a second class particle at $x$, then the next second class particle to its left is at the right-most site $n< x$ such that $Z_x(n) = 1$. 	
 Then, the second probability in \eqref{eq_proof_second_class_p_in_interval_0} equals
	\begin{equation} \label{eq_proof_second_class_p_in_interval_1} \begin{aligned}
	& \Pb ( \exists x_r > v: \eta(x_r) = 2, \ \forall i \in [u,x_r-1]: \eta(i) \neq 2) \\
	&= \Pb \Bigl (  \exists  x_r > v: \eta(x_r) = 2, \ \sup_{n < x_r }\{Z_{x_r} (n)=1\} < u    \Bigr  ) .
	\end{aligned}
	\end{equation}
	From the proof of Lemma 2.1 of~\cite{FFK94}, we obtain $\sup_{n < x_r }\{Z_{x_r} (n)=1\} > -\infty$ almost surely. Similarly as in~\cite{FFK94}, Hoeffding's inequality yields
	
	\begin{equation} \label{eq_proof_second_class_p_in_interval_2}
	\begin{aligned}
	\eqref{eq_proof_second_class_p_in_interval_1} \leq & \sum\nolimits_{k=v+1}^\infty \Pb \Bigl ( \eta(k) = 2, \ \sup_{n < k }\{Z_{k} (n)=1\} < u    \Bigr  ) \\
	\leq & \sum\nolimits_{k=v+1}^\infty \sum\nolimits_{j=k-u +1}^\infty \Pb \Bigl (  \sup_{n < 0 }\{Z_{0} (n)=1\} = -j    \Bigr ) \\
	\leq & \sum\nolimits_{k=v+1}^\infty \sum\nolimits_{j=k-u +1}^\infty C e^{-c j} \leq C e^{-c \ell}
	\end{aligned}
	\end{equation}
	since $\ell=v-u$.
	
	We bound the first probability in \eqref{eq_proof_second_class_p_in_interval_0} by similar means, relying on the alternative queueing construction of $\mu^{\rho_1,\rho_2}$ explained in Remark~\ref{remark_alternative_queueing_construction}.
\end{proof}

\begin{proof}[Proof of Lemma~\ref{lemma_second_class_particle_exists}]
	Let $E$ be the set of edges in the construction of $\mu^{\rho_1,\rho_2}$ and define $F = \{ \forall i \in \Z: \eta(i) \neq 2\}$. With probability $1$, there exist $j \in \Z$ and $ i \geq j$ such that $s(j) = 1, a(i) = 1$ and $(i,j) \in E$. Since the edges in $E$ do not intersect, it holds
	\begin{equation} \label{eq_proof_second_class_particle_exists_1}
	\begin{aligned}
	\Pb(F) = \ & \Pb(\{\exists j \in \Z, i \geq j: s(j)=1, a(i) = 1, (i,j) \in E \} \cap F) \\
	\leq \ & \Pb(\exists j \in \Z, i \geq j: \ \mathcal{A}_{[n,i]} \geq \mathcal{S}_{[n,j]} \text{ for all } n < j) \\
	\leq \ & \sum_{j \in \Z, i \geq j} \Pb ( \mathcal{A}_{[n,i]} \geq \mathcal{S}_{[n,j]} \text{ for all } n < j).
	\end{aligned}
	\end{equation}
	For $j \in \Z, i \geq j$ fixed and $n^* = j - \tfrac{4}{\rho_2} (i-j) - t$ with $t > 0$, we have
	\begin{equation} \label{eq_proof_second_class_particle_exists_2}
	\begin{aligned}
	\Pb ( \mathcal{A}_{[n,i]} \geq \mathcal{S}_{[n,j]} \text{ for all } n < j)  &\leq   \Pb ( \mathcal{A}_{(n^*,j]} + i - j \geq \mathcal{S}_{(n^*,j]} ) \\
	& \leq \Pb (\mathcal{A}_{(n^*,j]} \geq m) + \Pb ( \mathcal{S}_{(n^*,j]} \leq m + i - j).
	\end{aligned}
	\end{equation}
	Choosing $m = (\rho_1 + \tfrac{1}{2} \rho_2)(j-n^*) +j-i$, we have
	$m + i - j \leq (\rho_1 + \rho_2)(j-n^*)$ and $m \geq \rho_1 (j-n^*)$. Hoeffding's inequality yields
	\begin{equation} \begin{aligned}
	& \Pb (\mathcal{A}_{(n^*,j]} \geq m) \leq e^{-2(j-n^*)  (\rho_1 - \tfrac{m}{j-n^*})^2} \leq e^{-\tfrac{1}{8} \rho_2^2 t }, \\
	& \Pb ( \mathcal{S}_{(n^*,j]} \leq m + i - j) \leq e^{-2(j-n^*) (\tfrac{1}{2} \rho_2)^2} \leq e^{-\tfrac{1}{2} \rho_2^2 t}.
	\end{aligned}
	\end{equation}
	Since the first inequality in \eqref{eq_proof_second_class_particle_exists_2} holds for any choice of $t$, we derive
	\begin{equation}
	\Pb ( \mathcal{A}_{[n,i]} \geq \mathcal{S}_{[n,j]} \text{ for all } n < j)  = 0.
	\end{equation}
	Together with \eqref{eq_proof_second_class_particle_exists_1}, this implies $\Pb(F) = 0$.
\end{proof}

\section{Computations for the hydrodynamics theory} \label{Appendix_GHD}
In this section, we compute the average current $\vec{\mathrm{j}}(\vec \rho)$ and the susceptibility matrix $C$ given in \eqref{eq.C} and the subsequent paragraph. We use the notation from Section~\ref{section_motivation}, but write $\E$ instead of $\E_{\mu_{\vec \rho}}$ and set $\eta_\alpha(j) = \eta_\alpha(j,0)$ and $\mathcal{J}_\alpha(j) = \mathcal{J}_\alpha(j,0)$.
For the two-species TASEP, the currents of first and second class particles are given by
\begin{equation}
\begin{aligned}
\mathcal{J}_1(j)& = \eta_1(j)(1-\eta_1(j+1)),\\
\mathcal{J}_2(j)& = \eta_2(j)(1-\eta_1(j+1)-\eta_2(j+1))-\eta_2(j+1)\eta_1(j).
\end{aligned}
\end{equation}
Due to stationarity, their distributions remain unchanged over time.
Next notice that we can rewrite
\begin{equation}
\mathcal{J}_2(j)= (\eta_1(j)+\eta_2(j))(1-\eta_1(j+1)-\eta_2(j+1))-\eta_1(j)(1-\eta_1(j+1)),
\end{equation}
and using the fact that $\eta_1$ and $\eta_1+\eta_2$ have Bernoulli product distributions, we immediately get
\begin{equation}
\begin{aligned}
\E[\mathcal{J}_1(j)]&=\rho_1(1-\rho_1),\\
\E[\mathcal{J}_2(j)]&=(\rho_1+\rho_2)(1-\rho_1-\rho_2)-\rho_1(1-\rho_1)=\rho_2(1-\rho_2)-2\rho_1\rho_2.
\end{aligned}
\end{equation}

Let us now compute the entries of the susceptibility matrix $C$. We have
\begin{equation}
C_{1,1}=\sum_{j\in\Z}\Cov(\eta_1(j),\eta_1(0))=\sum_{j\in\Z}\left(\Pb(\eta_1(j)=\eta_1(0)=1)-\rho_1^2\right)=\rho_1(1-\rho_1)
\end{equation}
since the marginal distribution of the first class particles is a Bernoulli product measure with density $\rho_1$, so the only non-zero contribution in the sum comes from $j=0$.
Similarly, it holds
\begin{equation}
C_{1,1}+C_{1,2}+C_{2,1}+C_{2,2}=\sum_{j\in\Z}\Cov(\eta_1(j)+\eta_2(j),\eta_1(0)+\eta_2(0))=(\rho_1+\rho_2)(1-\rho_1-\rho_2).
\end{equation}
Using $C_{1,2}=C_{2,1}$ we obtain
\begin{equation}
C=\left(\begin{array}{cc}
  \rho_1(1-\rho_1) & C_{1,2} \\
  C_{1,2} & (\rho_1+\rho_2)(1-\rho_1-\rho_2)-\rho_1(1-\rho_1)-2C_{1,2}
  \end{array}\right).
\end{equation}
Plugging this into the identity \eqref{eq1}, $A C = C A^T$, we get that $C_{1,2}=-\rho_1(1-\rho_1)$.

Alternatively, to compute the expected values of the speed, one can apply Theorem~1.7 of~\cite{AAV11}, which explicitly characterises the joint distribution of two consecutive speeds in the TASEP speed process. This result is proven by mapping the speed process to a realisation of $\mu_{\vec \rho}$, see Corollary~5.4 of~\cite{AAV11}, and employing its queueing representation. Similarly, to compute the matrix $C$, for instance to determine $C_{2,2}$, one separates the contributions of $j=0$, $j>0$ and $j<0$, uses translation invariance, and applies Lemma~6.2 of~\cite{AAV11} with $x=\rho_1$ and $y=\rho_1+\rho_2$.

\section{On the queue length for general asymmetry} \label{appendix_E}

In Section~\ref{section_properties_queue_length_ASEP}, we claim that the stationary queue length $Q_i$ constructed in Section~\ref{section_queueing_representation_ASEP} fulfils: there exist a finite set $I$ and constants $b < \infty, \beta > 0, c>0$ such that $V(x)=e^{cx}$ satisfies
\begin{equation}
\E[V(Q_i)|Q_{i+1}=x]-V(x) \leq - \beta V(x) + b \Id_I(x).
\end{equation}
This can be seen as follows.
We have
\begin{equation}
\begin{aligned}
& \E[V(Q_i)|Q_{i+1}=x]-V(x) \\ = & \rho_1(1-\rho_1-\rho_2)V(x+1) + (1-\rho_1)(\rho_1+\rho_2)(1-q^x)V(x-1) \\ & - [\rho_1(1-\rho_1-\rho_2) + (1-\rho_1)(\rho_1+\rho_2)(1-q^x)]V(x) \\
= & V(x)[\rho_1(1-\rho_1-\rho_2)(e^c-1) + (1-\rho_1)(\rho_1+\rho_2)(1-q^x)(e^{-c}-1)].
\end{aligned}
\end{equation}
We denote the right hand side by $V(x)f(c,x)$. Then, $f(c,x) < 0$ is equivalent to
\begin{equation}
e^c < \frac{(1-\rho_1)(\rho_1+\rho_2)(1-q^x)}{\rho_1(1-\rho_1-\rho_2)},
\end{equation}
and as $x \to \infty$, the right hand side converges to
\begin{equation}
\frac{(1-\rho_1)(\rho_1+\rho_2)}{\rho_1(1-\rho_1-\rho_2)} > 1.
\end{equation}
Thus, there exist some $x_0 \in \N$ and $c,\beta > 0$ such that for all
$x>x_0$, it holds
\begin{equation}
\E[V(Q_i)|Q_{i+1}=x]-V(x) \leq - \beta V(x).
\end{equation}
For $x \in I \coloneqq \{0, \dots, x_0 \}$, we have $V(x)f(c,x) \leq b$ for some $b< \infty$. This yields the claim.

\section*{Acknowledgements} The authors are grateful to Pedro Cardoso for discussions on the ABC model and to Ofer Busani for exchanges on the queueing representations. They thank the two anonymous referees for their careful reading and constructive suggestions.

\section*{Declarations}
This work was partially supported by the Deutsche Forschungsgemeinschaft (DFG, German Research Foundation) through the Hausdorff Center for Mathematics (EXC 2047/1, project-ID 390685813) and the Collaborative Research Centre \qq{Analysis of Criticality: From Complex Phenomena to Models and Estimates} (CRC 1720, project-ID 539309657). The authors have no competing interests to declare that are relevant to the content of this article.



\begin{thebibliography}{10}

\bibitem{ACG23}
A.~Aggarwal, I.~Corwin, and P.~Ghosal.
\newblock {The ASEP speed process}.
\newblock {\em Adv. Math.}, 422:109004, 2023.

\bibitem{ACH24}
A.~Aggarwal, I.~Corwin, and M.~Hegde.
\newblock {Scaling limit of the colored ASEP and stochastic six-vertex models}.
\newblock {\em arXiv:2403.01341v2}, 2024.

\bibitem{Agg18}
A.~Aggarwal. Current fluctuations of the stationary ASEP and six-vertex model. \emph{Duke Math. J.}, {167}:269--384, 2018.

\bibitem{ANP25}
A.~Aggarwal, M.~Nicoletti, and L.~Petrov.
\newblock {Colored interacting particle systems on the ring: stationary
  measures from Yang-Baxter equation}.
\newblock {\em Compos. Math.}, 161:1855--1922, 2025.

\bibitem{AAV11}
G.~Amir, O.~Angel, and B.~Valk\'{o}.
\newblock {The TASEP speed process}.
\newblock {\em Ann. Probab.}, 39:1205--1242, 2011.

\bibitem{ABGM21}
G.~Amir, O.~Busani, P.~Gon\c{c}alves, and J.B.~Martin.
\newblock {The TAZRP speed process}.
\newblock {\em Ann. Inst. Henri Poincaré Probab. Stat.}, 57:1281--1305, 2021.

\bibitem{Ang06}
O.~Angel.
\newblock {The stationary measure of a 2-type totally asymmetric exclusion
  process}.
\newblock {\em J. Comb. Theory Ser. A}, 113:625--635, 2006.

\bibitem{Asm03}
S.~Asmussen.
\newblock {\em Applied probability and queues}.
\newblock Springer Verlag, New York, 2003.

\bibitem{BFP09}
J.~Baik, P.L. Ferrari, and S.~P{\'e}ch{\'e}.
\newblock {Limit process of stationary TASEP near the characteristic line}.
\newblock {\em Comm. Pure Appl. Math.}, 63:1017--1070, 2010.

\bibitem{BFP12}
J.~Baik, P.L. Ferrari, and S.~P{\'e}ch{\'e}.
\newblock {Convergence of the two-point function of the stationary TASEP}.
\newblock In {\em {Singular Phenomena and Scaling in Mathematical Models}},
  pages 91--110. Springer, 2014.

\bibitem{BR00}
J.~Baik and E.M. Rains.
\newblock Limiting distributions for a polynuclear growth model with external
  sources.
\newblock {\em J. Stat. Phys.}, 100:523--542, 2000.

\bibitem{BS09}
M.~Balázs and T.~Seppäläinen.
\newblock Fluctuation bounds for the asymmetric simple exclusion process.
\newblock {\em ALEA Lat. Am. J. Probab. Math. Stat.}, 6:1--24, 2009.

\bibitem{BFS21}
C.~Bernardin, T.~Funaki, and S.~Sethuraman.
\newblock {Derivation of coupled KPZ-Burgers equation from multi-species
  zero-range processes}.
\newblock {\em Ann. Appl. Probab.}, 31:1966--2017, 2021.

\bibitem{BF22}
A.~Bufetov and P.L. Ferrari.
\newblock {Shock fluctuations in TASEP under a variety of time scalings}.
\newblock {\em Ann. Appl. Probab.}, 32:3614--3644, 2022.

\bibitem{Bus24}
O.~Busani.
\newblock {Diffusive scaling limit of the Busemann process in last passage
  percolation}.
\newblock {\em Ann. Probab.}, 52:1650--1712, 2024.

\bibitem{BSS23}
O.~Busani, T.~Sepp{\"a}l{\"a}inen, and E.~Sorensen.
\newblock {Scaling limit of the TASEP speed process}.
\newblock {\em arXiv:2211.04651v4}, 2022.

\bibitem{BSS24}
O.~Busani, T.~Seppäläinen, and E.~Sorensen.
\newblock {Scaling limit of multi-type invariant measures via the directed
  landscape}.
\newblock {\em Int. Math. Res. Not. IMRN}, 2024:12382--12432, 2024.

\bibitem{CCGO26}
G.~Cannizzaro, P.~Cardoso, L.~Gr\"afner, and A.~Occelli.
\newblock {Equilibrium fluctuations for a multi-species particle system with long jumps}.
\newblock {\em arXiv:2604.03777}, 2026.


\bibitem{CGMO24}
G.~Cannizzaro, P.~Gon\c{c}alves, R.~Misturini, and A.~Occelli.
\newblock {From ABC to KPZ}.
\newblock {\em Probab. Theory Relat. Fields}, 191:361--420, 2025.


\bibitem{CGHS18}
Z.~Chen, J.~de~Gier, I.~Hiki, and T.~Sasamoto.
\newblock {Exact confirmation of 1D nonlinear fluctuating hydrodynamics for a
  two-species exclusion process}.
\newblock {\em Phys. Rev. Lett.}, 120:240601, 2018.

\bibitem{CGHSU22}
Z.~Chen, J.~de~Gier, I.~Hiki, T.~Sasamoto, and M.~Usui.
\newblock {Limiting current distribution for a two species asymmetric exclusion
  process}.
\newblock {\em Comm. Math. Phys.}, 395:59--142, 2022.

\bibitem{CFS16}
S.~Chhita, P.L. Ferrari, and H.~Spohn.
\newblock {Limit distributions for KPZ growth models with spatially homogeneous
  random initial conditions}.
\newblock {\em Ann. Appl. Probab.}, 28:1573--1603, 2018.

\bibitem{DDSMS14}
S.G. Das, A.~Dhar, K.~Saito, C.B. Mendl, and H.~Spohn.
\newblock {Numerical test of hydrodynamic fluctuation theory in the
  Fermi-Pasta-Ulam chain}.
\newblock {\em Phys. Rev. E}, 90:012124, 2014.

\bibitem{DOV22}
D.~Dauvergne, J.~Ortmann, and B.~Vir\'ag.
\newblock The directed landscape.
\newblock {\em Acta Math.}, 229:201--285, 2022.

\bibitem{DV21}
D.~Dauvergne and B.~Virág.
\newblock {The scaling limit of the longest increasing subsequence}.
\newblock {\em arXiv:2104.08210}, 2021.

\bibitem{DJLS93}
B.~Derrida, S.A. Janowsky, J.L. Lebowitz, and E.R. Speer.
\newblock Exact solution of the totally asymmetric simple exclusion process:
  shock profiles.
\newblock {\em J. Stat. Phys.}, 73:813--842, 1993.

\bibitem{EMP09}
M.~Evans, K.~Mallick, and S.~Prolhac.
\newblock {The matrix product solution of the multi-species partially
  asymmetric exclusion process}.
\newblock {\em J. Phys. A: Math. Theor.}, 42:165004, 2009.

\bibitem{EFM09}
M.R. Evans, P.A. Ferrari, and K.~Mallick.
\newblock {Matrix representation of the stationary measure for the multispecies
  TASEP}.
\newblock {\em J. Stat. Phys.}, 135:217--239, 2009.

\bibitem{FFK94}
P.A. Ferrari, L.R.G. Fontes, and Y.~Kohayakawa.
\newblock {Invariant measures for a two-species asymmetric process}.
\newblock {\em J. Stat. Phys.}, 76:1153--1177, 1994.

\bibitem{FKS91}
P.A. Ferrari, C.~Kipnis, and E.~Saada.
\newblock Microscopic structure of travelling waves in the asymmetric simple
  exclusion process.
\newblock {\em Ann. Probab.}, 19:226--244, 1991.

\bibitem{FM06}
P.A. Ferrari and J.B.~Martin.
\newblock {Multi-class processes, dual points and M/M/1 queues}.
\newblock {\em Markov Process. Related Fields}, 12:175--201, 2006.

\bibitem{FM07}
P.A. Ferrari and J.B. Martin.
\newblock {Stationary distribution of multi-type totally asymmetric exclusion
  processes}.
\newblock {\em Ann. Probab.}, 35:807--832, 2007.

\bibitem{FG26}
P.L. Ferrari and S.~Gernholt.
\newblock {Tagged particle fluctuations for TASEP with dynamics restricted by a
  moving wall}.
\newblock {\em Ann. Appl. Probab.}, 36:451--493, 2026.

\bibitem{FN19}
P.L. Ferrari and P.~Nejjar.
\newblock {Statistics of TASEP with three merging characteristics}.
\newblock {\em J. Stat. Phys.}, 180:398--413, 2020.

\bibitem{FN24}
P.L. Ferrari and P.~Nejjar.
\newblock The second class particle process at shocks.
\newblock {\em Stoch. Proc. Appl.}, 170:104298, 2024.

\bibitem{FO17}
P.L. Ferrari and A.~Occelli.
\newblock {Universality of the GOE Tracy-Widom distribution for TASEP with
  arbitrary particle density}.
\newblock {\em Electron. J. Probab.}, 23:1--24, 2018.

\bibitem{FSS13}
P.L. Ferrari, T.~Sasamoto, and H.~Spohn.
\newblock {Coupled Kardar-Parisi-Zhang equations in one dimension}.
\newblock {\em J. Stat. Phys.}, 153:377--399, 2013.

\bibitem{FS05a}
P.L. Ferrari and H.~Spohn.
\newblock Scaling limit for the space-time covariance of the stationary totally
  asymmetric simple exclusion process.
\newblock {\em Comm. Math. Phys.}, 265:1--44, 2006.

\bibitem{Ger24}
S.~Gernholt.
\newblock {TASEP with a general initial condition and a deterministically
  moving wall}.
\newblock {\em Electron. J. Probab.}, 30:1--38, 2025.

\bibitem{GS11}
R.M. Grisi and G.M. Schütz.
\newblock{Current symmetries for particle systems with several conservation laws}.
\newblock {\em J. Stat. Phys.}, 145:1499--1512, 2011.

\bibitem{Har72}
T.E. Harris.
\newblock {Nearest-neighbor Markov interaction processes on multidimensional
  lattices}.
\newblock {\em Adv. Math.}, 9:66--89, 1972.

\bibitem{Har78}
T.E. Harris.
\newblock {Additive set-valued Markov processes and graphical methods}.
\newblock {\em Ann. Probab.}, 6:355--378, 1978.

\bibitem{KPZ86}
M.~Kardar, G.~Parisi, and Y.Z. Zhang.
\newblock Dynamic scaling of growing interfaces.
\newblock {\em Phys. Rev. Lett.}, 56:889--892, 1986.

\bibitem{KV86}
C.~Kipnis and S.R.S. Varadhan.
\newblock {Central limit theorem for additive functionals of reversible Markov
  processes and applications to simple exclusions}.
\newblock {\em Comm. Math. Phys.}, 104:1--19, 1986.

\bibitem{KHS15}
M.~Kulkarni, D.A. Huse, and H.~Spohn.
\newblock {Fluctuating hydrodynamics for a discrete Gross-Pitaevskii equation:
  Mapping onto the Kardar-Parisi-Zhang universality class}.
\newblock {\em Phys. Rev. A}, 92:043612, 2015.

	\bibitem{LS25}
B.~Landon and P.~Sosoe.
\newblock{Tail estimates for the stationary stochastic six-vertex model and ASEP}.
\newblock {\em Probab. Math. Phys.}, 6:1327--1378, 2025.

\bibitem{Lez98}
P.~Lezaud.
\newblock {Chernoff-type bound for finite Markov chains}.
\newblock {\em Ann. Appl. Probab.}, 8:849--867, 1998.

\bibitem{Lig76}
T.M. Liggett.
\newblock Coupling the simple exclusion process.
\newblock {\em Ann. Probab.}, 4:339--356, 1976.

\bibitem{Mar20}
J.B.~Martin.
\newblock {Stationary distributions of the multi-type ASEP}.
\newblock {\em Electron. J. Probab.}, 25:1--41, 2020.

\bibitem{MQR17}
K.~Matetski, J.~Quastel, and D.~Remenik.
\newblock {The KPZ fixed point}.
\newblock {\em Acta Math.}, 227:115--203, 2021.

\bibitem{MS13}
C.B. Mendl and H.~Spohn.
\newblock {Dynamic correlators of Fermi-Pasta-Ulam chains and nonlinear
  fluctuating hydrodynamics}.
\newblock {\em Phys. Rev. Lett.}, 111:230601, 2013.

\bibitem{MS14}
C.B. Mendl and H.~Spohn.
\newblock Equilibrium time-correlation functions for one-dimensional hard-point
  systems.
\newblock {\em Phys. Rev. E}, 90:012147, 2014.

	\bibitem{MT93}
S.P.~Meyn and R.L.~Tweedie. \emph{Markov chains and stochastic stability}. Springer Verlag, London, 1993.

\bibitem{PSSS15}
V.~Popkov, A.~Schadschneider, J.~Schmidt, and G.M. Schütz.
\newblock {Fibonacci family of dynamical universality classes}.
\newblock {\em Proc. Nat. Acad. Sci. U.S.A.}, 112:12645--12650, 2015.

\bibitem{PSS15}
V.~Popkov, J.~Schmidt, and G.M. Schütz.
\newblock {Universality classes in two-component driven diffusive systems}.
\newblock {\em J. Stat. Phys.}, 160:835--860, 2015.

\bibitem{PS01}
M.~Pr{\"a}hofer and H.~Spohn.
\newblock Current fluctuations for the totally asymmetric simple exclusion
  process.
\newblock In V.~Sidoravicius, editor, {\em In and out of equilibrium}, Progress
  in Probability. Birkh{\"a}user, 2002.

\bibitem{PS02b}
M.~Pr{\"a}hofer and H.~Spohn.
\newblock Exact scaling function for one-dimensional stationary {KPZ} growth.
\newblock {\em J. Stat. Phys.}, 115:255--279, 2004.

\bibitem{RSS00}
N.~Rajewsky, T.~Sasamoto, and E.R. Speer.
\newblock Spatial particle condensation for an exclusion process on a ring.
\newblock {\em Physica A}, 279:123--142, 2000.

\bibitem{RDKKS24}
D.~Roy, A.~Dhar, K.~Khanin, M.~Kulkarni, and H.~Spohn.
\newblock {Universality in coupled stochastic Burgers systems with degenerate
  flux Jacobian}.
\newblock {\em J. Stat. Mech.}, 2024:033209, 2024.

\bibitem{SW17}
G.M. Schütz and B.~Wehefritz-Kaufmann.
\newblock {Kardar-Parisi-Zhang modes in $d$-dimensional directed polymers}.
\newblock {\em Phys. Rev. E}, 96:032119, 2017.

\bibitem{SS21}
T.~Seppäläinen and E.~Sorensen.
\newblock {Global structure of semi-infinite geodesics and competition
  interfaces in Brownian last-passage percolation}.
\newblock {\em Probab. Math. Phys.}, 4:667--760, 2023.

\bibitem{Spe94}
E.R. Speer.
\newblock {The two species totally asymmetric simple exclusion process}.
\newblock In {\em {On Three Levels: Micro, Meso and Macroscopic Approaches in
  Physics}}, pages 91--102. C.M.M. Fannes and A. Verbuere, 1994.

\bibitem{Spi70}
F.~Spitzer.
\newblock {Interaction of Markov processes}.
\newblock {\em Adv. Math.}, 5:246--290, 1970.

\bibitem{Spo13}
H.~Spohn.
\newblock {Nonlinear fluctuating hydrodynamics for anharmonic chains}.
\newblock {\em J. Stat. Phys.}, 154:1191--1227, 2014.

\bibitem{Spo16}
H.~Spohn.
\newblock {Fluctuating hydrodynamics approach to equilibrium time correlations
  for anharmonic chains}.
\newblock {\em Thermal transport in low dimensions: from statistical physics to
  nanoscale heat transfer, in Lect. Notes Phys.}, pages 107--158, 2016.

\bibitem{SS15}
H.~Spohn and G.~Stoltz.
\newblock {Nonlinear fluctuating hydrodynamics in one dimension: the case of
  two conserved fields}.
\newblock {\em J. Stat. Phys.}, 160:861--884, 2015.

\bibitem{TV03}
B.~Tóth and B.~Valkó.
\newblock{Onsager relations and Eulerian hydrodynamic limit for systems with several conservation laws}.
\newblock {\em J. Stat. Phys.}, 112:497--521, 2003.

\bibitem{TW96}
C.A. Tracy and H.~Widom.
\newblock On orthogonal and symplectic matrix ensembles.
\newblock {\em Comm. Math. Phys.}, 177:727--754, 1996.

\bibitem{vanB12}
H.~van Beijeren.
\newblock Exact results for anomalous transport in one-dimensional hamiltonian
  systems.
\newblock {\em Phys. Rev. Lett.}, 108:180601, 2012.

\end{thebibliography}

\end{document}